\documentclass[a4paper,11pt]{amsart}

\usepackage{latexsym}
\usepackage{amssymb,amsmath, amsthm}
\usepackage{mathtools, tensor}
\usepackage{tikz, float, caption}
\usepackage{mathtools}
\usetikzlibrary{patterns}
\usepackage{mathdots}
\usepackage{graphics}
\usepackage{url}
\usepackage[enableskew,vcentermath]{youngtab}

\usepackage{ctable, multirow}

\newbox\mybox
\def\centerfigure#1{%
    \setbox\mybox\hbox{#1}%
    \raisebox{-0.5\dimexpr\ht\mybox+\dp\mybox}{\copy\mybox}%
}
\allowdisplaybreaks

\newcommand{\blue}[1]{\color[rgb]{0.25,0.25,1}{#1}}

\newcommand{\red}[1]{\color[rgb]{1,0.0,0.0}{#1}}

\definecolor{darkgrey}{rgb}{0.6,0.6,0.6}

\newtheorem{theorem}{Theorem}[section]
\newtheorem*{theorem*}{Theorem}
\newtheorem{corollary}[theorem]{Corollary}

\newtheorem{proposition}[theorem]{Proposition}

\newtheorem{lemma}[theorem]{Lemma}

\theoremstyle{definition}
\newtheorem{definition}[theorem]{Definition}
\newtheorem{remark}[theorem]{Remark}
\newtheorem{example}[theorem]{Example}

\newcommand{\N}{\mathbf{N}}
\newcommand{\Z}{\mathbf{Z}}
\newcommand{\Q}{\mathbf{Q}}

\renewcommand{\epsilon}{\varepsilon}

\DeclareMathOperator{\GL}{GL}
\DeclareMathOperator{\SL}{SL}

\DeclareMathOperator{\Sym}{Sym}

\DeclareMathOperator{\Tr}{Tr}
\DeclareMathOperator{\sgn}{sgn}

\DeclareMathOperator{\SSYT}{SSYT}
\DeclareMathOperator{\RSSYT}{RSSYT}

\DeclareMathOperator{\rank}{R}
\DeclareMathOperator{\erank}{E}
\DeclareMathOperator{\srank}{S}
\DeclareMathOperator{\epar}{{\mathcal{E}\hskip-0.5pt\mathcal{P}}}
\DeclareMathOperator{\spar}{{\mathcal{SP}}}
\DeclareMathOperator{\shape}{sh}

\renewcommand{\theta}{\vartheta}

\newcommand{\sh}{d}
\newcommand{\shp}{d'}

\newcommand{\Par}{\mathrm{Par}}

\newcommand{\MGL}[1]{\mathrm{MGL}^{(#1)}}
\newcommand{\PP}{\mathcal{PP}}
\newcommand{\qbinom}[2]{\genfrac{[}{]}{0pt}{}{#1}{#2}}
\newcommand{\sqbinom}[2]{\genfrac{|}{|}{0pt}{}{#1}{#2}}
\newcommand{\jj}{k}
\newcommand{\E}{\! E}
\renewcommand{\b}{b}
\newcommand{\Se}[3]{\mathcal{S}_{#1}^{#2}(#3)}
\newcommand{\Seb}[3]{\mathcal{S}_{#1}^{#2}\bigl(#3\bigr)}

\newcommand{\z}{z}
\newcommand{\SSym}{\hskip-0.5pt\Sym}
\newcommand{\SSSym}{\hskip-1.5pt\Sym}

\newcommand{\bigdot}[1]{\draw[fill=black] #1 circle (3pt);}
\newcommand{\twidth}{0.15}
\newcommand{\gscale}{0.65}
\newcommand{\ggscale}{0.575}
\newcommand{\axessep}{0.5}
\newcommand{\vaxessep}{0.1}

\newcommand{\htick}[1]{\draw (#1,-\twidth)--(#1,\twidth);}
\newcommand{\vtick}[1]{\draw (-\twidth,#1)--(\twidth,#1);}
\renewcommand{\k}{m}

\newcommand{\mfrac}[2]{{\textstyle\frac{#1}{#2}}}

\newcounter{thmlistcnt}
\newenvironment{thmlist}%
	{\setcounter{thmlistcnt}{0}%
	\begin{list}{\emph{(\roman{thmlistcnt})}}{%
		\usecounter{thmlistcnt}%
		\setlength{\topsep}{0pt}%
		\setlength{\leftmargin}{0pt}%
		\setlength{\itemsep}{0pt}%
		\setlength{\labelwidth}{24pt}
		\setlength{\itemindent}{36pt}}%
	}%
	{\end{list}}%

\newcounter{thmlistcnta}
\newenvironment{thmlista}%
	{\setcounter{thmlistcnta}{0}%
	\begin{list}{\emph{(\alph{thmlistcnta})}}{%
		\usecounter{thmlistcnta}%
		\setlength{\topsep}{0pt}%
		\setlength{\leftmargin}{0pt}%
		\setlength{\itemsep}{0pt}%
		\setlength{\labelwidth}{24pt}
		\setlength{\itemindent}{36pt}}%
	}%
	{\end{list}}%
	
\newcounter{rectanglethm}

\newlength{\hatchspread}
\newlength{\hatchthickness}
\newlength{\hatchshift}
\newcommand{\hatchcolor}{}
\tikzset{hatchspread/.code={\setlength{\hatchspread}{#1}},
         hatchthickness/.code={\setlength{\hatchthickness}{#1}},
         hatchshift/.code={\setlength{\hatchshift}{#1}},
         hatchcolor/.code={\renewcommand{\hatchcolor}{#1}}}
\tikzset{hatchspread=3pt,
         hatchthickness=0.4pt,
         hatchshift=0pt,
         hatchcolor=black}
\pgfdeclarepatternformonly[\hatchspread,\hatchthickness,\hatchshift,\hatchcolor]
   {custom north east lines}
   {\pgfqpoint{\dimexpr-2\hatchthickness}{\dimexpr-2\hatchthickness}}
   {\pgfqpoint{\dimexpr\hatchspread+2\hatchthickness}{\dimexpr\hatchspread+2\hatchthickness}}
   {\pgfqpoint{\dimexpr\hatchspread}{\dimexpr\hatchspread}}
   {
    \pgfsetlinewidth{\hatchthickness}
    \pgfpathmoveto{\pgfqpoint{\dimexpr\hatchshift-0.15pt}{-0.15pt}}
    \pgfpathlineto{\pgfqpoint{\dimexpr\hatchspread+0.15pt}{\dimexpr\hatchspread-\hatchshift+0.15pt}}
    \ifdim \hatchshift > 0pt
      \pgfpathmoveto{\pgfqpoint{-0.15pt}{\dimexpr\hatchspread-\hatchshift-0.15pt}}
      \pgfpathlineto{\pgfqpoint{\dimexpr\hatchshift+0.15pt}{\dimexpr\hatchspread+0.15pt}}
    \fi
    \pgfsetstrokecolor{\hatchcolor}
    \pgfusepath{stroke}
   }	
	
\definecolor{grey}{rgb}{0.8,0.8,0.8}	
\newcommand{\youngbox}[2]{\draw  (#2,-#1) rectangle (#2-1,-#1-1);}
\newcommand{\fillbox}[3]{\draw[fill=#3] (#2,-#1) rectangle (#2-1,-#1-1);}
\newcommand{\hatchbox}[3]{\draw[pattern=custom north east lines, hatchspread=#3] (#2,-#1) rectangle (#2-1,-#1-1);}
\newcommand{\younglabel}[3]{\node at (#1-0.5,-#2-0.5) {#3};}

\newcommand{\lambdastarp}{\lambda^{\star'}}
\newcommand{\clambda}[1]{c(\lambda / \lambda^\star)_{#1}}
\newcommand{\C}{\mathbf{C}}

\newcommand{\eqv}[2]{\!\tensor*[^{}_{#1\hskip1pt}]{\sim}{^{}_{#2}}\hskip-0.5pt}
\newcommand{\eqvs}[2]{\hbox{$\!\tensor*[^{}_{#1\hskip1pt}]{\sim}{^{}_{#2}}\hskip0.5pt$}}

\newcommand{\nlambda}{{\lambda}}
\newcommand{\nmu}{\mu}
\newcommand{\m}{m}

\newcommand{\elld}{{\ell^\dagger}}

\newcommand{\Ft}[2]{F\bigl( \,\young(#1,#2) \,\bigr)}
\numberwithin{equation}{section}

\newcommand{\bh}{\ell}
\newcommand{\ah}{n}

\renewcommand{\b}{b}
\renewcommand{\c}[2]{c_{#1}(#2)}

\newcommand{\ellp}{\ell^\dagger}
\newcommand{\ellpp}{\ell^\ddagger}
\renewcommand{\mp}{m^\dagger}
\newcommand{\mpp}{m^\ddagger}
\newcommand{\np}{n^\dagger}

\newcommand{\ap}{a'}
\newcommand{\bp}{b'}
\newcommand{\cp}{c'}
\newcommand{\p}{p}

\renewcommand{\P}[2]{P^{(#1)}_{#2}}
\newcommand{\QQ}[2]{Q^{(#1)}_{#2}}

\newcommand{\lambdae}[2]{\lambda_{#2#1}}
\newcommand{\mue}[2]{\mu_{#2#1}}
\newcommand{\etae}[2]{\eta_{#2#1}}
\newcommand{\bP}{\mathbf{p}}
\newcommand{\bQ}{\mathbf{q}}

\newcommand{\Hc}{u}


\begin{document}
\title[Plethysms of symmetric functions]{Plethysms of symmetric functions 
and representations of $\SL_2(\C)$}
\author{Rowena Paget and Mark Wildon}
\date{\today}

\begin{abstract}
Let $\nabla^\lambda$ denote the Schur functor
 labelled by the partition~$\lambda$ and let $E$ be the natural representation
 of $\mathrm{SL}_2(\mathbb{C})$.
 We make a systematic study of when there is an isomorphism
 $\nabla^\lambda \!\Sym^\ell \!E \cong \nabla^\mu \!\Sym^m \! E$ of representations of $\mathrm{SL}_2(\mathbb{C})$.
 Generalizing
 earlier results of King and Manivel, we classify all
 such isomorphisms when $\lambda$ and $\mu$ are conjugate partitions and
 when one of $\lambda$ or $\mu$ is a rectangle. We give a complete classification when
$\lambda$ and $\mu$ each have at most two rows or columns or is a hook partition and
 a partial classification when $\ell = m$. 
 As a corollary of a more general result on Schur functors labelled by skew partitions
 we also determine all cases when $\nabla^\lambda \!\Sym^\ell \!E$ is irreducible.
 The methods used are from representation
 theory and combinatorics; in particular, we make explicit the close connection
 with MacMahon's enumeration of plane partitions, and prove a new $q$-binomial identity in this setting.
\end{abstract}

\maketitle
\thispagestyle{empty}
\section{Introduction}

Let $\SL_2(\C)$ be the special linear group of $2 \times 2$ complex
matrices of determinant~$1$ 
and let $E$ be its natural $2$-dimensional representation.
The irreducible complex representations of $\SL_2(\C)$ are, up to isomorphism, precisely the symmetric
powers $\Sym^n \E$ for $n \in \N_0$.
A classical result, discovered by Cayley and Sylvester in the setting of invariant theory, 
states that if $a$, $b \in \N$ then 
the representations
$\Sym^a \Sym^b \E$ and 
$\Sym^b \Sym^a \E$ of $\SL_2(\C)$ are isomorphic. 
More recently, King and Manivel independently proved that $\nabla^{(a^b)} \Sym^{b+c-1} \E$ is invariant,
up to $\SL_2(\C)$-isomorphism, under permutation of $a$, $b$ and $c$. Here $\nabla^{(a^b)}$ 
is an instance of the Schur functor $\nabla^\lambda$, defined in \S\ref{subsec:Schur}.
Motivated by these results, the purpose of this article is to make
a systematic study of when there is a \emph{plethystic isomorphism}
\begin{equation} \nabla^\lambda \SSym^\ell \E \cong \nabla^\mu \SSSym^m \E \label{eq:SLiso} \end{equation} 
of $\SL_2(\C)$-representations. By taking the characters of each side,~\eqref{eq:SLiso} is equivalent to
\begin{equation} 
q^{-\ell |\nlambda|/2} s_\nlambda(1,q,\ldots, q^\ell) = q^{-\m |\nmu|/2} s_\nmu(1,q,\ldots,q^{\m}). 
\label{eq:SchurEq}\end{equation}
where $s_\nlambda$ is the Schur function for the partition $\nlambda$. 
By Remark~\ref{remark:plethysm}, we also have $s_\lambda(1,q,\ldots, q^\ell) =
(s_\lambda \circ s_\ell)(1,q)$ where $\circ$ is the plethysm product.
Thus~\eqref{eq:SLiso} can be investigated using a circle of powerful combinatorial ideas;
these include Stanley's Hook Content Formula \cite[Theorem 7.12.2]{StanleyII}.

Our results
reveal numerous surprising isomorphisms, not predicted by any existing results in the literature,
and a number of new obstacles to plethystic isomorphism. In particular,
we 
prove a converse to the King and Manivel result. We also note Lemma~\ref{lemma:removable},
which implies that, in the typical case for~\eqref{eq:SchurEq}, the Young diagrams of
$\lambda$ and $\mu$ have
the same number of removable boxes. Borrowing from the title of \cite{KacDrum}, these are all cases
where 
one may `hear the shape of a partition'.

\subsection*{Main results}
Let $\ell(\nlambda)$ denote the number of parts of a partition $\nlambda$
and let $a(\lambda)$ denote its first part, setting $a(\varnothing) = 0$. 

\begin{definition}\label{defn:eqv}
Given non-empty partitions $\nlambda$ and $\nmu$ and $\ell$, $\m \in \N$ such that
\hbox{$\ell \ge \ell(\nlambda)-1$} and $\m \ge \ell(\nmu) -1$,
we set $\nlambda \eqv{\ell}{\m} \nmu$ if and only if
$ \nabla^\lambda \SSym^\ell \E \cong \nabla^\mu \SSSym^m \E$ as representations of $\SL_2(\C)$.
\end{definition}

We refer to the relation $\eqv{\ell}{m}$ as \emph{plethystic equivalence}.
By Lemma~\ref{lemma:columnRemoval}, we have $\lambda \eqvs{\ell(\lambda)-1}{\ell(\lambda)-1} \overline{\lambda}$,
where, by definition, 
$\overline{\lambda}$ is $\lambda$ with its columns of length $\ell(\lambda)$ removed.
Such plethystic equivalences arise from the triviality of the representation $\bigwedge^{\ell+1} 
\SSym^{\ell}$
of $\SL(E)$; as we show in Example~\ref{ex:ex} below, they can be dispensed
with by using  Lemma~\ref{lemma:transitive} and Proposition~\ref{prop:prime} to reduce
to the following `prime' case. 

\begin{definition}
A plethystic equivalence $\lambda \eqv{\ell}{m} \mu$ is \emph{prime} if 
$\ell \ge \ell(\lambda)$ and $m \ge \ell(\mu)$.
\end{definition}

To avoid technicalities, we state our first main theorem
in a slightly weaker form than in the main text. 
Given a partition $\lambda$ with Durfee square of size $d$,
let $\epar(\lambda)$ be the partition, shown in Figure~1 in~\S\ref{sec:conjugates}, 
obtained from the first $d$ rows
of $\lambda$ by deleting the maximal rectangle containing the Durfee square of $\lambda$. 
Let $\spar(\lambda)$ be defined analogously,
replacing rows with columns. Thus $\spar(\lambda) = \epar(\lambda')'$ where $\lambda'$ is the conjugate partition to~$\lambda$.
The `if' direction of the theorem below
was proved by King in \hbox{\cite[\S 4.2]{KingSU2Plethysms}}.

\begin{theorem}\label{thm:conjugates}
Let $\lambda$ and $\mu$ be partitions.
There exist infinitely many pairs $(\ell, \m)$ such that $\nlambda \eqv{\ell}{\m} \nmu$ if and only if
$\epar(\nlambda) = \spar(\nlambda)'$ and $\nmu = \nlambda'$. In this case, the pairs are 
all $(\ell, \m)$ such that $\ell = \ell(\nlambda)-1 + k$ and $\m = \ell(\nmu) - 1 + k$ 
for some $k \in \N_0$.
\end{theorem}

Our second main theorem sharpens this result to show that `infinitely many' may be replaced with `three', 
and, in the case of prime equivalences, with `two'.

\vbox{
\begin{theorem}\label{thm:multipleEquivalences} Let $\lambda$ and $\mu$ be partitions.
\begin{thmlist} 
\item
There are two distinct pairs $(\ell,m)$, $(\ellp,\mp)$
such that $\lambda \eqv{\ell}{m} \mu$ and $\lambda \eqvs{\ellp}{\mp} \mu$ 
are prime plethystic equivalences 
if and only if \emph{either} $\lambda = \mu$
\emph{or} $\epar(\lambda) = \spar(\lambda)'$ and $\lambda = \mu'$.
\item 
There are three distinct pairs
$(\ell,m)$, $(\ellp,\mp)$, $(\ellpp,\mpp)$ 
such that $\lambda \eqv{\ell}{m} \mu$, $\lambda \eqvs{\ellp}{\mp} \mu$ 
and $\lambda \eqvs{\ellpp}{\mpp} \mu$ 
if and only if \emph{either} $\lambda = \mu$
\emph{or} $\epar(\lambda) = \spar(\lambda)'$ and $\lambda = \mu'$.
\item 
There exist distinct $n$, $\np \in \N$ such that $\lambda\hskip-1pt \eqv{n}{n} \hskip-0.5pt\mu$ 
and $\lambda \eqvs{\np}{\np} \mu$ if and only if $\lambda = \mu$.
\end{thmlist}
\end{theorem}}

It is clear that no still sharper result
can hold in (i) or (iii); Example~\ref{ex:ex} below shows that
the same is true for (ii).

If $\ell(\lambda) \le r$, 
let $\lambda^{\circ r}$ denote the complementary partition
to $\lambda$ in the $r \times a(\lambda)$ box, defined formally by
$\lambda^{\circ r}_{r+1-i} = a(\lambda) - \lambda_{i}$ for $1 \le i \le r$.
The `if' direction of the following theorem was proved in \cite[\S 4.2]{KingSU2Plethysms}.

\begin{theorem}\label{thm:complements}
Let $r \in \N$ and let  $\nlambda$ be a partition with $\ell(\nlambda) \le r$. 
Then $\nlambda \eqv{\ell}{\ell} \nlambda^{\circ r}$ if and only
if $r = \ell + 1$ or $\nlambda = \nlambda^{\circ r}$.
\end{theorem}




Our fourth  main result includes the converse of the King and Manivel six-fold symmetries mentioned at the outset.
Again, to avoid technicalities, we state it in a slightly weaker form below.

\newcommand{\rectangleNoCRtext}{Let $\nlambda$ be a partition and let $a$, $b$, $c \in \N$.
If $\ell \ge \ell(\nlambda)$ then $\nlambda \eqv{\ell}{b+c-1} (a^b)$ if and only if $\nlambda$ 
is rectangular,
with $\lambda = (\ap^\bp)$, and $(\ap, \bp, \ell -\bp + 1)$ is a permutation of $(a,b,c)$.}
\setcounter{rectanglethm}{\value{theorem}}

\begin{theorem}\label{thm:rectangles}
\rectangleNoCRtext
\end{theorem}

Extending a remark of King and part of Manivel's proof, we show that the `if' direction of 
Theorem~\ref{thm:rectangles}
is the representation-theoretic realization of the six-fold symmetry group
of plane partitions; these symmetries generalize conjugacy for ordinary partitions.
MacMahon \cite{MacMahon1896} found a beautiful closed form for the generating
function of plane partitions that makes these symmetries algebraically obvious. We use
this to prove a new $q$-binomial identity that implies the symmetry by swapping $b$ and $c$.

Taking $b=1$ in the full version of Theorem~\ref{thm:rectangles} we 
obtain the following classification, notable 
because of the connection with Hermite reciprocity and Foulkes' Conjecture discussed
later in the introduction.

\begin{corollary}\label{cor:oneRow}
Let $\nlambda$ be a partition and let $a$, $c \in \N$. There is an
isomorphism $\nabla^\nlambda \SSym^\ell \E \cong \Sym^a \Sym^c \E$ of $\SL_2(\C)$-representations
if and only if $\nlambda$ is obtained by adding columns of length
$\ell+1$ to one of the partitions $(a)$, $(1^a)$, $(c)$, $(1^c)$, $(a^c)$, $(c^a)$,
and $\ell$ is respectively $c$, $a+c-1$, $a$, $a+c-1$, $c$, $a$.
\end{corollary}

The entirely new results begin in \S\ref{sec:irreducible} where we 
consider skew Schur functions and prove a necessary and sufficient condition
for $s_{\nlambda / \nlambda^\star}(1,q,\ldots, q^\ell)$ to be equal, 
up to a power of $q$, to $1+q+\cdots + q^n$ for some $n \in \N_0$.
This is equivalent to the irreducibility of 
$\nabla^{\nlambda/\nlambda^\star}\E$. 
(We outline a construction of skew Schur functors in Remark~\ref{remark:skewNabla} below.)
Specializing this result to partitions, we characterize all irreducible plethysms.

\begin{corollary}\label{cor:irreducible}
Let $\nlambda$ be a partition and let $\ell \in \N$. There exists $n \in \N_0$ 
such that $\nabla^\nlambda \Sym^\ell \E \cong \Sym^n \E$
if and only if one of:
\begin{thmlist}
\item $\ell = 1$ and $\ell(\nlambda) \le 2$;
\item $\ell \ge 2$ and either $\lambda = (\p^{\ell+1})$ or
$\nlambda = (\p,(\p-1)^\ell)$ or $\lambda = (\p^\ell,\p-1)$ for some $\p \in \N$.
\end{thmlist}
\end{corollary}

Since $\nabla^\nlambda \SSym^\ell \E$ is irreducible if and only if $\lambda \eqv{\ell}{1}{(n)}$ 
for some $n \in \N_0$, Corollary~\ref{cor:irreducible} can also be obtained 
 from the full version of Theorem~\ref{thm:rectangles}, or, more directly,
from Corollary~\ref{cor:oneRow}.


In \S\ref{sec:twoRowAndHook} we classify all equivalences 
$\nlambda \eqv{\ell}{\m} \nmu$
when $\nlambda$ and $\nmu$ are two-row, two-column or hook partitions. 
To give a good flavour of this, we state the result for
equivalences between two-row and hook partitions.

\begin{theorem}
Let $\nlambda$ be a non-hook partition with exactly two parts and let~$\nmu$
 be a hook partition with 
non-zero arm length and leg length. If $\ell \ge \ell(\nlambda)$ 
and $\m \ge \ell(\nmu)$ then
$\nlambda \eqv{\ell}{\m} \nmu$ if and only if the relation is one~of
\smallskip
\begin{thmlist}
\item $(a,b) \raisebox{3pt}{$\begin{matrix*}[l]\eqv{a-b+1}{a} (a-b+1,1^b) \\ 
                                               \eqv{a-b+1}{2(a-b)} (b+1,1^{a-b}),\end{matrix*}$}$
\item $(3\ell-3, 2\ell-1) 
\raisebox{3pt}{$\begin{matrix*}[l]\eqv{\ell}{3\ell-4} (\ell+1,1^{\ell-2}) \\
                                  \eqv{\ell}{3\ell-2} (\ell-1,1^{\ell}).\end{matrix*}$}$
\end{thmlist}
\end{theorem}

In \S\ref{sec:equalDegree} we consider
the case of prime equivalences in which $\ell = \m$. 
Building on Theorem~\ref{thm:complements}, we obtain the following partial classification.

\begin{theorem}\label{thm:equalDegree}
Let $\nlambda$ and $\nmu$ be partitions and let $\ell \ge \ell(\lambda), \ell(\mu)$. 
\begin{thmlist}
\item[\emph{(a)}] If $\ell \le 4$ then $\nlambda \eqv{\ell}{\ell} \nmu$ if and only if $\nlambda = \nmu$
or $\nlambda^{\circ (\ell+1)} = \nmu$.
\item[\emph{(b)}]
For all $\ell \ge 5$ there exist infinitely many distinct pairs 
$(\nlambda, \nmu)$ such that $\nlambda \not= \mu$, $\nlambda 
\not= \nmu^{\circ (\ell+1)}$, and $\nlambda \eqv{\ell}{\ell} \nmu$.
\end{thmlist}
\end{theorem}

We end in \S\ref{sec:solitary} where we show that there exist infinitely many
partitions whose only plethystic equivalences are the inevitable column removals
from Lemma~\ref{lemma:columnRemoval}
and the complement equivalences from Theorem~\ref{thm:complements}.

\begin{theorem}\label{thm:solitary}
Let $\delta(k) = (k,k-1,\ldots, 1)$ and let $\ell$, $m \in \N$. 
Let $\mu$ be a partition.
Suppose that $\ell \ge k$ and $m \ge \ell(\mu)$ and that $\mu\not= \delta(k)$.
Then $\delta(k) \eqv{\ell}{m} \mu$ if and only if $\ell = m$, $\ell > k$ and 
$\mu = \delta(k)^{\circ (\ell+1)}$.
\end{theorem}

The following example is chosen to illustrate many of our main results.

\begin{example}
\label{ex:ex}
Let $b$, $c$, $d \in \N$. Since $\bigl( (b+c)^c, c^d \bigr)$ is the complement
of $\bigl( (b+c)^b, b^d \bigr)$ in the $(b+c+d) \times (b+c)$ box, 
Theorem~\ref{thm:complements} implies that 
\[ \bigl( (b+c)^b, b^d \bigr) \eqv{b+c+d-1}{b+c+d-1} \bigl( (b+c)^c, c^d \bigr). \]
The column removals relevant to Lemma~\ref{lemma:columnRemoval} are
$\overline{\bigl( (b+c)^b, b^d \bigr)} = (c^b)$ and 
$\overline{\bigl( (b+c)^c, c^d \bigr)} = (b^c)$.
By Lemma~\ref{lemma:columnRemoval} there are non-prime plethystic equivalences
$\bigl((b+c)^b, b^d \bigr) \eqv{b+d-1}{b+d-1} (c^b)$ and
$\bigl((b+c)^c, c^d \bigr) \eqv{c+d-1}{c+d-1} (b^c)$. 
By either Theorem~\ref{thm:conjugates} or Theorem~\ref{thm:rectangles}, we
have $(c^b) \eqv{b+d-1}{c+d-1} (b^c)$. Thus
\[ \bigl( (b+c)^b, b^d \bigr) \eqvs{b+d-1}{b+d-1} (c^b) \eqvs{b+d-1}{c+d-1} (b^c) \eqvs{c+d-1}{c+d-1}
\bigl( (b+c)^c, c^d \bigr). \]
By Lemma~\ref{lemma:transitive} this chain can be read as the factorization of
a non-prime plethystic equivalence
\[ \bigl( (b+c)^b, b^d \bigr) \eqv{b+d-1}{c+d-1} \bigl( (b+c)^c, c^d \bigr). \]
By Theorem~\ref{thm:multipleEquivalences}(ii), there are precisely two plethystic equivalences
between $\bigl( (b+c)^b, b^d \bigr)$ and $\bigl( (b+c)^c, c^d \bigr)$, namely the two found above.
As expected from this theorem, only one of these equivalences is prime.

\end{example}

\subsection{Outline}
In the remainder of this introduction we illustrate the critical Theorem~\ref{thm:eqvConds}, 
which collects a number of equivalent conditions for 
the plethystic equivalence in Definition~\ref{defn:eqv} 
by giving
two short proofs that $\Sym^a \SSym^b \E \cong \Sym^b\SSym^a \E$ for all $a$, $b \in \N_0$. 
In the spirit of
this work, one proof also gives a converse. 
We then give a brief literature survey, organized around the different generalizations
of this isomorphism.

In \S\ref{sec:background} we construct the irreducible representations of $\SL_2(\C)$ as the
symmetric powers $\Sym^\ell \E$, and give other basic results from representation theory.
We then give an explicit model for the representations $\nabla^\lambda \SSym^\ell \E$.
While $\nabla^\nlambda E$ is non-zero only if \hbox{$\ell(\nlambda) \le 2$}, 
the representation $\nabla^\nlambda \SSym^\ell \E$ is non-zero
whenever $\ell \ge \ell(\nlambda)-1$. This explains the ubiquity of
this condition in this work, and why we require the generality
of Schur functors, despite working only with representations of $\GL_2(\C)$ and
its subgroups. To
make the paper largely self-contained, we 
end by defining Schur functions.

The reader may prefer to treat \S\ref{sec:background} as a reference
and begin reading in~\S\ref{sec:eqvConds}, where we state and prove
Theorem~\ref{thm:eqvConds}. 
In \S\ref{sec:basicEqvs}
we collect some useful basic properties of the relations~$\!\!\eqv{\ell}{\m}$.
In \S\S\ref{sec:conjugates}--\ref{sec:solitary} we prove
the main results, as already outlined. 
Theorem~\ref{thm:complements} requires the statement of Theorem~\ref{thm:multipleEquivalences},
which in turn uses the statement of Theorem~\ref{thm:conjugates}; several later
theorems need the statement of Theorem~\ref{thm:complements}.
Apart from this, the sections may be read independently.

\subsection{Hermite reciprocity}\label{subsec:Hermite}

The isomorphism $\Sym^\ah \SSym^\bh \E \cong \Sym^\bh \SSym^\ah \E$ of $\GL_2(\C)$-representations
for $\ah$, $\bh \in \N$
was first discovered, in the context of invariant theory, by
Hermite \cite[end \S 1]{Hermite}. It was
observed by Cayley in \cite[\S 20]{CayleyQuantics}, where he acknowledges Hermite's prior discovery;
 some special cases may be seen in Sylvester \cite{SylvesterForms}, published in the same 
journal issue as~\cite{Hermite}. Thus it is also known (for instance in the title of \cite{Manivel})
as the Cayley--Sylvester formula. An invariant theory proof in modern
language may be found in \cite[3.3.4]{SpringerInvariantTheory}. Another elegant proof,
using the symmetric group, is in
\cite[Corollary~2.12]{GiannelliArchMath}.
We offer two proofs that illustrate
different conditions in Theorem~\ref{thm:eqvConds}. Each shows
that $(\ah)\eqv{\bh}{\ah} (\bh)$, or equivalently, Hermite reciprocity for representations of
~$\SL_2(\C)$.
Then, since the degrees on each side are equal,
it follows from Proposition~\ref{prop:eqvCondsMoreover} that there is a $\GL_2(\C)$-isomorphism.
 The first proof is well known, and is sketched 
in \cite[Exercise 6.19]{FH}; later in~\S\ref{subsec:rectanglesIf} we give its generalization to plane partitions.
Yet another proof (including the converse) can be given
using Theorem~\ref{thm:eqvConds}(i).

\begin{proof}[Proof by tableaux]
By~Theorem~\ref{thm:eqvConds}(g), we have $(\ah) \eqv{\bh}{\ah} (\bh)$ if and only if
$|\Seb{e}{\bh}{\ah}| = |\Seb{e}{\ah}{\bh}|$ for all $e \in \N_0$, where, by definition,
$\Se{e}{\bh}{\ah}$ is the set of semistandard tableaux of shape $(\ah)$ with entries
from $\{0,1,\ldots, \bh\}$ whose sum of entries is~$e$. 
Let $t$ be such a tableau, having exactly $c_i$ entries of $i$ for each $i \in \{0,1,\ldots, b\}$.
Then, reading its unique row from right to left, and ignoring any zeros,~$t$ encodes the partition 
$(\bh^{c_\bh},\ldots, 1^{c_1}\hskip-0.5pt)$ of $e$.
Hence $|\Se{e}{\bh}{\ah}|$ is the number of partitions of $e$
whose diagram is contained in the $\ah \times \bh$ box. By conjugating  partitions, this
number is invariant under swapping $\ah$ and~$\bh$.
\end{proof}

\begin{proof}[Proof by Stanley's Hook Content Formula]
The content of the partition
$(\ah)$ is $\{0,1,\ldots, \ah-1\}$, and its hook lengths are $\{1,2,\ldots, \ah\}$.
(These terms are defined in  Definition~\ref{defn:hlContent}.)
By Theorem~\ref{thm:eqvConds}(h),
$(\ah) \eqv{\bh}{m} (n')$ if and only if
\[ \{\bh+1,\bh+2, \ldots, \ah+\bh\} \bigl/ \{1,2,\ldots, \ah\} = \{m+1,m+2,\ldots, 
m+n'\} 
\bigl/ \{1,2,\ldots, n'\}. \]
where $/$ denotes a difference multiset, as defined in \S\ref{subsec:differenceMultiset}. 
Equivalently, the multiset unions
$\{\bh+1,\bh+2, \ldots, \ah+\bh\} \,\cup\, \{1,2,\ldots, n'\}$ and 
$\{m+1,m+2,\ldots,m+n'\} \,\cup\, \{1,2,\ldots, \ah\}$ are equal.
If $\ah = n'$ then, cancelling $\{1,2,\ldots, \ah\}$ from each side, we see that
$\bh = m$, giving a trivial solution.
Otherwise we may suppose by symmetry that $\ah < n'$. Now $\ah+1$ is in the first union 
and so $m = \ah$;  comparing greatest elements we see that $n' = \bh$. We
therefore have $n'= \bh$ and $m = \ah$, corresponding to Hermite
reciprocity.
\end{proof}

We remark that the first proof shows that that partitions contained 
in the $n \times \ell$ box are enumerated, according to their size, 
by a character of $\SL_2(\mathbb{C})$. In particular by Theorem~\ref{thm:eqvConds},
the sequence is unimodal that is,
first weakly increasing and then weakly decreasing.

\subsection{Literature on $\SL_2(\C)$-plethysms}
By Hermite reciprocity, the multiplicity of any Schur function
$s_{(\ell n - d, d)}$ labelled by a two-part partition
is the same in $s_{(n)} \circ s_{(\ell)}$ and $s_{(\ell)} \circ s_{(n)}$.
More generally, Foulkes conjectured in \cite{Foulkes} that if $n \ge \ell$ then
$s_{(n)} \circ s_{(\ell)} - s_{(\ell)} \circ s_{(n)}$ is a non-negative integral linear combination
of Schur functions; Foulkes' Conjecture has been proved only when $n \le 5$ (see \cite{CheungIkenmeyerMkrtchyan})
and when $n$ is very large compared to $\ell$ (see \cite{Brion}).
A number of `stability' results on plethysm are relevant to 
this setting. For example, a special case of the theorem on page 354 of \cite{Brion}
implies that the multiplicity of $\Sym^r \E$ in $\Sym^n \Sym^\ell \E$ 
is at most the multiplicity of $\Sym^{r+n} \E$ in $\Sym^n \Sym^{\ell+1} \E$.
The first proof of Hermite reciprocity above translates this into a non-trivial combinatorial
result comparing partitions of $r$ in the $n \times \ell$ box and partitions of $r+n$
in the $n \times (\ell+1)$ box.

In \cite{KingSU2Plethysms}, King proves the `if' direction of Theorem~\ref{thm:rectangles},
and sketches a proof of a weaker version of the converse. He mentions
as one motivation the Wronskian
isomorphism $\bigwedge^{b} \SSym^{b+c-1} \E \cong \Sym^b \SSym^c \E$ of representations
of the compact subgroup $\mathrm{SU}_2(\C)$ of $\SL_2(\C)$. This is interpreted by Wybourne
in \cite{Wybourne} as an equality between the number of completely antisymmetric
states of $b+c-1$ identical bosons each of angular momentum $c/2$ and the number
of symmetric states of $b$ identical bosons each of angular momentum $c/2$. 
(There is a typographic error between (13) and (14) in \cite{Wybourne}; $m+1+n$ should be
$m+1-n$, as in (13).)
This realizes the well-known
equality between the number of $c$-multisubsets of $\{1,\ldots, b\}$
and the number of $c$-subsets of $\{1,\ldots, b+c-1\}$. 
By Lemma 4.1 in \cite{McDowellI}, the special
case of the Wronskian isomorphism $\bigwedge^2 \SSym^{c+1} \E \cong \Sym^2 \SSym^c \E$
holds when $E$ is the natural representation of any finite special
linear group $\SL_2(\mathbb{F}_q)$. It would be interesting to have further examples of such
`modular plethysms'.

The second main result of \cite{BowmanPaget} classifies all partitions $\lambda$ and $\nu$
such that the plethysm $s_\lambda \circ s_\nu$ is equal to a single Schur function.
Apart from the obvious $s_\lambda \circ s_{(1)} = s_\lambda$, the
only examples are $s_{(1,1)} \circ s_{(1,1)} = s_{(2,1,1)}$ and $s_{(1,1)} \circ s_{(2)} = 
s_{(3,1)}$. By Remark~\ref{remark:plethysm} and~\eqref{eq:SchurSpec},
the formal character of $\nabla^\lambda \SSym^\ell \E$, evaluated at $1$ and $q$ is 
$(s_\lambda \circ s_{(\ell)})(1,q)$.
Our Corollary~\ref{cor:irreducible} therefore shows that
there are more irreducible plethysms when we work with symmetric functions truncated
to two variables.
The equality $(s_{(1^5)} \circ s_{(2)})(x_1,x_2,x_3) = s_{(2,2)}(x_1,x_2,x_3)$,
corresponding to the isomorphism $\bigwedge^5 \SSym^2 U \cong \nabla^{(2,2)} U$ 
where $U$ is a
$3$-dimensional complex vector space, gives a similar `non-generic' example for three variables.

Corollary~\ref{cor:irreducible} is itself a special case of Theorem~\ref{thm:irreducible}
on skew Schur functors. While we do not require it in this work, we note that
a combinatorial formula
for the corresponding plethysm $s_{\lambda/\lambda^\star}(1,q,\ldots,q^\ell)
= (s_{\lambda / \lambda^\star} \circ s_{(\ell)})(1,q)$ is given by Morales, Pak and Panova
in \cite[Theorem 1.4]{MoralesPakPanovaI} in terms of certain `excited' Young diagrams
of shape $\lambda / \lambda^\star$ first defined by
Ikeda and Naruse in \cite{IkedaNaruse}. This result is a generalization
of Stanley's Hook Content Formula (see \cite[Theorem 7.21.2]{StanleyII}),
one of the main tools in this work.
As a corollary the authors obtain a formula
due to Naruse \cite{NaruseTalk} for the number of standard tableaux of shape 
$\lambda / \lambda^\star$.

For further general
background on plethysms we refer the reader to \cite{LoehrRemmel} and
to the survey in~\cite{PagetWildonGeneralizedFoulkes}.

\section{Background}
\label{sec:background}

\subsection{Representations of $\SL_2(\C)$}

Let $G$ be a subgroup of $\GL_2(\C)$ containing $\SL_2(\C)$.
A representation $\rho : G \rightarrow \GL(V)$ is said to be \emph{polynomial}
if, with respect to a chosen basis of $V$, each matrix entry in $\rho(g)$ is a polynomial
in the matrix entries of $g \in G$. If these polynomials
all have the same degree $r$, we say that $V$ has \emph{degree} $r$.
We define the \emph{character} of a polynomial representation $V$ of $G$ to be
the unique two variable polynomial~$\Phi_V$ such that
\[\Tr_V   \left( \begin{matrix} \alpha & 0 \\  0 & \beta \end{matrix} \right)  = \Phi_V(\alpha,
\beta) \]
for all non-zero $\alpha$, $\beta \in \C$. We define the \emph{$Q$-character} of $V$
to be the Laurent polynomial $\Psi_V$ such that $\Psi_V(Q) = \Phi_V(Q^{-1}, Q)$.

Remarkably every smooth representation of $\SL_2(\C)$ is polynomial. Thus the following
summary theorem is a restatement of a basic result in Lie Theory. 

\begin{theorem}\label{thm:SLbasic} Let $V$
be a polynomial representation of $\SL_2(\C)$.
Then~$V$ is isomorphic to a direct
sum of irreducible representations. 
Moreover, if $V$ is irreducible then there exists a unique $\ell \in \N_0$
such that $V \cong \Sym^\ell \E$.
\end{theorem}

\begin{proof}
See~\cite[Chapter~12]{FultonHarrisReps}.
\end{proof}

Let $E$ be a $2$-dimensional
complex vector space with basis $e_1$, $e_2$.
The diagonal matrix with entries $1/\alpha$ and $\alpha$ acts on the
canonical basis element $e_1^{l-k} e_2^k$ of $\Sym^\ell \E$
by multiplication by $\alpha^{2k-\ell}$. Therefore
\begin{align}
\Phi_{\Sym^\ell \E}(x,y) &= x^\ell + \cdots + x y^{\ell-1} + y^\ell,\label{eq:GLirredChar} 
\intertext{and so}
\Psi_{\Sym^\ell \E}(Q) &= Q^{-\ell} + \cdots + Q^{\ell-2} + Q^\ell.\label{eq:SLirredChar}
\end{align}

\begin{lemma}\label{lemma:SLchar}
Let $V$ be a polynomial representation of $\SL_2(\C)$. Then $V$ is determined up to isomorphism
by its $Q$-character $\Psi_V$. Moreover, $\Psi_V(Q) = \Psi_V(Q^{-1})$.
\end{lemma}

\begin{proof}
Since the Laurent polynomials in~\eqref{eq:SLirredChar} are linearly independent, the result 
is immediate
from Theorem~\ref{thm:SLbasic}.
\end{proof}

\subsection{Partitions} Let $\Par(r)$ denote the set of partitions of $r \in \N$.
We write $|\lambda| = r$ if $\lambda \in \Par(r)$. We have already introduced
the notation $\ell(\lambda)$ for the number of parts of $\lambda$. If $i > \ell(\lambda)$
then we set $\lambda_i = 0$.
The \emph{Young diagram} of $\lambda$ is
the set
$\{ (i,j) : 1 \le i \le \ell(\lambda), 1 \le j \le \lambda_i\}$; we refer
to its elements as \emph{boxes}, and draw $[\lambda]$ using the `English' convention
with its longest row at the top of the page, as in Example~\ref{ex:polytabloid} below.

\subsection{Tableaux}\label{subsec:tableaux}
A \emph{$\lambda$-tableau} is a function
$t : [\lambda] \rightarrow \N_0$. If $t(i,j) = b$,
then we say that $t$ has \emph{entry} $b$ in box $(i,j)$,
and write $t_{(i,j)} = b$. If the entries of each row of $t$ are weakly increasing
when read from left to right we say that $t$ is \emph{row-semistandard}.
If the entries of each column of $t$ are strictly increasing when read from
top to bottom, we say that $t$ is \emph{column standard}.
If both conditions hold, we say that $t$ is \emph{semistandard}.
Let $\RSSYT_{\le \ell}(\lambda)$ 
and $\SSYT_{\le \ell}(\lambda)$ be the sets of row semistandard
and semistandard $\lambda$-tableaux respectively,
with entries in $\{0,1,\ldots, \ell\}$. Note that $0$ is permitted as an entry.
Given a permutation $\sigma$ of the boxes $[\lambda]$, and a $\lambda$-tableau $t$,
we define $\sigma \cdot t$ by $(\sigma \cdot t)(i,j) = t\bigl( \sigma^{-1}(i,j) \bigr)$.
Thus if $t$ has entry $b$ in box $(i,j)$ then $\sigma \cdot t$ has entry $b$ in
box $\sigma (i,j)$. Let $C(\lambda)$ be the group of all permutations that permute
within themselves boxes in the same column of $[\lambda]$.

We define the \emph{weight} of tableau $t$, denoted $|t|$, to be the sum of its entries.


\subsection{A construction of $\nabla^\lambda \SSym^\ell \E$}
\label{subsec:Schur}

Fix $\ell \in \N$ and let $V = \langle v_0, \ldots, v_\ell\rangle$
be an $(\ell+1)$-dimensional complex vector space.
Given a $\lambda$-tableau $t$ with entries from $\{0,1,\ldots, \ell \}$, define
\begin{equation}
\label{eq:tabloid} f(t) = \bigotimes_{i=1}^{\ell(\lambda)} \prod_{j=1}^{\lambda_i} v_{t(i,j)} \in
\bigotimes_{i=1}^{\ell(\lambda)} \Sym^{\lambda_i}V.\end{equation}
Define 
\begin{equation}
\label{eq:polytabloid} F(t) = \sum_{\tau \in C(t)} \sgn(\tau) f(\tau \cdot t).\end{equation}
We say that
$F(t)$ is the \emph{$\GL$-polytabloid corresponding to $t$}.
Observe that if $\sigma \in C(t)$ then 
\begin{equation}\label{eq:columnRelation} F(\sigma \cdot t) = \sgn(\sigma) F(t). \end{equation}
Hence $F(t) = 0$ if $t$ has a repeated entry in a column.

It is clear that $\{f(t) : t \in \RSSYT_{\le \ell}(\lambda) \}$  is
a  basis of $\bigotimes_{i=1}^{\ell(\lambda)} \Sym^{\lambda_i}V$.
Thus given any $g \in \GL(V)$, there exist unique coefficients 
$\alpha_s \in \C$ for $s \in \RSSYT_{\le \ell}(\lambda)$
such that
\[ g f(t) = \sum_{s \in \RSSYT_{\le \ell}(\lambda)} \alpha_s f(s). \]
It is routine to check that if $\sigma$ is a permutation of $[\lambda]$ then
$g f(\sigma \cdot t) = \sum_{s \in \RSSYT_{\le \ell}(\lambda)} \alpha_s f(\sigma \cdot s)$.
It now follows from the definition in~\eqref{eq:polytabloid}
that the linear span of the $F(t)$ for~$t$ a 
$\lambda$-tableau with
entries from $\{0,1,\ldots, \ell \}$ is a $\GL(V)$-subrepresentation of
\smash{$\bigotimes_{i=1}^{\ell(\lambda)} \Sym^{\lambda_i}V$}; this is  $\nabla^{\lambda}V$.
In particular, it is clear that $\nabla^{(n)}V \cong \Sym^n V$ for each $n \in \N_0$.

\begin{example}\label{ex:oneColumn}
By~\eqref{eq:columnRelation} the representation $\nabla^{(1^n)}V$
has as a basis all $\GL$-polytabloids $F(t)$ where $t$ is a standard $(1^n)$-tableau
with entries from $\{0,1,\ldots, \ell\}$. Moreover, the linear map 
$\nabla^{(1^n)} V \rightarrow \bigwedge^n V$ 
defined by
$F(t) \mapsto v_{t_{(1,1)}} \wedge \cdots \wedge v_{t_{(n,1)}}$
is an isomorphism of representations of $\GL(V)$. In particular, if $n = \ell+1$
then $\nabla^{(1^n)}$ is the determinant representation of $\GL(V)$.
\end{example}

More generally we have the following theorem.

\begin{theorem}\label{thm:standardBasis}
The $\GL$-polytabloids $F(s)$ for $s \in \SSYT_{\le \ell}(\lambda)$ are a
$\C$-basis of $\nabla^\lambda V$.
\end{theorem}

\begin{proof}
See either \cite[Proposition~2.11]{deBoeckPagetWildon} or \cite[Chapter 8]{FultonYT}.
\end{proof}

\begin{definition}\label{defn:nablaSym}
Let $\lambda$ be a partition and let $\ell \in \N$. 
Let $E$ be the natural representation of $\GL_2(\C)$.
Let $\rho : \GL_2(\C) \rightarrow \GL(V)$ be the
representation corresponding to $V$.
We define $\nabla^\lambda \SSym^\ell \E$ to be the 
restriction of the representation $\nabla^\lambda V$ of $\GL(V)$ to the image of 
$\rho$.
\end{definition}

Let $v_k = e_1^{\ell-k}e_2^k$ for $0 \le k \le \ell$ be the
canonical basis of $\Sym^\ell E$. Using this basis in Definition~\ref{defn:nablaSym},
the action of $g \in \GL(E)$ on a $\GL$-polytabloid
$F(s)$
may be computed by the following device:
formally replace each entry $b$ of $s$ with $gv_b$, expressed as a linear combination
of $v_0, v_1, \ldots, v_\ell$.
Then expand multilinearly, and use
the column relation~\eqref{eq:columnRelation} followed by Garnir relations
(see \cite[Corollary 2.6]{deBoeckPagetWildon} or \cite[Chapter 8]{FultonYT})
to express
the result as a linear combination of $\GL$-polytabloids $F(t)$ for semistandard tableaux~$t$.

\begin{example}\label{ex:polytabloid}
Take $\ell = 2$ so $V = \Sym^2 \E = \langle e_1^2, e_1e_2, e_2^2 \rangle$.
The action of a  lower-triangular matrix $g \in \GL_2(\C)$ on $V$ is given, with respect
to the chosen basis, by
\[ \left( \begin{matrix} \alpha & 0 \\ \gamma & \delta \end{matrix} \right) \longmapsto
\left( \begin{matrix} \alpha^2 & 0 & 0 \\ 2\alpha\gamma & \alpha\delta & 0 \\
\gamma^2 & \gamma\delta & \delta^2 \end{matrix} \right). \]
In its action on $\nabla^{(2,1)}\SSym^2 \E$ we have
\begin{align*} g \Ft{02}{1} &= F\Bigl(\, 
\raisebox{-0.2in}{\begin{tikzpicture}[x=1.6cm,y=-0.6cm]
\draw[line width=0.7pt] (0,0)--(2.85,0)--(2.85,1)--(1.75,1)--(1.75,2)--(0,2)--(0,0); 
\draw[line width=0.7pt] (1.75,0)--(1.75,1); \draw[line width=0.7pt] (0,1)--(1.75,1); 
\node at (0.85,0.5) {$\scriptstyle \alpha^2 v_0 + 2\alpha\gamma v_1 + \gamma^2 v_2$};
\node at (0.8,1.5) {$\scriptstyle \alpha\delta v_1 + \gamma\delta v_2$}; 
\node at (2.3,0.5) {$\scriptstyle \delta^2 v_2$};
\end{tikzpicture}}\,\Bigr) \\
&= \alpha^3 \delta^3 \Ft{02}{1} + \alpha^2 \gamma \delta^3 \Ft{02}{2} 
\\
& \hskip0.8in + 2\alpha \gamma^2 \delta^3 \Ft{12}{2} + \alpha \gamma^2 \delta^3 \Ft{22}{1} \\
&= \alpha^3 \delta^3 \Ft{02}{1} + \alpha^2 \gamma \delta^3 \Ft{02}{2} 
+ (2\alpha \gamma^2 \delta^3  - \alpha \gamma^2 \delta^3) \Ft{12}{2}.
 \end{align*}
 where the third line uses the column relation in~\eqref{eq:columnRelation}.
\end{example}

\begin{lemma}\label{lemma:SchurChar}
Let $\lambda$ be a partition and let $\ell \in \N_0$. We have
\[ \Phi_{\nabla^\lambda \SSym^\ell\E}(1,q) = \sum_{t \in \SSYT_{\le \ell}(\lambda)} q^{|t|}. \]
\end{lemma}

\begin{proof}
By Theorem~\ref{thm:standardBasis}, the $F(t)$ for $t \in \SSYT_{\le \ell}(\lambda)$ are
a basis of $\nabla^\lambda \Sym^\ell \E$. Let $g \in \GL_2(\C)$ be diagonal with
entries $\alpha$ and $\delta$. Let $\tau \in C(t)$ and let $u = \tau \cdot t$.
By~\eqref{eq:tabloid}, 
\[ g \cdot f(u) = \bigotimes_{i=1}^{\ell(\lambda)} \prod_{j=1}^{\lambda_i} (g \cdot v_{t(i,j)}).\]
Since $g \cdot v_k = \alpha^{\ell -k} \delta^k v_k$, we have $g \cdot f(u) =
\alpha^{\ell |\lambda| - |u|}\delta^{|u|} f(u)$ where 
$|u| = \sum_{(i,j) \in [\lambda]} u(i,j)$ 
is the weight $|u|$ defined above. This is also the weight of $t$.
Therefore each $F(t)$ is an eigenvector for $g$ with eigenvalue
$\alpha^{\ell |\lambda| - |t|}\delta^{|t|}$. The lemma follows. 
\end{proof}

\subsection{Symmetric functions and plethysm}
\label{subsec:plethysm}

Let $\C[x_0,x_1,\ldots ]$ be the polynomial ring in the indeterminates $x_0, x_1, \ldots $.
We define a \emph{symmetric function} $f$ to be a family $f^{(n)}(x_0,\ldots, x_n)$
of symmetric polynomials in $\C[x_0,x_1,\ldots ]$ such that
\begin{equation}\label{eq:compat} 
f^{(n)}(x_0,\ldots, x_m,0,\ldots, 0) = f^{(m)}(x_0,\ldots, x_m) \end{equation}
for all $m$, $n \in \N_0$ with $m \le n$.
The notation is simplified (without introducing any ambiguity) by writing 
$f(x_0,\ldots, x_\ell)$ for $f^{(\ell)}(x_0,\ldots, x_\ell)$. 

\begin{definition}\label{defn:SchurFunction}
Let $\lambda$ be a partition. Given a $\lambda$-tableau $t$ with entries
from $\N_0$, let $x^t = x_0^{a_0(t)}x_1^{a_1(t)} \ldots $ where $a_k(t)$ is the
number of entries of $t$ equal to $k \in \N_0$.
The \emph{Schur function} $s_{\lambda}$ is the symmetric function defined by
\[ s_{\lambda}(x_0,x_1,\ldots, x_\ell) = \sum_{t \in \SSYT_{\le \ell}(\lambda)} x^t. \]
\end{definition}

The compatibility condition~\eqref{eq:compat} is easily checked. Let $\C[q]$ be a polynomial ring.
Observe that 
when $x_k$ is specialized to $q^k$, the monomial $x^t$ becomes $q^{|t|}$, where, as usual,
$|t|$ is the weight of $t$. Therefore 
\begin{equation}\label{eq:SchurSpecTableaux}
s_\lambda(1,q,\ldots, q^\ell) = \sum_{t \in \SSYT_{\le \ell}(\lambda)} q^{|t|}.
\end{equation}
It follows immediately 
from our definition and Lemma~\ref{lemma:SchurChar} that
\begin{equation}\label{eq:SchurSpec}
\Phi_{\nabla^\lambda \SSym^\ell \E}(1,q) = s_\lambda(1,q,\ldots, q^\ell) .
\end{equation}
This equation is the main bridge we need between representation theory and combinatorics.

\begin{remark}\label{remark:plethysm}
The plethysm product of symmetric functions is defined in
\cite{LoehrRemmel}, \cite[Ch.~1, Appendix A]{MacDonald}
and \cite[Ch.~7, Appendix 2]{StanleyII}.
For our purposes, we may define
$(s_\lambda \circ s_{(\ell)})(x,y)$ 
 by formally
substituting the monomials summands of $s_{(\ell)}(x,y) = x^\ell + x^{\ell-1} y + \cdots
+ y^\ell$ for the $\ell+1$ variables in $s_\lambda(x_0, \ldots, x_\ell)$. 
That is, $(s_\lambda \circ s_{(\ell)})(x,y) = s_\lambda(x^\ell, x^{\ell-1}y, \ldots,
y^\ell)$. Hence $\Phi_{\nabla^\lambda \SSym^\ell \E}(1,q) =
(s_\lambda \circ s_{(\ell)})(1,q)$, as mentioned after~\eqref{eq:SchurEq} earlier. 
\end{remark}

For Theorem~\ref{thm:eqvConds}(i) we require the original  definition of 
Schur polynomials using determinants and antisymmetric polynomials. Given a 
sequence $(\gamma_0,\gamma_1,\ldots, \gamma_\ell)$ of non-negative integers, define
\begin{equation}
\label{eq:a} a_\gamma(x_0,x_1,\ldots, x_\ell) = \det \left( \begin{matrix}
x_0^{\gamma_0} & x_0^{\gamma_1} & \cdots & x_0^{\gamma_\ell} \\ x_1^{\gamma_0} & x_1^{\gamma_1} & \cdots &
x_1^{\gamma_\ell} \\ \vdots & \vdots & \ddots & \vdots \\ x_\ell^{\gamma_0} & x_\ell^{\gamma_1} & \cdots 
& x_\ell^{\gamma_\ell}
\end{matrix} \right). \end{equation}
By~\cite[Theorem 7.15.1]{StanleyII}, if $\ell \ge \ell(\gamma)-1$ then
\begin{equation}\label{eq:antiSym}
s_\lambda(x_0,x_1,\ldots, x_\ell) = \frac{a_{\lambda + (\ell,\ell-1,\ldots, 0)}(x_0,x_1,
\ldots, x_\ell)}{a_{(\ell,\ell-1,\ldots, 0)}(x_0,x_1,\ldots, x_\ell)}.
\end{equation}

\subsection{Stanley's Hook Content Formula}

\begin{definition}\label{defn:b}
Let $\lambda$ be a partition. We define the \emph{minimum weight} of $\lambda$, denoted $\b(\lambda)$, by
$\b(\lambda) = \sum_{j=1}^{a(\lambda)} \binom{\lambda'_j}{2}$.
%
\end{definition}

Equivalently, $\b(\lambda) = \sum_{i=1}^{\ell(\lambda)} (i-1)\lambda_i$.
Observe that $\b(\lambda)$ is the weight of the semistandard $\lambda$-tableau
having~$\lambda_i$ entries of $i-1$ in row $i$; as the terminology suggests,
this tableau has the minimum weight
of any tableau in $\SSYT_{\le \ell}(\lambda)$. 
It follows that $q^{\b(\lambda)}$ is the summand of $s_\lambda(1,q,\ldots, q^\ell)$ of minimum degree.

\begin{definition}\label{defn:hlContent}
Let $\lambda$ be a partition. The \emph{hook length} of $(i,j) \in [\lambda]$,
denoted $h_{(i,j)}(\lambda)$, is $(\lambda_i - i) + (\lambda'_j - j) + 1$.
The \emph{content} of $(i,j) \in [\lambda]$ is $j-i$.  Let $H(\lambda) = \{h_{(i,j)}(\lambda) : (i,j)
\in [\lambda]\}$ and $C(\lambda) = \{j-i : (i,j) \in [\lambda]\}$ be the corresponding
multisets. 
\end{definition}

For example, the unique greatest hook length of a non-empty partition $\lambda$ is
$h_{(1,1)}(\lambda) = \bigl(a(\lambda)-1\bigr) + 
\bigl(\ell(\lambda)-1 \bigr) + 1 = a(\lambda) + \ell(\lambda)-1$.
The least element of $C(\lambda)$ is $1-\ell(\lambda)$.
Therefore whenever $\ell \ge \ell(\lambda) - 1$ we have $C(\lambda) + l+1 \subseteq \N$.

For $m \in \N$, let $[m]_q$ be the \emph{quantum integer} defined~by
\begin{equation}\label{eq:quantumInteger} [m]_q = \frac{q^m-1}{q-1} = 1 + q + \cdots + q^{m-1}. \end{equation}




\begin{theorem}[Stanley's Hook Content Formula]\label{thm:HCF}
Let $\lambda$ be a partition and let $\ell \in \N$. Then
\[ s_\lambda(1,q,\ldots, q^\ell) = q^{\b(\lambda)}\frac{\prod_{(i,j) \in [\lambda]} [j-i+\ell+1]_q}
{\prod_{(i,j) \in [\lambda]} [h_{(i,j)}(\lambda)]_q} . \] 
\end{theorem}

\begin{proof}
This is a restatement of \cite[Theorem 7.21.2]{StanleyII}
using the quantum integer notation. Note that our $\ell$ appears in~\cite{StanleyII} as $\ell-1$.
\end{proof}

\subsection{Pyramids}\label{subsec:pyramids}

In this subsection we prove an antisymmetric analogue of Stanley's Hook Content Formula. Most
of the ideas may be found in \cite[\S 7.21]{StanleyII}, so no originality is claimed.

\begin{definition}\label{defn:differences}
We define the \emph{differences} $\delta(\lambda)$ of a partition $\lambda$
by $\delta(\lambda)_j =  \lambda_j - \lambda_{j+1} + 1$ for each $j \in \N$.
For $\ell \ge \ell(\lambda) - 1$,
let $\Delta_\ell(\lambda)$ be the multiset whose elements are all 
$\delta(\lambda)_j + \cdots + \delta(\lambda)_{k-1}$ for $1 \le j < k \le \ell+1$.
\end{definition}

Observe that if $j < k$ then $\lambda_j - \lambda_k + k - j = \delta(\lambda)_j + \cdots + \delta(\lambda)_{k-1}$.

\begin{lemma}\label{lemma:antiSymSpec}
Let $\lambda$ be a partition such that $\ell \ge \ell(\lambda)-1$.
There exists $c \in \N_0$ such that 
\[ s_\lambda(1,q,\ldots, q^\ell) = q^c
\frac{\prod_{1 \le j < k \le \ell+1} [\delta(\lambda)_j + \cdots + \delta(\lambda)_{k-1}]_q }{\prod_{1 \le j < k \le \ell+1}
[k-j]_q}. \]
\end{lemma}

\begin{proof}
Taking $\gamma = (0,1,\ldots, \ell)$ in~\eqref{eq:a} and transposing the matrix
we get the Vandermonde identity
\[ \prod_{0 \le j < k \le \ell} (x_j-x_k) = \det \left( \begin{matrix}
1 & 1 & \ldots & 1 \\
x_0 & x_1 & \ldots & x_\ell \\
\vdots & \vdots & \ddots & \vdots \\
x_0^\ell & x_1^\ell & \ldots & x_\ell^\ell \end{matrix} \right). \] 
By~\eqref{eq:a} we have
\[ a_{\gamma}(1,q,\ldots,q^\ell) = \det \left(\begin{matrix}
1 & 1 & \ldots & 1 \\
q^{\gamma_0} & q^{\gamma_1} & \ldots & q^{\gamma_\ell} \\
\vdots & \vdots & \ddots & \vdots \\
q^{\ell \gamma_0} & q^{\ell \gamma_1} & \ldots & q^{\ell \gamma_\ell} \end{matrix} \right). \]
Therefore, specializing $x_j$ to $q^{\gamma_j}$ in the Vandermonde determinant, we obtain
\[ a_\gamma(1,q,\ldots, q^\ell) = \!\!\prod_{0 \le j < k \le \ell} q^{\gamma_k}(q^{\gamma_j-\gamma_k}-1). \]
Set $\gamma_{j} = \lambda_{j+1} + \ell - j$ for $0 \le j \le \ell$, and use the observation
just before the lemma to get
\[ a_{\lambda + (\ell,\ldots,1,0)}(1,q,\ldots, q^\ell) = 
q^{C(\lambda)} \!\!\prod_{1 \le j < k \le \ell+1} (q^{\delta(\lambda)_j + \cdots + \delta(\lambda)_{k-1}} - 1)\]
for some $C(\lambda) \in \N_0$. The special case $\lambda = \varnothing$ gives the specialized
Vandermonde identity 
\[ a_{(\ell,\ldots,1,0)} = q^{C(\varnothing)} \!\!\prod_{1 \le j < k \le \ell+1} (q^{k-j} - 1). \]
Taking the ratio of these two equations and using~\eqref{eq:antiSym} 
we obtain
\[ s_{\lambda}(1,q,\ldots, q^\ell) = q^c \frac{\prod_{1 \le j < k \le \ell+1} (q^{\delta(\lambda)_j 
+ \cdots + \delta(\lambda)_{k-1}} - 1)}{\prod_{1 \le j < k \le \ell+1} (q^{k-j}-1)} \]
where $c = C(\lambda) - C(\varnothing)$. This is equivalent to the claimed identity.
\end{proof}

\begin{corollary}\label{cor:delta}
Let $\lambda$ be a partition and let $\ell \ge \ell(\lambda) - 1$. Then
\[ s_\lambda(1,q,\ldots, q^\ell) = q^{\b(\lambda)}\frac{\prod_{x \in \Delta_\ell(\lambda)} [x]_q}{[1]_q^\ell [2]_q^{\ell-1} \ldots [\ell]_q}.\]
\end{corollary} 

\begin{proof}
By Lemma~\ref{lemma:antiSymSpec} there exists $c \in \N_0$ such that
\[ s_\lambda(1,q,\ldots, q^\ell) = q^c \frac{\prod_{1 \le j < k \le \ell+1} [\delta(\lambda)_j + \cdots + \delta(\lambda)_{k-1}]_q }{\prod_{1 \le j < k \le \ell+1}
[k-j]_q}. \]
The factors in the numerator are the quantum integers from $\Delta_\ell(\lambda)$.
The factors in the denominator are the quantum integers from
$\{1^\ell, 2^{\ell-1}, \ldots, \ell \}$, where the exponents indicate multiplicities. By~\eqref{eq:SchurSpecTableaux},
$q^{\b(\lambda)}$ is the monomial of least degree in $s_\lambda(1,q,\ldots, q^\ell)$. Since each quantum
integer is congruent to $1$ modulo $q$, the result follows.
\end{proof}

It is convenient to display the elements of the multiset $\Delta_\ell(\lambda)$ in a \emph{pyramid}
of $\ell$ rows, numbered from $1$, in which row $i$ has entries 
\smash{$\delta(\lambda)_j + \cdots + \delta(\lambda)_{j+i-1}$},
for $j \in \{1,\ldots, \ell-(i-1)\}$. Thus, writing $P^{(i)}_j$ for the entry in position~$j$ of row~$i$
of the pyramid $P$, we have
 \smash{$P^{(i+1)}_j = P^{(i)}_j + P^{(i)}_{j+1} - P^{(i-1)}_{j+1}$}. By convention, we
set \smash{$P^{(0)}_j = 0$} for each $j$.


\begin{example}\label{ex:pyramid}
We take $\ell = m = 5$. The  partitions $(8,7,2,2)$ and $(8,6,3)$ have 
differences $(2,6,1,3,1)$ and $(3,4,4,1,1)$, respectively. Their
pyramids are
\[ \scalebox{0.9}{$\begin{matrix} 
2  & 6  & 1 & 3 & 1 \\ 
8  & 7  & 4 & 4 \\
9 & 10 & 5 \\
12 & 11 \\
13 \end{matrix},$} \qquad\qquad
\scalebox{0.9}{$\begin{matrix}
3  & 4  & 4 & 1 & 1 \\
7  & 8  & 5 & 2 \\
11 & 9  & 6 \\
12 & 10 & \\
13
 \end{matrix}$.}
\]
Each pyramid has multiset of entries $\{1^2,2,3,4^2,5,6,7,8,9,10,11,12,13\}$.  
By Corollary~\ref{cor:delta}, cancelling the equal denominators
 $[1]_q^5[2]_q^4[3]_q^3[4]_q^2[5]_q$ we see that
$s_{(8,7,2,2)}(1,q,q^2,q^3,q^4, q^5)$ and 
$s_{(8,6,3)}(1,q,q^2,q^3,q^4, q^5)$ are equal up to a power of~$q$.
By Theorem~\ref{thm:eqvConds}(e) below, 
$(8,7,2,2)\eqv{5}{5} (8,6,3)$.
\end{example}

This example is generalized in Proposition~\ref{prop:exceptionalEqv}.

\section{Equivalent conditions for the plethystic isomorphism}
\label{sec:eqvConds}

\subsection{Difference multisets}
\label{subsec:differenceMultiset}
The following formalism simplifies the main results of this section and is convenient
throughout this paper.

\begin{definition}
A \emph{difference multiset} is a pair 
$(X,Z)$ of finite multisubsets of $\N$, denoted $X / Z$.
If $x \in \N$ has multiplicity $a$ in $X$ and $b$ in $Z$, 
then the \emph{multiplicity} of $x$ in $X/Z$ is $a - b$. Two difference
multisets are \emph{equal} if their multiplicities agree for all $x \in \N$.
\end{definition}

Alternatively, a difference multiset may be regarded as an element of the
free abelian group on $\N$. This point of view justifies our definition
of equality and makes obvious many simple
algebraic rules for manipulating difference multisets.
For example, $X/Z = Y/W$ 
if and only if $X/Y = Z/W$. 

The following lemma is used implicitly in (4.8) in \cite{KingSU2Plethysms}.

\begin{lemma}\label{lemma:uniqueFactorization}
Let $X$ and $Y$ be finite multisubsets of $\N$. In the polynomial ring $\C[q]$, we have
$\prod_{x \in X} (q^x - 1) = \prod_{y \in Y} (q^{y} - 1)$ if and only if $X = Y$.
\end{lemma}

\begin{proof}
If either $X$ or $Y$ is empty the result is obvious. In the remaining cases, 
let $u$ be greatest such that $\prod_{x \in X} (q^x - 1)$ has $\mathrm{e}^{2\pi \mathrm{i}/u}$
as a root. By choice of $u$, $q^u-1$ is a factor in the left-hand side. 
Since $\mathrm{e}^{2\pi \mathrm{i}/u}$
is also a root of $\prod_{y \in Y} (q^{y} - 1)$, 
the same argument shows that $q^u-1$ is a factor in the right-hand side.
Therefore  $u = \max X = \max Y$ and
it follows inductively that $X = Y$.
\end{proof}

\begin{corollary}\label{cor:uniqueFactorization}
Let $X/Z$ and $Y/W$ be difference multisets. Working in the field of fractions of $\Q[q]$, we have
\[ \frac{\prod_{x \in X} (q^x-1)}{\prod_{z \in Z} (q^z-1)} = \frac{\prod_{y \in Y} (q^y-1)}{\prod_{w \in W} (q^w-1)} \]
if and only if $X/Z = Y/W$. 
\end{corollary}

\begin{proof}
Multiply through by $\prod_{z \in Z} (q^z-1) \prod_{w \in W} (q^w-1)$ and then
apply Lemma~\ref{lemma:uniqueFactorization}.
\end{proof}

We apply this corollary to the polynomial quotients in Theorem~\ref{thm:HCF}
and Corollary~\ref{cor:delta} in the proof of Theorem~\ref{thm:eqvConds} below.

\subsection{Portmanteau Theorem}
Recall from~\S\ref{subsec:tableaux} that the weight of a tableau, denoted $|t|$, is its sum of entries.
The minimum weight $b(\lambda)$ is defined in Definition~\ref{defn:b}.
Given a partition $\lambda$ and $\ell \in \N_0$, let 
$\Se{e}{\ell}{\lambda}$ be the
set of all semistandard $\lambda$-tableaux with entries from $\{0,1,\ldots, \ell\}$ 
whose weight is $e$. Thus~\eqref{eq:SchurSpecTableaux} can be restated as
\begin{equation}\label{eq:Se} 
s_\lambda(1,q,\ldots, q^\ell) = \sum_{e \in \N_0} |\Se{e}{\ell}{\lambda}| q^e. 
\end{equation}
In~(h) and (i) below the notation indicates a difference
multiset, as defined in the previous subsection. The multisets
$C(\lambda)$ and $H(\lambda)$ are defined in Definition~\ref{defn:hlContent}
and $\Delta_\ell(\lambda)$ is defined in Definition~\ref{defn:differences}.

\begin{theorem}\label{thm:eqvConds}
Let $\lambda$ and $\mu$ be partitions and let $\ell$, $m \in \N$ be such that $\ell \ge \ell(\lambda)-1$
and $m \ge \ell(\mu) - 1$.
The following are equivalent:
\begin{thmlista}
\item $\lambda \eqv{\ell}{m} \mu$;
\item $\nabla^\lambda \SSym^\ell \E \cong \nabla^\mu \SSym^m \E$ as representations
of $\SL_2(\C)$;
\item $\Psi_{\nabla^\lambda \SSym^\ell \E}(Q) = \Psi_{\nabla^\mu \SSym^m \E}(Q)$; 
\item $q^{-\ell |\lambda|/2} s_\lambda(1,q,\ldots, q^\ell) = q^{-m |\mu|/2} s_\mu(1,q,\ldots, q^m)$;
\item $s_\lambda(1,q,\ldots, q^\ell) = q^d s_\mu(1,q,\ldots, q^m)$ for some $d \in \Z$;
\item $q^{-\b(\lambda)} s_\lambda(1,q,\ldots, q^\ell) = q^{-\b(\mu)} s_\mu(1,q,\ldots, q^m)$;
\item there exists $d \in \Z$ such that $|\Se{e+d}{\ell}{\lambda}| = |\Se{e}{m}{\mu}|$ for all $e \in \N_0$;
\item $(C(\lambda) + \ell + 1) / H(\lambda) = (C(\mu) + m + 1) / H(\mu)$;
\item $\Delta_\ell(\lambda) / \{1^\ell, 2^{\ell -1}, \ldots, \ell \} = 
\Delta_m(\mu) / \{1^{m}, 2^{m - 1}, \ldots, m \}$.
\end{thmlista}
Moreover if any of these conditions hold then $-\frac{\ell|\lambda|}{2} + \b(\lambda) = -\frac{m|\mu|}{2} +\b(\mu)$,
each side in \emph{(d)} is unimodal  and centrally symmetric and about its constant term,
and the constant $d$ in \emph{(e)} and \emph{(g)} is $b(\lambda) - b(\mu)$.
\end{theorem}

\begin{proof}
By Definition~\ref{defn:eqv}, (a) and (b) are equivalent.
By Lemma~\ref{lemma:SLchar}, (b) and (c) are equivalent.
By definition $\Psi_{\nabla^\lambda \SSym^\ell \E}(Q) = 
\Phi_{\nabla^\lambda \SSym^\ell \E}(Q^{-1},Q)$. Hence, by~\eqref{eq:SchurSpec}
and the homogeneity of $\nabla^\lambda \SSym^\ell \E$,
\[ \Phi_{\nabla^\lambda \SSym^\ell \E}(q^{-1/2}, q^{1/2}) = q^{-\ell |\lambda|/2} s_\lambda(1,q,\ldots, q^\ell) . \]
Therefore (c) and (d) are equivalent.
Clearly (d) implies (e).
Conversely, suppose that (e) holds, with $q^d s_\lambda(1,q,\ldots, q^\ell) = s_\mu(1,q,\ldots, q^m)$. 
By the previous displayed equation
and Lemma~\ref{lemma:SLchar}, $q^{-\ell |\lambda|/2} s_\lambda(1,q,\ldots, q^\ell)$ 
is a linear combination of polynomials of the form $q^{-b/2} + q^{-(b-1)/2} + \cdots + q^{b/2}$.
Hence it is centrally symmetric and unimodal about its constant term 
Now, comparing points of central symmetry, (e) implies that $d + \ell|\lambda|/2 = m |\mu|/2$. 
Multiplying both sides of $q^d s_\lambda(1,q,\ldots, q^\ell) = s_\mu(1,q,\ldots, q^m)$ by $q^{-d - \ell|\lambda|/2}
= q^{-m |\mu|/2}$, we obtain (d).
We noted before~Definition~\ref{defn:hlContent} 
that $s_\lambda(1,q,\ldots, q^\ell)$ has minimum degree 
summand $q^{\b(\lambda)}$.
Therefore (e) and~(f) are equivalent. 
By~\eqref{eq:Se},
(e) and (g) are equivalent.
The remainder of the `moreover' part now follows by comparing the $q$ powers in (d) and (f).
By Stanley's Hook Content Formula, as stated in Theorem~\ref{thm:HCF}, (f)
holds if and only if
\[ \frac{\prod_{(i,j) \in [\lambda]} [j-i+\ell+1]_q}
{\prod_{(i,j) \in [\lambda]} [h_{(i,j)}(\lambda)]_q}
=
\frac{\prod_{(i,j) \in [\mu]} [j-i+m+1]_q}
{\prod_{(i,j) \in [\mu]} [h_{(i,j)}(\mu)]_q}. \]
By Corollary~\ref{cor:uniqueFactorization}, this is equivalent to (h).
Finally (f) and (i) are equivalent by the same argument with
Corollary~\ref{cor:uniqueFactorization}
 applied
to the right-hand side in Corollary~\ref{cor:delta}.
\end{proof}

\subsection{Extending plethystic isomorphisms}
We end by considering when an $\SL_2(\C)$-isomorphism $\nabla^\lambda \SSym^\ell \E \cong \nabla^\mu \SSym^m \E$ 
extends to an overgroup of $\SL_2(\C)$. 
The following lemma is used to show
that the only obstruction is the determinant representation of $\GL(E)$.
 

\begin{lemma}\label{lemma:eqvCondsMoreover}
Let $\lambda$ and $\mu$ be partitions and let $\ell$, $m \in \N$ be such that $\ell \ge \ell(\lambda)-1$
and $m \ge \ell(\mu) - 1$ and $\lambda \eqv{\ell}{m} \mu$.
Set $D = -\frac{\ell |\lambda|}{2} + \frac{m |\mu|}{2}$. Then $D \in \Z$
and 
\[ {\det}^D \otimes \nabla^\lambda \SSym^\ell \E \cong  \nabla^\mu \SSym^m \E \]
as representations of $\GL_2(\C)$.
\end{lemma}

\begin{proof}
By the `moreover' part in Theorem~\ref{thm:eqvConds}, $D = -\b(\lambda) + \b(\mu)$, and hence $D \in \Z$.
By~\eqref{eq:SchurSpec}, we have
$s_\lambda(1,q,\ldots, q^\ell) = \Phi_{\nabla^\lambda \SSym^\ell \E}(1,q)$. Therefore
\begin{equation}
\label{eq:eqvCondsMoreover1} 
q^D \Phi_{\nabla^\lambda \SSym^\ell \E}(1,q) = \Phi_{\nabla^\mu \SSym^m \E}(1,q). \end{equation}
Since representations of $\GL_2(\C)$ are completely reducible,
they are determined by their characters. Since $\Phi_{\det}(\alpha, \beta) = \alpha\beta$,
the $\GL_2(\C)$ representations
$\det^D \otimes \nabla^\lambda \SSym^\ell \E$ and
to $\nabla^{\mu}\SSym^{\m} \E$ are isomorphic if and only if
\[ (\alpha \beta)^D \Phi_{\nabla^\lambda \SSym^\ell \E}(\alpha, \beta) =  
\Phi_{\nabla^\mu \SSym^m \E}(\alpha, \beta) \]
for all $\alpha$, $\beta \in \C^\times$. Since $\Phi_{\nabla^\lambda \SSym^\ell \E}$ is homogeneous
of degree $\ell |\lambda|$ and $\Phi_{\nabla^\mu\SSym^m \E}$ is homogeneous of degree $m|\mu|$, 
this holds if and only if
\[ (\alpha \beta)^D \alpha^{\ell|\lambda|} \Phi_{\nabla^\lambda \SSym^\ell \E}(1, \beta/\alpha) 
= \alpha^{m |\mu|} \Phi_{\nabla^\mu \SSym^m \E}(1,\beta/\alpha).\]
for all $\alpha$, $\beta \in \C^\times$. Using~\eqref{eq:eqvCondsMoreover1} to rewrite
$\Phi_{\nabla^\mu \SSym^m \E}(1,\beta/\alpha)$ on the right-hand side we get
the equivalent condition that $(\alpha\beta)^D \alpha^{\ell |\lambda|} = 
\alpha^{m|\mu|} (\beta/\alpha)^D$ 
for all $\alpha$, $\beta \in \C^\times$. This holds by our choice of $D$.
\end{proof}

For $d \in \N_0$, let
$U_d = \{ \omega \in \C : \omega^d = 1\}$ and
let 
\[ \MGL{d}_2(\C) = \{ g \in \GL_2(\C) : \det g \in U_d\}. \] 
If $\SL_2(\C) \le G \le \GL_2(\C)$ then 
either $\{\det g : g \in G \}$ is one of the subgroups $U_d$
or it is dense (in the Zariski topology) on $\C^\times$. 
In the latter case, if $V$ and~$W$ are polynomial
representations of $\GL_2(\C)$ and $V \cong W$ as representations of $\SL_2(\C)$, then the isomorphism
extends to $G$ if and only if it extends to~$\GL_2(\C)$. 

\begin{proposition}\label{prop:eqvCondsMoreover}
Let $\lambda$ and $\mu$ be partitions and let $\ell$, $m \in \N$ be such that $\ell \ge \ell(\lambda)-1$
and $m \ge \ell(\mu) - 1$. Suppose that $\lambda\eqv{\ell}{m} \mu$.
Set $D = -\frac{\ell |\lambda|}{2} + \frac{m |\mu|}{2}$. Then $D \in \Z$ 
and $\nabla^\lambda \SSym^\ell \E$ is isomorphic to $\nabla^{\mu}\SSym^{\m} \E$
as representations of $\MGL{d}_2(\C)$
if and only if $d$ divides $D$.
\end{proposition}

\begin{proof}
By Lemma~\ref{lemma:eqvCondsMoreover}, we have the required isomorphism
if and only if ${\det}^D$ is the trivial representation of $\MGL{d}_2(\C)$.
This holds if and only if $U_d$ has exponent dividing $D$, so if and only if $d$ divides $D$.
\end{proof}

In particular, $\nabla^\lambda \SSym^\ell \E$ and $\nabla^\mu \SSym^m \E$ are isomorphic as representations
of $\GL_2(\C)$ if and only if they are isomorphic as representations of $\SL_2(\C)$ and
the degrees $\ell |\lambda|$ and $m |\mu|$ are equal.

It is worth noting that the proof of Lemma~\ref{lemma:eqvCondsMoreover} 
made essential use of the fact that the representations involved are homogeneous.
For example, if $V = \det \oplus \det^2$
and $W = E \otimes \det$ then $\Phi_V(1,q) = \Phi_W(1,q) = q+ q^2$.
But~$V$ and~$W$ are not isomorphic, even after restriction to $\SL_2(\C)$. 

\section{Basic properties of equivalence}
\label{sec:basicEqvs}

Given a non-empty partition $\lambda$ let $\overline{\lambda}$ be the partition
obtained by removing all columns of length $\ell(\lambda)$ from $\lambda$.

\begin{lemma}\label{lemma:columnRemoval}
If $\lambda$ is a partition then
either $\overline{\lambda}$ is empty or
$\lambda \eqv{\ell(\lambda)-1}{\ell(\lambda)-1} \overline{\lambda}$.
\end{lemma}

\begin{proof}
Let $\ell = \ell(\lambda) - 1$ and suppose that $\lambda$ has precisely
$c$ columns of length $\ell(\lambda)$.
We may suppose that $a(\lambda) > c$.
If~$t$ is a semistandard tableau of shape $\lambda$ with entries from $\{0,1,\ldots, \ell\}$
then the first $c$ columns of $t$ each  have entries $0$, $1$, \ldots, $\ell$ read from
top to bottom. Let $\overline{t}$ be the tableau obtained from $t$ by removing these columns.
Using the model for $\nabla^\lambda \SSym^\ell \E$ in Definition~\ref{defn:nablaSym},
we see that there is a linear isomorphism 
$\phi : \nabla^\lambda \SSym^\ell E \rightarrow \nabla^{\overline{\lambda}} \SSym^\ell E$ 
defined by $F(t) \mapsto F(\overline{t})$. Moreover, by a routine generalization
of Example~\ref{ex:oneColumn},
if $h \in \GL(\SSym^\ell \E)$ then
\[ \phi \bigl( h F(t) \bigr) = (\det h)^c h \phi \bigl( F(t) \bigr). \]
If $g \in \GL(E)$, each matrix coefficient of $g$ in its 
action on  $\SSym^\ell \E$ is a polynomial of degree $\ell$, and so the
determinant of $g$ acting on $\SSym^\ell \E$ has degree $\ell(\ell + 1)$.
We deduce that $\phi$ is an isomorphism
\[ \nabla^\lambda \SSym^\ell E \cong (\det E)^{c \ell(\ell+1)/2} \otimes \nabla^\lambda \SSym^\ell \E \]
of representations of $\GL(E)$.  Hence by Theorem~\ref{thm:eqvConds}(b), we have
$\lambda \eqv{\ell}{\ell} \overline{\lambda}$, as required.
\end{proof}

An alternative, but we believe less conceptual,
proof of Lemma~\ref{lemma:columnRemoval} can be
given by applying Theorem~\ref{thm:eqvConds}(g) to the bijection $t \mapsto \overline{t}$.

When $\ell\not=m$ the relation $\eqv{\ell}{m}$ is neither reflexive nor transitive.
The following lemma is the correct replacement for transitivity.

\begin{lemma}\label{lemma:transitive}
Let $k$, $\ell$, $m \in \N$ and let $\lambda$, $\mu$, $\nu$ be
partitions. If $\lambda \eqv{k}{\ell} \mu$ and \smash{$\mu\eqv{\ell}{m} \nu$} then
\smash{$\lambda \eqv{k}{m} \nu$}.
\end{lemma}

\begin{proof}
This is immediate from Theorem~\ref{thm:eqvConds}(b).
\end{proof}

\begin{proposition}\label{prop:prime}
Let $\lambda$ and $\mu$ be partitions. Then $\lambda \eqv{\ell(\lambda)-1}{\ell(\mu)-1} \mu$
if and only if $\overline{\lambda} \eqv{\ell(\lambda)-1}{\ell(\mu)-1} \overline{\mu}$.
\end{proposition}

\begin{proof}
This is immediate from Lemma~\ref{lemma:columnRemoval} and Lemma~\ref{lemma:transitive}.
\end{proof}

\begin{lemma}\label{lemma:removable}
Let $\lambda$ and $\mu$ be partitions and let $\ell$, $m \in \N$ be such that
$\ell \ge \ell(\lambda) -1$ and $m \ge \ell(\mu) - 1$.  Let $\lambda^\star$ and $\mu^\star$
be the partitions obtained from $\lambda$ and $\mu$ by removing
all columns of length $\ell+1$ and $m+1$, respectively. If $\lambda \eqv{\ell}{m} \mu$
then $\lambda^\star$ and $\mu^\star$ have the same number of removable boxes.
\end{lemma}

\begin{proof}
If $\ell = \ell(\lambda) - 1$ then $\lambda^\star = \overline{\lambda}$ 
and by Lemma~\ref{lemma:columnRemoval}
$\lambda \eqv{\ell}{\ell} {\lambda^\star}$. Otherwise $\ell > \ell(\lambda) - 1$
and $\lambda = \lambda^\star$. Hence $\lambda \eqv{\ell}{\ell} \lambda^\star$
and $\mu \eqv{m}{m} \mu^\star$.  By Lemma~\ref{lemma:transitive},
$\lambda^\star \eqv{\ell}{m} \mu^\star$. By Theorem~\ref{thm:eqvConds}(g)
and the final statement in this theorem, 
\[ \bigl| \Se{b(\lambda)+1}{\ell}{\lambda^\star} \bigr| = \bigl|\Se{b(\mu)+1}{m}{\mu^\star} \bigr|. \]
For each removable box in $\lambda^\star$, there is a
corresponding semistandard tableau of shape $\lambda^\star$ and
weight $b(\lambda^\star)+1$, obtained from the unique semistandard
tableau of shape $\lambda^\star$ and minimal weight $b(\lambda^\star)$ 
by increasing the entry in the removable box by $1$.
Conversely, every element of $\Se{\b(\lambda^\star)}{\ell}{\lambda^\star}$ arises
in this way.
A similar result holds for~$\mu^\star$. The displayed equation therefore implies
that the numbers of removable boxes are the same.
\end{proof}

\begin{lemma}\label{lemma:selfEquivalence}
Let $\lambda$ be a non-empty partition and let $\ell$, $m \in \N$
be such that $\ell, m \ge \ell(\lambda) - 1$. Then $\lambda \eqv{\ell}{m} \lambda$ 
if and only if $\ell = m$.
\end{lemma}

\begin{proof}
By Theorem~\ref{thm:eqvConds}(h) we have $\bigl( C(\lambda) + \ell + 1 \bigr) / H(\lambda)
= \bigl( C(\lambda) + m + 1 \bigr) / H(\lambda)$. Cancelling the equal sets of hook lengths
we have $C(\lambda) + \ell + 1 = C(\lambda) + m + 1$. Since $\lambda$ is non-empty
it follows that $\ell = m$.
\end{proof}

Recall that $a(\lambda)$ denotes the first part of a partition $\lambda$.

\begin{lemma}\label{lemma:aGreatest}
Let $\lambda$ be a partition and let $\ell \in \N$ be such that $\ell \ge \ell(\lambda)$.
The unique greatest element of $C(\lambda) + \ell + 1$ is $a(\lambda) + \ell$.
\end{lemma} 

\begin{proof}
The box $\bigl( 1,a(\lambda) \bigr)$ of $[\lambda]$ 
has the unique greatest content of any box in $\lambda$, namely of $a(\lambda) - 1$.
\end{proof}

Recall that a plethystic equivalence $\lambda \eqv{\ell}{m} \mu$ is \emph{prime} if
$\ell \ge \ell(\lambda)$ and $m \ge \ell(\mu)$.

\begin{proposition}\label{prop:aRelation}
Let $\lambda$ and $\mu$ be partitions.
If $\lambda \eqv{\ell}{m} \mu$ is a prime equivalence then $a(\lambda) + \ell = a(\mu) + m$.
\end{proposition}

\begin{proof}
By hypothesis, each hook length of $\lambda$ 
is at most $a(\lambda) + \ell(\lambda) - 1$. By Lemma~\ref{lemma:aGreatest},
the unique greatest element of $\bigl( C(\lambda) + \ell + 1 \bigr) / H(\lambda)$ is $a(\lambda) + \ell$.
The lemma now follows from Theorem~\ref{thm:eqvConds}(h). 
\end{proof}

\section{Conjugate partitions}
\label{sec:conjugates}

\subsection{Background}
The \emph{rank} of a non-empty partition $\lambda$, denoted $\rank(\lambda)$, 
is the maximum $r$ such that $\lambda_r \ge r$.
The \emph{Durfee square} of $\lambda$ is the subset $\{(i,j) : 1 \le i, j \le \rank(\lambda)\}$
of its Young diagram.
Theorem~\ref{thm:conjugates} also requires 
the following less standard definitions.

\begin{definition}\label{defn:southEastPartitions}
Let $\lambda$ be a partition and let $d = \rank(\lambda)$.
\begin{itemize}
\item[(i)] The \emph{south-rank} of $\lambda$, denoted $\srank(\lambda)$, 
is the maximum $j \in \N_0$ such that $\lambda_{d+j} = d$.
\item[(ii)] The \emph{south-partition} of $\lambda$, denoted $\spar(\lambda)$ is
$(\lambda_{d + \srank(\lambda) + 1}, \ldots, \lambda_{\ell(\lambda)})$.
\item[(iii)] The \emph{east-rank} of $\lambda$, denoted $\erank(\lambda)$, is $\srank(\lambda')$.
\item[(iv)] The \emph{east-partition} of $\lambda$, denoted $\epar(\lambda)$, is $\spar(\lambda')'$.
\end{itemize}
\end{definition}


\noindent These quantities are shown in Figure~1.
For example, the partition $\lambda = (8,6,5,3,3,1)$ shown in Figure~2
has
$\rank(\lambda) = 3$, $\erank(\lambda) = 2$, $\epar(\lambda) = (3,1)$,
$\srank(\lambda) = 2$ and $\spar(\lambda) = (1)$.

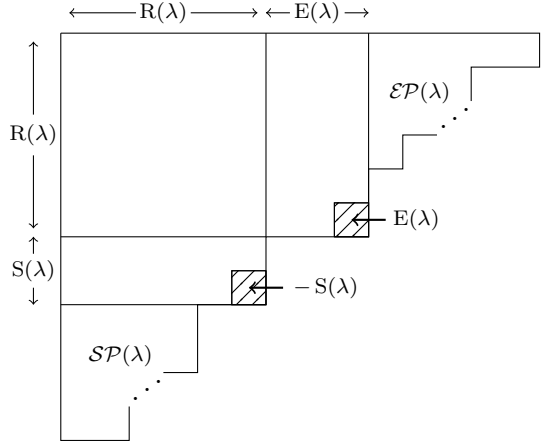
\begin{figure}[h]\label{figure:Durfee}

\begin{tikzpicture}[x=0.45cm,y=-0.45cm]
\renewcommand{\l}{-0.8}
\newcommand{\lh}{-0.6}
\newcommand{\h}{-0.6}
\newcommand{\e}{0.6}
\newcommand{\eh}{0.8}
\draw (0,0)--(14,0)--(14,1)--(12,1)--(12,2);
\draw (11,3)--(10,3)--(10,4)--(9,4)--(9,6)--(6,6)--(6,8)--(4,8)--(4,10)--(3,10);
\draw (0,0)--(0,12)--(2,12)--(2,11);
\draw(0,6)--(6,6); \draw(6,0)--(6,6);
\draw (9,0)--(9,6); \draw(0,8)--(6,8);
\node at (\l,3) {$\scriptstyle \rank(\lambda)$};
\draw[->] (\l,3-\e)--(\l,0.25); \draw[->] (\l,3+\e)--(\l,5.75);
\node at (3,\lh) {$\scriptstyle \rank(\lambda)$};
\draw[->] (3-\eh,\lh)--(0.25,\lh); \draw[->] (3+\eh,\lh)--(5.75,\lh);
\node at (7.5,\lh) {$\scriptstyle \erank(\lambda)$};
\draw[->] (7.5+\eh,\lh)--(9,\lh); \draw[->] (7.5-\eh,\lh)--(6,\lh);
\node at (\l,7) {$\scriptstyle \srank(\lambda)$};
\draw[->] (\l,7.5)--(\l,8); \draw[->] (\l,6.5)--(\l,6);

\draw (8,6)--(9,6)--(9,5)--(8,5)--(8,6);

\draw (6,7)--(6,8)--(5,8)--(5,7)--(6,7);

\draw[->, thick] (9.5,5.5)--(8.5,5.5);
\node at (10.4,5.5) {$\scriptstyle \erank(\lambda)$};

\draw[->, thick] (6.5,7.5)--(5.5,7.5);
\node at (7.75,7.5) {$\scriptstyle -\srank(\lambda)$};

\hatchbox{-8}{6}{6pt}
\hatchbox{-6}{9}{6pt}

\node at (10.5,1.75) {$\scriptstyle \epar(\lambda)$};
\node at (1.75,9.5) {$\scriptstyle \spar(\lambda)$};
\node at (2.5,10.25) {$\iddots$};
\node at (11.5,2.25) {$\iddots$};

\end{tikzpicture}

\caption{The statistics $\rank(\lambda)$, $\erank(\lambda)$, $\srank(\lambda)$
and partitions $\epar(\lambda)$, $\spar(\lambda)$. Two boxes have their content indicated. }

\end{figure}

We begin with three equivalent conditions for the existence of infinitely many plethystic
equivalences between distinct partitions $\lambda$ and $\mu$. 
We then prove a fourth equivalent condition, namely that
$\mu = \lambda'$ and $\epar(\lambda) = \spar(\lambda)'$,
obtaining Theorem~\ref{thm:KingStrongest} and, \emph{a fortiori}, Theorem~\ref{thm:conjugates}.

\begin{proposition}\label{prop:King}
Let $\lambda$ and $\mu$ be non-empty partitions. The following
are equivalent.
\begin{thmlist}
\item There exist infinitely many
pairs $(\ell,m)$ such that $\lambda \eqv{\ell}{m} \mu$.
\item There exists $l^\dagger \ge a(\lambda) + 2\bigl( \ell(\lambda) - 1\bigr)$ and
$m^\dagger \ge a(\mu) + 2\bigl( \ell(\mu) - 1\bigr)$
such that $\lambda \eqv{l^\dagger}{m^\dagger} \mu$.
\item $H(\lambda) = H(\mu)$ and there exists
 $\sh \in \Z$ such that $C(\lambda) + \sh = C(\mu)$.
\end{thmlist} 
\end{proposition}

\begin{proof} 
Suppose (i) holds. By Theorem~\ref{thm:eqvConds}(h), 
there exist arbitrarily large $\ell$ such that, for some $m$,
\begin{equation}
\label{eq:hcKing}
\bigl( C(\lambda) + \ell + 1 \bigr) / H(\lambda) =  \bigl( C(\mu) + m + 1 \bigr) / H(\mu). \end{equation}
When $\ell$ is very large $C(\lambda) + \ell + 1$ is disjoint
from $H(\lambda)$, and by Lemma~\ref{lemma:aGreatest}
the greatest element with non-zero multiplicity in the left-hand side
is \hbox{$a(\lambda) + \ell$}. Hence $m$ is also very large and (ii) holds.
If $\ell^\dagger$ and $m^\dagger$ satisfy~(ii) then, by~\eqref{eq:hcKing},
$\min (C(\lambda) + \ell^\dagger + 1) = -\ell(\lambda) + \ell^\dagger + 2
> a(\lambda) + \ell(\lambda) - 1 = \max H(\lambda)$,
and similarly $\min ( C(\mu) + m^\dagger + 1) > \max H(\mu)$.
Hence, the multisets $C(\lambda) + \ell^\dagger + 1$ and $H(\lambda)$ 
are disjoint, as are the multisets $C(\mu) + m^\dagger + 1$ and $H(\mu)$,
and so we have $H(\lambda) = H(\mu)$ and $C(\lambda) + \ell^\dagger + 1 = C(\mu) + m^\dagger + 1$.
Moreover, comparing minimum elements in~\eqref{eq:hcKing} we have
\[ -\ell(\lambda) + \ell^\dagger = -\ell(\mu) + m^\dagger. \]
Hence (iii) holds taking $\sh = \ell^\dagger - m^\dagger$.
Finally if (iii) holds, then (i) holds whenever $\ell - m = \sh$.
\end{proof}

We remark that the bound in (ii) is tight: for example, by the Hermite reciprocity
seen in~\S\ref{subsec:Hermite}, if $\lambda = (n)$ and $\mu = (n+1)$ where $n \in \N$, then
$\lambda\eqv{n+1}{n} \mu$ and $n+1 \ge a(\lambda) + 2\bigl( \ell(\lambda) -1\bigr) = n$.
As expected from (ii), $n \not\ge a(\mu) + 2\bigl( \ell(\mu) - 1) = n+1$.

Work of Craven~\cite{CravenHookLengths} shows that there is no
simple characterization of when $H(\lambda) = H(\mu)$. Fortunately
the second condition in (iii) is much more tractable.

\begin{figure}[t]
\begin{center}
\begin{tikzpicture}[x=0.575cm,y=-0.575cm]
\renewcommand{\ss}{\scriptstyle}
\foreach \j in {1,2,3,4,5,6,7,8} 
	\youngbox{5}{\j};
\foreach \j in {1,2,3,4,5}
	\foreach \i in {3,4,5}
		\fillbox{\i}{\j}{grey};	
\foreach \j in {1,2,3}
	\fillbox{2}{\j}{grey};	
\foreach \j in {1,2,3}
	\fillbox{1}{\j}{grey}; 
\foreach \j in {1,2,3,4,5,6} 
	\youngbox{4}{\j};
\foreach \j in {1,2,3,4,5} 
	\youngbox{3}{\j};
\foreach \j in {1,2,3} 
	\youngbox{2}{\j};
\foreach \j in {1,2,3} 
	\youngbox{1}{\j};
\youngbox{0}{1}

\younglabel{1}{5}{$\ss 0$}; \younglabel{2}{5}{$\ss1$}; \younglabel{3}{5}{$\ss2$};  \younglabel{4}{5}{$\ss3$}; 
\younglabel{5}{5}{$\ss4$};\younglabel{6}{5}{$\ss5$};\younglabel{7}{5}{$\ss6$};\younglabel{8}{5}{$\ss7$};
\younglabel{1}{4}{$\ss-1$}; \younglabel{2}{4}{$\ss0$}; \younglabel{3}{4}{$\ss1$};  \younglabel{4}{4}{$\ss2$}; 
\younglabel{5}{4}{$\ss3$};\younglabel{6}{4}{$\ss4$};
\younglabel{1}{3}{$\ss-2$}; \younglabel{2}{3}{$\ss-1$}; \younglabel{3}{3}{$\ss0$};  \younglabel{4}{3}{$\ss1$}; 
\younglabel{5}{3}{$\ss2$};

\younglabel{1}{2}{$\ss-3$}; \younglabel{2}{2}{$\ss-2$}; \younglabel{3}{2}{$\ss-1$};
\younglabel{1}{1}{$\ss-4$}; \younglabel{2}{1}{$\ss-3$}; \younglabel{3}{1}{$\ss-2$};
\younglabel{1}{0}{$\ss-5$};


\end{tikzpicture}\quad
\begin{tikzpicture}[x=0.575cm,y=-0.575cm]
\renewcommand{\ss}{\scriptstyle}
\foreach \j in {1,2,3,4,5,6,7,8,9} 
	\youngbox{5}{\j};
\foreach \j in {1,2,3,4,5}
	\foreach \i in {3,4,5}
		\fillbox{\i}{\j}{grey};	
\foreach \i in {3,4,5}
	\fillbox{\i}{6}{grey};			 
\foreach \j in {1,2,3,4,5,6,7} 
	\youngbox{4}{\j};
\foreach \j in {1,2,3,4,5,6} 
	\youngbox{3}{\j};
\foreach \j in {1,2,3} 
	\youngbox{2}{\j};
\foreach \j in {1,2,3}
	\fillbox{2}{\j}{grey};
\youngbox{1}{1};
\younglabel{1}{5}{$\ss 0$}; \younglabel{2}{5}{$\ss1$}; \younglabel{3}{5}{$\ss2$};  \younglabel{4}{5}{$\ss3$}; 
\younglabel{5}{5}{$\ss4$};\younglabel{6}{5}{$\ss5$};\younglabel{7}{5}{$\ss6$};\younglabel{8}{5}{$\ss7$};
\younglabel{9}{5}{$\ss8$};
\younglabel{1}{4}{$\ss-1$}; \younglabel{2}{4}{$\ss0$}; \younglabel{3}{4}{$\ss1$};  \younglabel{4}{4}{$\ss2$}; 
\younglabel{5}{4}{$\ss3$};\younglabel{6}{4}{$\ss4$};\younglabel{7}{4}{$\ss5$};
\younglabel{1}{3}{$\ss-2$}; \younglabel{2}{3}{$\ss-1$}; \younglabel{3}{3}{$\ss0$};  \younglabel{4}{3}{$\ss1$}; 
\younglabel{5}{3}{$\ss2$};\younglabel{6}{3}{$\ss3$};

\younglabel{1}{2}{$\ss-3$};\younglabel{2}{2}{$\ss-2$};\younglabel{3}{2}{$\ss-1$};
\younglabel{1}{1}{$\ss-4$};
\younglabel{1}{0}{$\rule{0pt}{9pt}$};
\end{tikzpicture}
\end{center}
\caption{The content of
the partition obtained from $\lambda$ by deleting a part of size $\rank(\lambda)$ and inserting it as a new column 
is $C(\lambda) + 1$ is obtained
by adding $1$ to the content of each box of $[\lambda]$; this can be seen here by comparing the shaded and unshaded boxes.}
\end{figure}
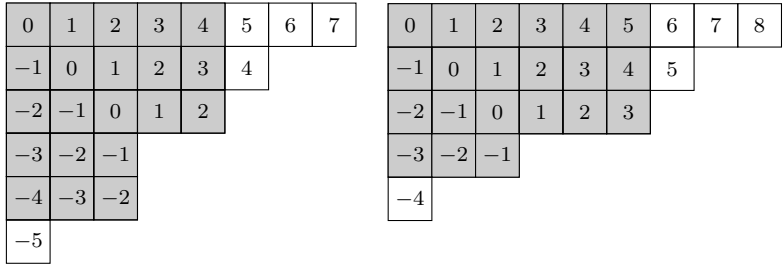

\begin{lemma}\label{lemma:contentShift}
Let $\lambda$ and $\mu$ be non-empty partitions and let $\sh \in \Z$. 
Then $C(\lambda) + \sh = C(\mu)$
if and only if $\rank(\lambda) = \rank(\mu)$, 
$\epar(\lambda) = \epar(\mu)$, $\spar(\lambda) = \spar(\mu)$ and
\[ \sh = -\erank(\lambda) + \erank(\mu) = \srank(\lambda) - \srank(\mu). \]
\end{lemma}

\begin{proof}
The `if' direction is implied by the special case when
$\erank(\mu) = \erank(\lambda) + 1$
and $\srank(\mu) = \srank(\lambda) - 1$. In this case
$\mu$ is obtained from $\lambda$ by
deleting the lowest of the $\srank(\lambda)$ parts of $\lambda$ of size $\rank(\lambda)$ 
and inserting $\rank(\lambda)$ boxes
as a new column at the right of the $\erank(\lambda)$ columns of $\lambda$ of size $\rank(\lambda)$.
We must show that $C(\mu) = C(\lambda) + 1$.
It is clear that adding $1$ to the content of the boxes $(i,j) \in [\lambda]$
with $i > \rank(\lambda) + \srank(\lambda) $ or $j > \rank(\lambda) + \erank(\lambda)$
gives the content of a corresponding box $(i-1,j)$ or $(i,j+1) \in [\mu]$. 
Moreover, as the shaded squares in Figure 2 show in a special case, 
adding $1$ to the content of each
remaining box in $[\lambda]$ gives the contents of the remaining boxes in $[\mu]$.

Conversely, suppose that $C(\lambda) + \sh = C(\mu)$.
It is clear that no member of $C(\lambda)$ can have multiplicity exceeding $\rank(\lambda)$.
As can be seen from the content of the two marked boxes in Figure~1,
the contents of multiplicity $\rank(\lambda)$ are precisely
$-\srank(\lambda),  \ldots, \erank(\lambda)$. 
Similarly in $C(\mu)$ the contents of maximum multiplicity are
$-\srank(\mu), \ldots, \erank(\mu)$, each with multiplicity $\rank(\mu)$. Therefore $\rank(\lambda) = \rank(\mu)$, $\erank(\lambda)
+ \sh = \erank(\mu)$ and $-\srank(\lambda) +\sh = -\srank(\mu)$.

The greatest element of $C(\lambda) + \sh$ is $\rank(\lambda) + \erank(\lambda) + 
a\bigl( \epar(\lambda)\bigr)- 1+ \sh$
and the greatest element of $C(\mu)$  is $\rank(\mu) + \erank(\mu) + a\bigl(\epar(\mu)\bigr) - 1$. Since $\rank(\lambda) = \rank(\mu)$ and
$\erank(\lambda) + \sh = \erank(\mu)$ it follows that $a\bigl(\epar(\lambda)\bigr) = a\bigl(\epar(\mu)\bigr)$.
Similarly comparing $C(\lambda) + \sh$ and $C(\mu)$ on 
their least elements shows that $\ell\bigl( \spar(\lambda) \bigr) = \ell\bigl( \spar(\mu) \bigr)$.
Let $\lambda^\star$ and $\mu^\star$ be the partitions obtained by
removing both the first row and column from $\lambda$ and $\mu$, respectively.
In each case this removes one box of each content between the least and greatest.
Therefore the hypothesis $C(\lambda) + \sh = C(\mu)$ implies that 
$C(\lambda^\star) + \sh = C(\mu^\star)$. If both sides are empty, we are done.
Otherwise, it follows by induction that $\epar(\lambda^\star) = 
\epar(\mu^\star)$
and $\spar(\lambda^\star) = \spar(\mu^\star)$, and hence
$\epar(\lambda) = \epar(\mu)$ and $\spar(\lambda) = \spar(\mu)$, as required.
\end{proof}

\begin{figure}[b]

\hspace*{-0.25in}\makebox[\textwidth]{
\begin{tikzpicture}[x=0.4cm,y=-0.4cm]
\renewcommand{\l}{-0.8}
\newcommand{\lh}{-0.6}
\newcommand{\h}{-0.6}
\newcommand{\e}{0.6}
\newcommand{\eh}{0.8}
\draw[very thick] (0,0)--(6,0)--(6,8)--(0,8)--(0,0);
\draw (0,0)--(14,0)--(14,1)--(12,1)--(12,2);
\draw (11,3)--(10,3)--(10,4)--(9,4)--(9,6)--(6,6)--(6,10)--(4,10)--(4,12)--(3,12);
\draw (0,0)--(0,14)--(2,14)--(2,13);
\draw(0,6)--(6,6); \draw(6,0)--(6,6);
\draw (9,0)--(9,6); \draw(0,8)--(6,8);
\draw (0,10)--(6,10);
\draw (6,0)--(6,10);
\node at (\l,3) {$\scriptstyle R$};
\draw[->] (\l,3-\e)--(\l,0.25); \draw[->] (\l,3+\e)--(\l,5.75);
\node at (3,\lh) {$\scriptstyle R$};
\draw[->] (3.3-\eh,\lh)--(0.25,\lh); \draw[->] (2.7+\eh,\lh)--(5.75,\lh);
\node at (7.5,\lh) {$\scriptstyle E$};
\draw[->] (7.2+\eh,\lh)--(9,\lh); \draw[->] (7.8-\eh,\lh)--(6,\lh);
\node at (\l,7) {$\scriptstyle \sh $};
\draw[->] (\l,7.4)--(\l,7.9); \draw[->] (\l,6.6)--(\l,6);
\node at (\l,9) {$\scriptstyle S$};
\draw[->] (\l,9.25-\e)--(\l,8.1);
\draw[->] (\l,8.75+\e)--(\l,10);

\node at (10.5,1.75) {$\scriptstyle \kappa$};
\node at (1.75,11.5) {$\scriptstyle \nu$};
\node at (2.5,12.25) {$\iddots$};
\node at (11.5,2.25) {$\iddots$};

\end{tikzpicture}\quad
\begin{tikzpicture}[x=0.4cm,y=-0.4cm]
\renewcommand{\l}{-0.8}
\newcommand{\lh}{-0.6}
\newcommand{\h}{-0.6}
\newcommand{\e}{0.6}
\newcommand{\eh}{0.8}
\draw[very thick] (0,0)--(8,0)--(8,6)--(0,6)--(0,0);
\draw (0,0)--(16,0)--(16,1)--(14,1)--(14,2);
\draw (13,3)--(12,3)--(12,4)--(11,4)--(11,6)--(6,6)--(6,8)--(4,8)--(4,10)--(3,10);
\draw (0,0)--(0,12)--(2,12)--(2,11);
\draw(0,6)--(6,6); \draw(6,0)--(6,6);
\draw (8,0)--(8,6); \draw(0,8)--(6,8);
\draw (11,0)--(11,6);
\node at (\l,3) {$\scriptstyle R$};
\draw[->] (\l,3-\e)--(\l,0.25); \draw[->] (\l,3+\e)--(\l,5.75);
\node at (3,\lh) {$\scriptstyle R$};
\draw[->] (3.3-\eh,\lh)--(0.25,\lh); \draw[->] (2.7+\eh,\lh)--(5.75,\lh);
\node at (7,\lh) {$\scriptstyle \sh$};
\draw[->] (6.5+\eh,\lh)--(7.9,\lh); \draw[->] (7.5-\eh,\lh)--(6,\lh);
\node at (9.5,\lh) {$\scriptstyle E$};
\draw[->] (9.2+\eh,\lh)--(11,\lh); \draw[->] (9.8-\eh,\lh)--(8.1,\lh);
\node at (\l,7) {$\scriptstyle S$};
\draw[->] (\l,7.4)--(\l,7.9); \draw[->] (\l,6.6)--(\l,6);

\node at (12.5,1.75) {$\scriptstyle \kappa$};
\node at (1.75,9.5) {$\scriptstyle \nu$};
\node at (2.5,10.25) {$\iddots$};
\node at (13.5,2.25) {$\iddots$};

\node at (0,13.5) {$\phantom{x}$};

\end{tikzpicture}}
\caption{The partitions $\lambda$ and $\mu$ in Lemma~\ref{lemma:hooksAfterContent}.}
\end{figure}
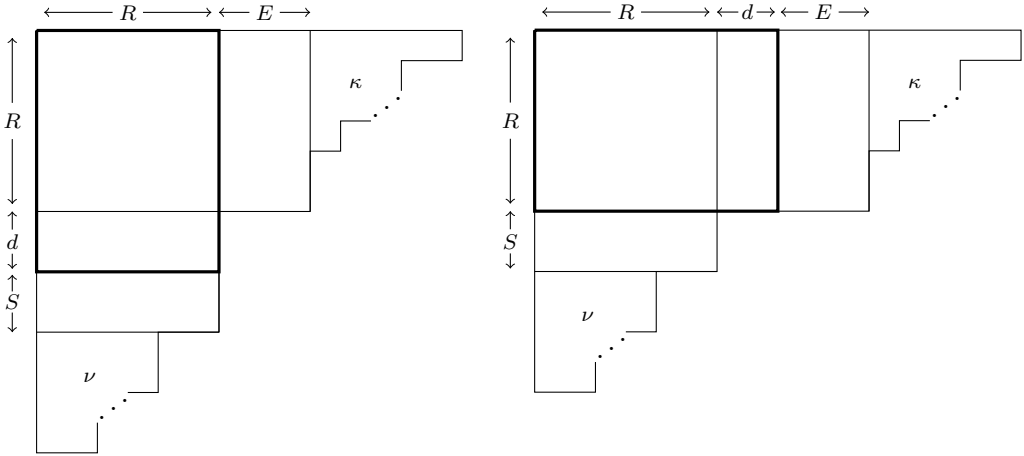

\begin{lemma}\label{lemma:hooksAfterContent}
Let $\lambda$ and $\mu$ be partitions such that
$\rank(\lambda) = \rank(\mu)$.
Suppose that for some $\sh \in \N$, we have
$\sh = -\erank(\lambda)  + \erank(\mu)  = \srank(\lambda) - \srank(\mu)$,
$\epar(\lambda) = \epar(\mu)$
and $\spar(\lambda) = \spar(\mu)$. Then $H(\lambda) = H(\mu)$
if and only if $\erank(\lambda) = \srank(\mu)$ and
$\spar(\lambda) = \epar(\lambda)'$.
\end{lemma}

\begin{proof}
Let $R = \rank(\lambda)$, $E = \erank(\lambda)$, $S = \srank(\mu)$,
$\kappa = \epar(\lambda)$ and $\nu = \spar(\lambda)$. 
Figure~3 shows the partitions $\lambda$ and $\mu$. 
Clearly if $E = S$ and $\kappa = \nu'$ then $\lambda' = \mu$ and
so $H(\lambda) = H(\mu)$. 
Conversely, suppose that $H(\lambda) = H(\mu)$. The hook lengths outside
the two thick rectangles in Figure~3 agree. 
If $(i,j)$ is a box in the Durfee square of $\lambda$ 
then 
\[ h_{(i,j)}(\lambda) = (R-j+E + \kappa_i) + (R-i + \sh + S + \nu'_j) + 1 
\]
where the parentheses indicate the arm and leg lengths. Similarly
if $(i,j)$ is in a box in the Durfee square of $\mu$ then
\[  h_{(i,j)}(\mu) = (R-j+\sh + E + \kappa_i ) + (R -i + S + \nu'_j) + 1. 
\]
Hence $h_{(i,j)}(\lambda) = h_{(i,j)}(\mu)$.
It remains to compare the hook lengths
\begin{align*}
h_{(R+i,j)}(\lambda) &= (R-j) + (S - i + \sh + \nu'_j) +1\\
h_{(i,R+j)}(\mu) &= (E-j + \sh + \kappa_{i}) + (R - i) + 1 \end{align*}
for $(R+i,j) \in \{R+1, \ldots, R+\sh\} \times \{1,\ldots, R\}$
and $(i,R+j) \in \{1,\ldots, R\} \times \{R+1,\ldots, R+\sh\}$.
Since $R > a(\nu)$ and $R > \ell(\kappa)$, 
the least such hook length for $\lambda$ is $h_{(R+\sh,R)}(\lambda) = S + 1$
and the least such hook length for $\mu$ is $h_{(R,R+\sh)}(\mu) = E + 1$.
Therefore $E = S$. 
Subtracting $R+S+\sh+1$ from the multisets of $\sh R$ hook lengths 
in the two
previous displayed equations, we see that
$H(\lambda) = H(\mu)$ if and only if there is an equality of multisets
\[ 
\begin{split}
\bigl\{ \nu'_j - i {}&{}- j:  (R+i,j) \in  \{R+1,\ldots, R+\sh\} \times \{1,\ldots, R\} \bigr\}
\\
&=
\bigl\{ \kappa_{i} - i - j: (i,R+j) \in \{1,\ldots, R\} \times \{R+1,\ldots, R+\sh \} \bigr\}.\end{split} \]
The unique greatest elements of each side are $\nu'_1 - 2$ and $\kappa_1 - 2$,
respectively.
Therefore $\nu'_1 = \kappa_1$. Cancelling 
the equal sets 
\begin{align*}
&\bigl\{ \nu'_1 - i - j : i \in \{1,\ldots, \sh\}, j \in \{1,\ldots, R\} \bigr\} \\
&\bigl\{ \kappa_1 - i - j : i \in \{1,\ldots, R\}, j \in \{1,\ldots, \sh\} \bigr\} \end{align*}
from each side we may repeat this argument inductively,
as in the proof of Lemma~\ref{lemma:contentShift},
to get $\nu' = \kappa$, as required.
\end{proof}

We are now ready to prove the slightly stronger version
of Theorem~\ref{thm:conjugates} stated below.

\begin{theorem}\label{thm:KingStrongest}
Let $\lambda$ and $\mu$ be distinct non-empty partitions. The following
are equivalent:
\begin{thmlist}
\item there exist infinitely many
pairs $(\ell,m)$ such that $\lambda \eqv{\ell}{m} \mu$;
\item $H(\lambda) = H(\mu)$ and there exists
 $\sh\in \Z$ such that $C(\lambda) + \sh= C(\mu)$;
\item $\lambda = \mu'$ and $\spar(\lambda) = \epar(\lambda)'$.
\end{thmlist} 
Moreover,
if any of these condition holds then $\lambda \eqv{\ell}{m} \mu$ if and only if
$\ell = \ell(\lambda)-1 + k$ and $m = \ell(\mu) - 1 + k$ for some $k \in \N_0$
and in \emph{(ii)} $\sh= \ell(\lambda) - \ell(\mu)$.
Finally, if $\lambda = \mu'$ but $\spar(\lambda) \not= \epar(\lambda')$  
then there are no plethystic equivalences between $\lambda$ and $\mu$.
\end{theorem}

\begin{proof}[Proof of Theorem~\ref{thm:KingStrongest}]
By Proposition~\ref{prop:King}, (i) and (ii) are equivalent.
Suppose that (ii) holds.
By swapping
$\lambda$ and $\mu$ if necessary, we may suppose that $\sh\in \N_0$.
Since $\lambda$ and $\mu$ are distinct and so $C(\lambda) \not= C(\mu)$, we have 
$\sh\in \N$. Now by Lemma~\ref{lemma:contentShift} 
followed by Lemma~\ref{lemma:hooksAfterContent} we get $\lambda' = \mu$
and $\spar(\lambda) = \epar(\lambda)'$,
as required. Conversely, these lemmas show that (iii) implies (ii) with $\sh= \ell(\lambda) -
\ell(\mu)$. For the `moreover' part, 
assuming~(ii), it follows from Theorem~\ref{thm:eqvConds}(h) that
$\lambda \eqv{\ell}{m} \mu$ whenever $\ell \ge \ell(\lambda) - 1$, $m \ge \ell(\mu) - 1$
and $\ell - m = \sh$, giving the claimed plethystic equivalences.
Conversely, suppose that $\lambda \eqv{\ell}{m} \mu$ and that (iii) holds. 
By Proposition~\ref{prop:aRelation} and~(iii), we have 
\[ \ell - m = a(\mu) -a(\lambda)  = \ell(\lambda) - \ell(\mu) \]
so $\ell = \ell(\lambda) -1 + k$ and $m = \ell(\mu) - 1 + k$ for some $k \in \N_0$, as required.
For the `finally part', the hypotheses imply that $H(\lambda) = H(\mu')$ so by Theorem~\ref{thm:eqvConds}(h),
if there is a plethystic equivalence then $C(\lambda) + d = C(\mu)$ for some $d \in \Z$. But this
contradicts Lemma~\ref{lemma:contentShift}.
\end{proof}

We end this section by remarking that, by Proposition~\ref{prop:eqvCondsMoreover},
a plethystic isomorphism 
$\nabla^\lambda \SSym^{\ell(\lambda) -1 +k} \E \cong
\nabla^{\lambda'} \SSym^{a(\lambda) -1 +k} \E$ 
given by Theorem~\ref{thm:KingStrongest}
extends to the overgroup $\MGL{d}_2(\C)$ of $\SL_2(\C)$ if and only
if $d$ divides $|\lambda|\bigl( a(\lambda) - \ell(\lambda) \bigr)/2$.
(Note this condition does not involve $k$.)
In particular, there is a $\GL_2(\C)$ isomorphism if and only
if $a(\lambda) = \ell(\lambda)$. But in this case, since $\epar(\lambda) = \spar(\lambda)'$,
we have $\erank(\lambda) = \srank(\lambda)$,
and so $\lambda = \lambda'$. 
We conclude that that there are infinitely many plethystic
isomorphisms of $\GL_2(\C)$-representations
$\nabla^\lambda \SSym^\ell \E \cong \nabla^\mu \SSym^m \E$ 
 if and only if $\lambda = \mu$.

\section{Multiple equivalences}
\label{sec:multipleEquivalences}

We need two lemmas on difference multisets.

\begin{lemma}\label{lemma:diff1} Let $X$ and $Y$ be finite multisubsets of $\Z$ and 
let $a$, $b$, $c \in \N_0$.
If $(X+a)/ X = (Y+b) / (Y+c)$ then \emph{either} $a=0$ and $b=c$ \emph{or} $a \not=0$, $b > c$,
$\max X + a = \max Y + b$ and $\min X = \min Y + c$.
\end{lemma}

\begin{proof}
Clearly $a=0$ if and only if $b=c$. Suppose neither is the case.
Since $a > 0$ the maximum element with non-zero multiplicity in the left-hand side is $\max X + a$.
Since it has positive multiplicity,
$\max X + a = \max Y + b$ and $b > c$. Similarly the minimum element in the left-hand side 
with non-zero multiplicity
is $\min X$, with negative multiplicity, and so $\min X  = \min Y + c$.
\end{proof}

\begin{lemma}\label{lemma:diff2} Let $Z$ and $W$ be finite multisubsets of $\Z$ and
let $t \in \Z$ be non-zero. If $Z/W = (Z+t)/(W+t)$ then $Z= W$.
\end{lemma}

\begin{proof}
Suppose, for a contradiction, that $Z \not= W$. Let $x$ to be greatest element with non-zero multiplicity
in $Z/W$. Clearly $x+t$ is the greatest element with non-zero multiplicity
in $(Z+t)/(W+t)$. But $Z/W = (Z+t) / (W+t)$ so $x = x+t$, hence $t=0$, a contradiction.
\end{proof}

\begin{proof}[Proof of Theorem~\ref{thm:multipleEquivalences}(i) and (ii)]
If $\ell = \ellp$ then from $\lambda \eqv{\ell}{m} \mu$ and $\lambda \eqv{\ellp}{\mp} \mu$ we get
$\mu \eqv{m}{\ell} \lambda$ and $\lambda \eqv{\ell}{\mp}{\mu}$, and so by Lemma~\ref{lemma:transitive},
$\mu \eqv{m}{\mp} \mu$. But now, by Lemma~\ref{lemma:selfEquivalence}, we have $m = \mp$,
contradicting that the pairs $(\ell,m)$ and $(\ellp, \mp)$ are distinct. Therefore
we may suppose that $\ell < \ellp$,
and in (ii) of the three
plethystic equivalences, two are prime. It therefore suffices to prove (i).

For (i), since $\ell \ge \ell(\lambda)$, Proposition~\ref{prop:aRelation} implies that
$a(\lambda) +\ell = a(\mu) + m$ and $a(\lambda) + \elld = a(\mu) + \mp$.
Let $t = \elld - \ell = \mp - m \in \N$ denote the common difference.
By Theorem~\ref{thm:eqvConds}(h), 
we have equalities of difference multisets
\begin{align} \bigl( C(\lambda) + \ell + 1 \bigr) / H(\lambda) &=
\bigl( C(\mu) + m +1 \bigr)/ H(\mu) \label{eq:multipleFirst} 
\\ 
\bigl( C(\lambda) + \ell + t + 1\bigr) / H(\lambda)  &= \bigl(  C(\mu ) + m + t + 1 \bigr) \notag/ H(\mu).
\end{align}
Hence 
\[ \bigl( C(\lambda) + \ell + 1 \bigr) / 
\bigl( C(\mu) + m+ 1 \bigr) 
= \bigl( C(\lambda) + \ell + t + 1 \bigr) / \bigl( C(\mu) + m + t + 1 \bigr).\] 
By Lemma~\ref{lemma:diff2}, we deduce that $C(\lambda) + \ell + 1 = C(\mu) + m +1$.
Writing $Z$ for this multiset, \eqref{eq:multipleFirst} can be restated as
$Z / H(\lambda) = Z / H(\mu)$, which 
implies that $H(\lambda) = H(\mu)$. 
Therefore either $\lambda = \mu$, or
the hypotheses for Theorem~\ref{thm:KingStrongest}(ii) hold,
and we may conclude that $\mu = \lambda'$ and $\spar(\lambda) = \epar(\lambda)'$. 
\end{proof}

\begin{proof}[Proof of Theorem~\ref{thm:multipleEquivalences}(iii)]
By Theorem~\ref{thm:eqvConds}(h) the hypotheses imply
\begin{align*} 
\bigl( C(\lambda) + n + 1 \bigr) / H(\lambda) &= \bigl( C(\mu) + n + 1\bigr) / H(\mu) \\
\bigl( C(\lambda) + \np + 1 \bigr) / H(\lambda) &= \bigl( C(\mu) + \np + 1\bigr) / H(\mu).
\end{align*}
Hence
\[ \bigl( C(\lambda) + n+1 \bigr) \cup \bigl( C(\mu) + \np +1 \bigr) = 
\bigl( C(\mu) + n+1 \bigr) \cup \bigl( C(\lambda) + \np + 1 \bigr). \]
Subtracting $n+1$ from every element of these multisets we obtain
\[
C(\lambda) / C(\mu) = \bigl( C(\lambda) + \np - n \bigr) / \bigl(C (\mu) + \np - n\bigr).\]
By Lemma~\ref{lemma:diff2} applied with $Z = C(\lambda)$, $W = C(\mu)$ and $t = \np - n$
we have $C(\lambda) = C(\mu)$. Therefore $\lambda = \mu$, as required. 
\end{proof}



\section{Complementary partitions}
Let $\lambda$ be a partition such that $\ell(\lambda) \le r$.
Recall that $\lambda^{\circ r}$ denotes
the complementary partition to $\lambda$ in the $r \times a(\lambda)$ box.
Equivalently, 
$\lambda^{\circ r}_i = a(\lambda) - \lambda_{r+1-i}$
for each $i \in \{1,\ldots, r\}$.
In~\S\ref{subsec:tableaux} we defined
the set $\SSYT_{\le \ell}(\lambda)$ of semistandard
$\lambda$-tableaux with entries in $\{0,1,\ldots, \ell\}$.
We extend this notation in the obvious way to define $\SSYT_{< r}(\lambda)$.
Given $t \in \SSYT_{< r}(\lambda)$, let $t^{\circ r}$ be the
unique column standard $\lambda^{\circ r}$-tableau $t^{\circ r}$ having
as its entries in column $j$ the complement in $\{0, 1,\ldots, r-1\}$ of the entries of 
$t$ in column $a(\lambda)+1-j$.
For example if 
$\lambda = (3,2,2,1)$ then $\lambda^{\circ 5} = (3,2,1,1)$ and under the map $t \mapsto t^{\circ 5}$
we have
\[ \young(000,11,22,3) \longmapsto\; \young(134,24,3,4). \]
The following proposition is
implicit in \cite[\S 4]{KingSU2Plethysms}.

\begin{proposition}\label{prop:KingBijection}
The map $t \mapsto t^{\circ r}$ is a self-inverse bijection
\[ \SSYT_{< r}(\lambda) \rightarrow \SSYT_{< r}(\lambda^{\circ r}). \]
\end{proposition}

\begin{proof}
The only non-obvious claim is that $t^{\circ r}$ is semistandard.
Suppose, for a contradiction, that columns
$a(\lambda)-j-1$ and $a(\lambda)-j$ of 
$t^{\circ r}$ have entries 
$m_1^\circ \le k_1^\circ$, \ldots, $m_{i-1}^\circ \le k_{i-1}^\circ$ and
$m_i^\circ > k_i^\circ$ when read from top to bottom. 
Let columns $j$ and $j+1$ of $t$ read
from top to bottom have entries $k_1 \le m_1$, \ldots, $k_h \le m_h$ where~$h$ is greatest 
such that $m_h < m_i^\circ$.
Then $\{m_1^\circ, \ldots, m_{i-1}^\circ, m_1, \ldots, m_h\}$ are 
all the numbers strictly less than $m_i^\circ$ in $\{0,1,\ldots, r-1\}$, 
since, by choice of $h$, if $m_{h+1}$ is defined then $m_{h+1} > m_i^\circ$.
But from the chain $m_i^\circ > k_i^\circ > \ldots > k_1^\circ$ and the inequalities 
$m_i^\circ > m_h \ge m_j \ge k_j$
for $j \in \{1,\ldots, h\}$, we see that $m_i^\circ$ is strictly greater than 
$i+h$ distinct numbers, a contradiction.
\end{proof}

\begin{proof}[Proof of Theorem~\ref{thm:complements}]
For the `if' direction we give a slightly
simplified version of the argument in \cite{KingSU2Plethysms}.
By construction of $t^{\circ r}$ we have $|t| + |t^{\circ r}| = a(\lambda) r(r-1)/2$.
Therefore Proposition~\ref{prop:KingBijection}
implies that $|\Se{e}{r-1}{\lambda^{\circ r}}| = |\Se{ar(r-1)/2 - e}{r-1}{\lambda}|$ for each $e \in \N_0$.
Recall from~\eqref{eq:Se} that $s_\lambda(1,q,\ldots, q^\ell) = \sum_{e\in \N_0} |\Se{e}{\ell}{\lambda}|q^e$.
Thus the coefficient of $q^e$ in $s_{\lambda^{\circ r}}(1,q,\ldots, q^{r-1})$
agrees with the coefficient of $q^{a(\lambda) r(r-1)/2 -e}$ in $s_{\lambda}(1,q,\ldots, q^{r-1})$, and so
\[ s_{\lambda^{\circ r}}(1,q,\ldots, q^{r-1}) = q^{a(\lambda)r(r-1)/2} s_\lambda(1,q^{-1},\ldots, q^{-(r-1)}). \]
By the central symmetry in the `moreover' part of Theorem~\ref{thm:eqvConds}, 
\[ q^{-(r-1)|\lambda|/2} s_\lambda(1,q,\ldots, q^{r-1}) = q^{(r-1)|\lambda|/2}s_\lambda(1,q^{-1},
\ldots, q^{-(r-1)}). \]
Hence
\[ s_{\lambda^{\circ r}}(1,q,\ldots, q^{r-1}) = q^{(r-1)(a(\lambda)r/2 - |\lambda|)} 
s_\lambda(1,q,\ldots, q^{r-1}).\] 
By Theorem~\ref{thm:eqvConds}(e), we have $\lambda^{\circ r} \eqv{r-1}{r-1} \lambda$, as required.

For the `only if' direction suppose that $\lambda \neq \lambda^{\circ r}$ and $\lambda \eqv{\ell}{\ell} \lambda^{\circ r}$.
From the `if' direction, we have  $\lambda \eqv{r-1}{r-1} \lambda^{\circ r}$.
By Theorem~\ref{thm:multipleEquivalences}(ii) we deduce that $\ell = r-1$, as required.
\end{proof}

\label{sec:complements}

\section{Rectangular equivalences and $q$-binomial identities}
\label{sec:rectangles}

In this section we determine all plethystic equivalences $\lambda \eqv{\ell}{m} \mu$
in which one or both of $\lambda$ and $\mu$ is a \emph{rectangle}, of the form $(a^b)$ with
$a$, $b \in \N$. Our main result is as follows.

\begin{theorem}\label{thm:rectanglesFull}
Let $\lambda$ be a partition and let $a$, $b$, $c \in \N$. 
Then $\lambda \eqv{\ell}{b+c-1} (a^b)$ if and only if
$\lambda$ is obtained by adding columns of length $\ell+1$ to a rectangle
$(\ap^\bp)$ with $\bp \le \ell$ and $(\ap, \bp, \ell-\bp+1)$ is a permutation of $(a,b,c)$.
\end{theorem}

Clearly this implies Theorem~\ref{thm:rectangles}. Conversely,
as seen in Example~\ref{ex:ex}, by using
Lemma~\ref{lemma:columnRemoval} and Lemma~\ref{lemma:transitive} 
one may reduce to the case when $\ell \ge \ell(\lambda)$ 
of a prime plethystic equivalence. 
Therefore Theorem~\ref{thm:rectanglesFull} follows routinely
from Theorem~\ref{thm:rectangles}.

In the following subsection we 
use Theorem~\ref{thm:eqvConds}(e) to show that the
`if' direction of Theorem~\ref{thm:rectangles} is the representation-theoretic realization of the six-fold
symmetry group of plane partitions.
Next we prove a new determinantal
formula using $q$-binomial coefficients 
of MacMahon's generating function of plane partitions.
We then prove the `only if' direction
of Theorem~\ref{thm:rectangles} using certain unimodal graphs to keep track
of the contents of rectangles.
The section ends
with the corollary for the case $b=1$; this generalizes the Hermite reciprocity
seen in~\S\ref{subsec:Hermite}.

\subsection{Plane partitions}\label{subsec:rectanglesIf}

Recall that a \emph{plane partition}
of shape $\lambda$ is a $\lambda$-tableau with entries from $\N$ whose
rows and columns are weakly decreasing, when read 
left to right and top to bottom.
Let $\PP(a,b,c)$ denote the set of plane partitions~$\pi$ with
entries in $\{1,\ldots, c\}$ whose shape $\shape(\pi)$ is contained in $[(a^b)]$.
Assigning~$0$ to each box of $[(a^b)] \, \backslash \shape(\pi)$
defines a bijection between $\PP(a,b,c)$
and the set of $(a^b)$-tableaux with
entries from $\N_0$ and weakly decreasing rows and columns.
Observe that if $t$ is such a tableau then rotating
$t$ by a half-turn and adding $j-1$ to 
every entry in row $j$ gives a semistandard tableau of shape $(a^b)$
with entries in $\{0,1, \ldots, b+c-1\}$. Again this map is bijective. Hence we have
\begin{equation}
\label{eq:SSYTtoPP}
q^{-a\binom{b}{2}} s_{(a^b)}(1,q,\ldots,q^{b+c-1}) = \sum_{\pi \in \PP(a,b,c)} q^{|\pi|} 
\end{equation}
where, extending our usual notation, $|\pi|$ denote the sum of entries of a plane partition $\pi$.

\begin{proof}[Proof of `if' direction of Theorem~\ref{thm:rectangles}]
Representing elements of $\PP(a,b,c,)$ plane partitions 
by three-dimensional Young diagrams contained
in the $a \times b \times c$ cuboid, 
it is clear that the right-hand side of~\eqref{eq:SSYTtoPP} 
is invariant under permutation of $a,b,c$.
The `if' direction now follows from Theorem~\ref{thm:eqvConds}(e).
\end{proof}

\subsection{MacMahon's identity}

In \cite[page 659]{MacMahon1896}, MacMahon proved that 
\begin{equation}
\label{eq:MacMahon} \sum_{\pi \in \PP(a,b,c)} q^{|\pi|} = \prod_{i=1}^a\prod_{j=1}^b \prod_{k=1}^c
\frac{q^{i+j+k-1}-1}{q^{i+j+k-2}-1}. \end{equation}
This makes the invariance of~\eqref{eq:SSYTtoPP} under
permutation of $a$, $b$ and $c$ algebraically obvious. 
For a modern proof using~\eqref{eq:SSYTtoPP} and Stanley's Hook Content Formula
see (7.109) and (7.111) in \cite{StanleyII}. 
In this section we prove Corollary~\ref{cor:det}, which gives
a new $q$-binomial form for the right-hand side of MacMahon's formula.
Specializing $q$ to $1$ in this corollary we obtain the attractive identity 
\[ \prod_{i=1}^a\prod_{j=1}^b \prod_{k=1}^c \frac{i+j+k-1}{i+j+k-2} =
\det\binom{b+c+i+j}{b+j}_{0 \le i, j \le a-1}. \]
Proving the invariance of the right-hand side
under permutation of $a$, $b$ and~$c$ was
asked as a MathOverflow question by T. Amdeberhan\footnote{See
\url{mathoverflow.net/q/322894/7709}.} in 2019.

\subsubsection*{Hermite reciprocity and $q$-binomial coefficients}
Recall from~\eqref{eq:quantumInteger} that $[m]_q = (q^m-1)/(q-1) \in \C[q]$ for $m \in \N_0$.
Set $[m]_q = 0$ if $m < 0$.
For $m$, $\ell \in \N_0$, we define the $q$-binomial coefficient $\qbinom{m}{\ell} \in \C[q]$ by
\[ \qbinom{m}{\ell} = \frac{[m]_q \ldots [m-\ell+1]_q}{[\ell]_q \ldots [1]_q}. \]
As motivation, we note that, by Stanley's Hook Content Formula (Theorem~\ref{thm:HCF}), we have $s_{(n)}(1,q,\ldots, q^\ell) = \qbinom{n+\ell}{\ell}$.
We saw in the first proof of Hermite reciprocity 
in \S\ref{subsec:Hermite} that $s_{(n)}(1,q,\ldots, q^\ell)$ is the
generating function for partitions contained in the $n \times \ell$ box. Thus
the well-known invariance of $\qbinom{n+\ell}{\ell}$
under swapping $n$ and $\ell$ is  equivalent to Hermite reciprocity.

\subsubsection*{Jacobi--Trudi}
Let $e_m$ be the elementary symmetric function of degree $m$.
By \cite[Proposition 7.8.3]{StanleyII} 
we have $e_\ell(1,q,\ldots, q^{m-1}) = q^{\binom{\ell}{2}} \qbinom{m}{\ell}$.
It will be convenient to denote the right-hand side by $\sqbinom{m}{\ell}$.
Using this notation, 
the dual Jacobi--Trudi formula (see \cite[Corollary 7.16.2]{StanleyII}) implies that
\begin{equation}
\label{eq:firstPP} s_{(a^b)}(1,q,\ldots,q^{b+c-1}) = \det \Bigl(\, 
\sqbinom{b+c}{b+j-i}\, \Bigr)_{0 \le i, j \le a-1}.\end{equation}

\subsubsection*{A determinantal form of MacMahon's identity}
We now apply row and column operations to the matrix in~\eqref{eq:firstPP} using the following two 
versions of the Chu--Vandermonde identity for our scaled $q$-binomial coefficients. To make the article
self-contained we include
bijective proofs
using that $\sqbinom{m}{\ell}$ is
the generating function enumerating $\ell$-subsets of 
$\{0,\ldots, m-1\}$ by their sum of entries. (This easily follows
from \cite[Proposition 1.7.3]{StanleyI}.)
A different proof of~\eqref{eq:vdmCols} using the $q$-binomial theorem
is given in the solution to Exercise~100 in \cite{StanleyI}.

\newcommand{\ml}{\ell}
\newcommand{\nm}{m}
\begin{lemma}\label{lemma:vdm} We have
\begin{align}
\sqbinom{\nm+r}{\ml+r} &= \sum_\jj q^{\nm\jj} \sqbinom{r}{\jj}\, \sqbinom{\nm}{\ml+r-\jj}, \label{eq:vdmCols}\\
 \sqbinom{\nm+r}{\ml+r} &= \sum_\jj q^{r(\ml+r-\jj)}\sqbinom{r}{\jj}\, \sqbinom{\nm}{\ml+r-\jj}. \label{eq:vdmRows} 
\end{align}
\end{lemma}

\begin{proof}
For~\eqref{eq:vdmCols},
observe that a $(\ml+r)$-subset of $\{0,1,\ldots, \nm+r-1\}$ containing exactly $\jj$
elements of  $\{\nm,\ldots, \nm+r-1\}$ has a unique decomposition as $Y \cup Z$
where $Y$ is a $\jj$-subset of $\{m,\ldots, \nm+r-1\}$
and $Z$ is an $(\ell+r-\jj)$-subset of $\{0,1,\ldots, \nm-1\}$. These pairs
are enumerated, according to their sum of entries,
by $q^{\nm\jj} \sqbinom{r}{\jj}$ and $\sqbinom{\nm}{\ml+r-\jj}$, respectively.
Identity~\eqref{eq:vdmRows} can be proved similarly by splitting  the subset as $Y' \cup Z'$
where $Y'$ is a $\jj$-subset of $\{0,\ldots, r-1\}$ and $Z'$ is an $(\ml+r-\jj)$-subset
of $\{r,\ldots, \nm+r-1\}$.
\end{proof}

\begin{proposition}\label{prop:dets}
For any $a, b, c \in \N$ we have
\begin{align*} 
 q^{a \binom{b}{2}} \prod_{i=1}^a\prod_{j=1}^b \prod_{k=1}^c
\frac{q^{i+j+k-1}-1}{q^{i+j+k-2}-1} 
&= \det \Bigl( \,\sqbinom{b+c+j}{b+j-i}\, \Bigr)
\\
&= q^{-A} \det \Bigl( \,q^{-ij} \sqbinom{b+c+i+j}{b+j}\, \Bigr)
\end{align*}
where each determinant is of an $a \times a$ matrix with entries
defined by taking $0 \le i,j \le a-1$, and $A = \binom{a}{2}b - \binom{a+1}{3}$.
\end{proposition}

\begin{proof}
Let $M$ be the matrix with entries $\sqbinom{b+c}{b+j-i}$ for $0 \le i,j \le a-1$
appearing in~\eqref{eq:firstPP}.
Let $C_j$ denote the $j$th column of $M$, where columns are numbered from $0$ up to $a-1$.
Let $M'$ be the matrix obtained from $M$ by replacing $C_j$ with the linear combination
\[ \sum_{j'=0}^j q^{(b+c)(j-j')} \sqbinom{j}{j-j'} C_{j'} \]
for each $j \in \{0,\ldots, a-1\}$. Since $\sqbinom{j}{0} = 1$, we have $\det M' = \det M$.
By~\eqref{eq:vdmCols}, taking $\nm=b+c$, $\ml = b-i$ and $r = j$ and replacing
the summation variable~$k$ with $j-j'$, we have
\[ \sqbinom{b+c+j}{b+j-i} = \sum_{j'} q^{(b+c)(j-j')} \sqbinom{j}{j-j'}\, 
\sqbinom{b+c}{b+j'-i}. \]
Therefore $M'$ has entries $\sqbinom{b+c+j}{b+j-i}$.
as required for the first equality.
Let $R'_i$ denote the $i$th row of $M'$. Let $M''$ be the matrix obtained from $M'$ by
replacing $R'_i$ with the linear combination
\[ \sum_{i'=0}^i q^{i(b-i')} \sqbinom{i}{i'} R'_{i'} \]
for each $i \in \{0,\ldots,a-1\}$. Since $q^{i(b-i)}\sqbinom{i}{i} = q^{i(b-i) + \binom{i}{2}}$
and 
\[ \sum_{i=0}^{a-1} \Bigl( i(b-i) + \binom{i}{2} \Bigr)
= \sum_{i=0}^{a-1} ib - \sum_{i=0}^{a-1} \binom{i+1}{2} = b\binom{a}{2} - \binom{a+1}{3} \]
we have $\det M'' = q^{\binom{a}{2}b - \binom{a+1}{3}} \det M'$.
By~\eqref{eq:vdmRows} taking $\nm=b+c+j$, $\ml = b+j-i$ and $r = i$ and replacing
the summation variable $k$ with $i'$ we have
\[ \sqbinom{b+c+i+j}{b+j} = \sum_{j'} q^{i(b+j-i')} \sqbinom{i}{i'}\, 
\sqbinom{b+c+j}{b+j-i'}. \]
Multiplying through by $q^{-ij}$, we see
that $M''$ has entries $q^{-ij}\sqbinom{b+c+i_j}{b+j}$. The final equality follows.
\end{proof}


\begin{corollary}\label{cor:det}
For any $a,b,c \in \N$ we have
\[ \prod_{i=1}^a\prod_{j=1}^b \prod_{k=1}^c
\frac{q^{i+j+k-1}-1}{q^{i+j+k-2}-1} = q^{1^2+ \cdots + (a-1)^2} 
 \det \Bigl( q^{-ij} \qbinom{b+c+i+j}{b+j} \Bigr)_{0 \le i,j < a}. \]
\end{corollary}

\begin{proof}
Let $N$ be the $a \times a$ matrix on the right-hand side.
Since $\smash{\sqbinom{b+c+i+j}{b+j}} = q^{\binom{b+j}{2}}\qbinom{b+c+i+j}{b+j}$
and \smash{$\sum_{j=0}^{a-1} \binom{b+j}{2} = \binom{a+b}{3} - \binom{b}{3}$},
it follows from the second equality in Proposition~\ref{prop:dets} 
that 
\[ q^{a\binom{b}{2}} \prod_{i=1}^a\prod_{j=1}^b \prod_{k=1}^c
\frac{q^{i+j+k-1}-1}{q^{i+j+k-2}-1} = q^{-\binom{a}{2}b + \binom{a+1}{3} +
\binom{a+b}{3} - \binom{b}{3}} \det N. \]
By direct calculation one finds that
\[ -a \binom{b}{2} - \binom{a}{2} b + \binom{a+1}{3} + \binom{a+b}{3} - \binom{b}{3}
= \frac{a(a-1)(2a-1)}{6}. \]
The identity now follows using $1^2 + \cdots + (a-1)^2 = a(a-\mfrac{1}{2})(a-1)/3$.
\end{proof}

\subsection{Plethystic equivalences between rectangles}
\renewcommand{\c}[2]{\mathrm{c}_{#1}^{(#2)}\hskip-0.5pt}
\newcommand{\h}[1]{\mathrm{h}_{#1}}
\newcommand{\ch}[2]{\mathrm{ch}_{#1}^{(#2)}}

In this subsection we prove the `only if' direction of Theorem~\ref{thm:rectangles}.
The following lemma gives one useful reduction.

\begin{lemma}\label{lemma:complementReduction}
Let $a$, $b \in \N$ and let $\mu$ be a partition. If $\sh \ge b$ and $m \ge \ell(\mu)-1$
then $(a^b) \eqv{\sh-1}{m} \mu$ if and only if $(a^{\sh-b}) \eqv{\sh-1}{m} \mu$.
\end{lemma}

\begin{proof}
Since $(a^{\sh-b})$ is the complement of $(a^b)$ in the $\sh \times a$ box,
we have $(a^b)\eqv{\sh-1}{\sh-1} (a^{\sh-b}) $ by 
Theorem~\ref{thm:complements}. Now apply Lemma~\ref{lemma:transitive}.
\end{proof}

As seen here, it is most convenient to work with the
shift applied to the content multiset: thus $\sh = \ell+1$ in the usual notation.
Recall from Definition~\ref{defn:hlContent} that $h_{(i,j)}(\lambda)$ denotes the hook length
of the box $(i,j) \in [\lambda]$.

\begin{definition}
Let $\lambda$ be a non-empty partition and 
let $\sh \ge \ell(\lambda)$. Define $\c{\lambda}{\sh} : \N_0 \rightarrow \N_0$ and $\h{\lambda} : \N_0 \rightarrow \N_0$ by
\begin{align*}
\c{\lambda}{\sh}(k) &= \bigl| \{ (i,j) \in [\lambda] : j-i+\sh = k \} \bigr|  \\
\h{\lambda}(k) &= \bigl| \{ (i,j) \in [\lambda] : h_{(i,j)} = k \}  \bigr|
\end{align*}
and the \emph{content-hook function} $\ch{\lambda}{\sh} : \N_0 \rightarrow \Z$ by 
$\ch{\lambda}{\sh} = \c{\lambda}{\sh} - \h{\lambda}$.
\end{definition}

By Theorem~\ref{thm:eqvConds}(h), we have 
\begin{equation} (a^b) \eqv{\sh-1}{\shp-1} ({a'}^{b'}) \iff
\ch{(a^b)}{\sh-1} = \ch{(\ap^\bp)}{\shp - 1}. \label{eq:chiff} \end{equation}

The equalities on the right-hand side of~\eqref{eq:chiff}
can easily be classified
from the graphs of the content-hook functions. 
We include full details to save the
reader case-by-case checking. 
As a visual guide, in inequalities
and graphs we 
write $x$-coordinates relevant to the content in bold.

\renewcommand{\blue}{}
\newcommand{\blues}{}
\renewcommand{\red}{\bf}
\newcommand{\reds}{\bf\scriptstyle}
\newcommand{\redsv}{}

\begin{lemma}\label{lemma:chgraph}
Let $b \le \sh$. If $b \le a$ then the graphs of $\c{(a^b)}{\sh}$ and $-\h{(a^b)}$ are

\begin{center}
\begin{tikzpicture}[x=\gscale cm,y=\gscale cm]
\draw[->](0,-0.5)--(0,2.5); \draw[->](-0.5,0)--(9.5,0);
\draw(3,0)--(5,2)--(7,2)--(9,0);
\bigdot{(3,0)}\htick{3}
\bigdot{(5,2)}\htick{5}\vtick{2}
\bigdot{(7,2)}\htick{7}
\bigdot{(9,0)}\htick{9}
\node[left] at (-\vaxessep,2) {$\redsv b$};
\node at (3,-\axessep) {$\reds -b+\sh$};
\node at (5,-\axessep) {$\reds \sh$};
\node at (7,-\axessep) {$\reds a-b+\sh$};
\node at (9,-\axessep) {$\reds a+\sh$};
\end{tikzpicture}

\begin{tikzpicture}[x=\gscale cm,y=\gscale cm]
\draw[->](0,-2.5)--(0,0.5); \draw[->](-0.5,0)--(9.5,0);
\draw(0,0)--(2,-2)--(4,-2)--(6,0);
\bigdot{(0,0)}\htick{0}
\bigdot{(2,-2)}\htick{2}\vtick{-2}
\bigdot{(4,-2)}\htick{4}
\bigdot{(6,0)}\htick{6}
\node[left] at (-\vaxessep,-2) {$\blues -b$};
\node at (2,\axessep) {$\blues b$};
\node at (4,\axessep) {$\blues a$};
\node at (6,\axessep) {$\blues a+b$};
\end{tikzpicture}\hspace*{12pt}
\end{center}
If $b \ge a$ then the graphs of $\c{(a^b)}{\sh}$ and $\h{(a^b)}$ are
\begin{center}
\hspace*{35pt}\begin{tikzpicture}[x=\gscale cm,y=\gscale cm]
\draw[->](0,-0.5)--(0,2.5); \draw[->](-0.5,0)--(11.5,0);
\draw(5,0)--(7,2)--(9,2)--(11,0);
\bigdot{(5,0)}\htick{5}
\bigdot{(7,2)}\htick{7}\vtick{2}
\bigdot{(9,2)}\htick{9}
\bigdot{(11,0)}\htick{11}
\node[left] at (-\vaxessep,2) {$\redsv a$};
\node at (5,-\axessep) {$\reds -b+\sh\hspace*{18pt}$};
\node at (7,-\axessep) {$\reds -b+a+\sh$};
\node at (9,-\axessep) {$\reds \sh$};
\node at (11,-\axessep) {$\reds a+\sh$};
\end{tikzpicture}

\hspace*{35pt}\begin{tikzpicture}[x=\gscale cm,y=\gscale cm]
\draw[->](0,-2.5)--(0,0.5); \draw[->](-0.5,0)--(11.5,0);
\draw(0,0)--(2,-2)--(4,-2)--(6,0);
\bigdot{(0,0)}\htick{0}
\bigdot{(2,-2)}\htick{2}\vtick{-2}
\bigdot{(4,-2)}\htick{4}
\bigdot{(6,0)}\htick{6}
\node[left] at (-\vaxessep,-2) {$\blues -a$};
\node at (2,\axessep) {$\blues a$};
\node at (4,\axessep) {$\blues b$};
\node at (6,\axessep) {$\blues a+b$};
\end{tikzpicture}\hspace*{12pt}
\end{center}
\end{lemma}

\begin{proof}
Suppose that $b \le a$.
The unique least and greatest elements of $C\bigl( (a^b) \bigr) + d$ are
$-b+1+d$ and $a-1+d$, respectively. Moreover $d$ and $a-b+d$ are the least
and greatest element of the maximum multiplicity~$b$.
Thus $({\red -b+d},0)$, $({\red b}, d)$,
$({\red a-b+d},0)$ and $({\red a+d},0)$ are points on the graph of
$\c{(a^b)}{d}$. It is easily seen that, between each adjacent
pair of points, the graph is linear.
The graph of $\h{(a^b)}$ can be found similarly.
\end{proof}

\newcommand{\casei}{{\blue b} \le {\red -b+\sh} \le {\red \sh} \le {\blue a} 
< {\blue a+b} \le {\red a-b+\sh} < {\red a+\sh}}
\newcommand{\caseii}{{\blue b} \le {\red -b+\sh} \le {\blue a} < {\red \sh}
\le {\blue a+b} \le {\red a-b+\sh} < {\red a+\sh}}
\newcommand{\caseiii}{{\blue b} \le {\blue a} < {\red -b+\sh} 
< {\blue a+b} < {\red \sh} \le {\red a-b+\sh} < {\red a+\sh}}
\newcommand{\caseiiip}{{\blue \bp} \le {\blue \ap} < {\red -\bp+\shp} 
< {\blue \ap+\bp} \le {\red \shp} \le {\red \ap-\bp+\shp} \le {\red \ap+\shp}}
\newcommand{\caseiv}{{\blue b} \le {\blue a} < {\blue a+b} \le {\red -b+\sh}
< {\red \sh} \le {\red a-b+\sh} < {\red a+\sh}}
\newcommand{\caseivd}{{\blue \bp} \le {\blue \ap} < {\blue \ap+\bp} \le {\red -\bp+\shp}
\le {\red \shp} \le {\red \ap-\bp+\shp} < {\red \ap+\shp}}
\newcommand{\casev}{{\blue a} < {\blue b} \le {\red -b+\sh} \le {\blue a+b}
\le {\red a-b+\sh} < {\red \sh} < {\red a+\sh}}
\newcommand{\casevd}{{\blue \ap} < {\blue \bp} \le {\red -\bp+\shp} \le {\blue \ap+\bp}
\le {\red \ap-\bp+\shp} < {\red \shp} < {\red \ap+\shp}}
\newcommand{\casevi}{{\blue a} < {\blue b} <  {\blue a+b} < {\red -b+\sh}
\le {\red a-b+\sh} \le {\red \sh} < {\red a+\sh}}
\newcommand{\casevid}{{\blue \ap} < {\blue \bp} <{\blue \ap+\bp} < {\red -\bp+\shp}
\le {\red \ap-\bp++\shp} \le {\red \shp} < {\red \ap+\shp}}

By Lemma~\ref{lemma:complementReduction}, we may reduce to the case where  $2b \le \sh$, which we now consider.

\begin{lemma}\label{lemma:orderInterleave}
Let $2b \le \sh$. If $b \le a$ then precisely one of:
\begin{thmlist}
\item $\casei$;
\item $\caseii$;
\item $\caseiii$;
\item $\caseiv$.
\end{thmlist}
If $a < b$ then precisely one of
\begin{thmlist}\setcounter{thmlistcnt}{4}
\item $\casev$;
\item $\casevi$.
\end{thmlist}
\end{lemma}

\begin{proof}
We have ${\red -b+d} < {\red d} < {\red a-b+d} < {\red a + d}$
and $b\le a \le a+b$. We must consider how these chains interleave.
By our reduction ${\blue b} \le {\red -b+\sh}$.

Suppose that $b \le a$.  
By the reduction, ${\blue a+b} \le {\red a-b+\sh}$, and so the interleaved chain
ends ${\red a-b+d} < {\red a+b}$.
If ${\red d} \le a+b$ then either
${\red d} \le a$, giving (i), or $a < {\red d} \le a+b$, giving (ii).
Otherwise $a+b < {\red d}$ and so $a < -{\red b} +{\red d}$. Either ${\red -b+d} < a+b$ giving~(iii)
or  $a+b \le {\red -b+d}$ giving~(iv).
 
Suppose that $a < b$. By the reduction,
$a+b \le {\red a-b+d} < {\red d} < {\red a} + {\red d}$,
so the interleaved chain ends ${\red a-b+d} < {\red d} < {\red a} + {\red d}$.
If ${\red -b+d} \le a+b$ we have (v), otherwise (vi).
\end{proof}

\begin{lemma}\label{lemma:chGraphs}
The graphs of $\ch{(a^b)}{\sh}$ in each of the cases in Lemma~\ref{lemma:orderInterleave} 
are as shown in Figure~4 overleaf.
\end{lemma}

\begin{proof}
This is routine from Lemma~\ref{lemma:chgraph} and Lemma~\ref{lemma:orderInterleave}.
\end{proof}

\begin{figure}
\newcommand{\pnode}{\node[left] at (-\vaxessep,-2) {$\phantom{\scriptstyle -2b+s-a}$};}

\begin{itemize}
\item[(i)] \centerfigure{\begin{tikzpicture}[x=\ggscale cm,y=\ggscale cm]
\draw[->](0,-2.5)--(0,2.5); \draw[->](-0.5,0)--(14.5,0);
\draw (0,0)--(2,-2)--(4,-2)--(6,0)--(8,0)--(10,2)--(12,2)--(14,0);
\bigdot{(0,0)}\htick{0}
\bigdot{(2,-2)}\htick{2}\vtick{-2}
\bigdot{(4,-2)}\htick{4}
\bigdot{(6,0)}\htick{6}
\bigdot{(8,0)}\htick{8}
\bigdot{(10,2)}\htick{10}\vtick{2}
\bigdot{(12,2)}\htick{12}
\bigdot{(14,0)}\htick{14}
\node[left] at (-\vaxessep,-2) {$-b$};
\node[left] at (-\vaxessep,2) {$b$};
\pnode
\node at (2,\axessep) {$\blue b$};
\node at (4,-\axessep) {$\reds -b+\sh$};
\node at (6.4,-\axessep) {$\reds \sh\phantom{-b}$};
\node at (8,\axessep) {$\blue a$};
\node at (10,\axessep) {$\blue a+b$};
\node at (11.75,-\axessep) {$\reds a-b+\sh$};
\node at (14,-\axessep) {$\reds a+\sh$};
\end{tikzpicture}}

\medskip

\item[(ii)] \centerfigure{\begin{tikzpicture}[x=\ggscale cm,y=\ggscale cm]
\draw[->](0,-3.5)--(0,3.5); \draw[->](-0.5,0)--(15.5,0);
\draw (0,0)--(3,-3)--(5,-3)--(6,-2)--(8,2)--(9,3)--(12,3)--(15,0);
\bigdot{(0,0)}\htick{0}
\bigdot{(3,-3)}\htick{3}\vtick{-3}
\bigdot{(5,-3)}\htick{5}
\bigdot{(6,-2)}\htick{6}\vtick{-2}
\bigdot{(8,2)}\htick{8}\vtick{2}
\bigdot{(9,3)}\htick{9}\vtick{3}
\bigdot{(12,3)}\htick{12}
\bigdot{(15,0)}\htick{15}
\pnode
\node[left] at (-\vaxessep,-3) {$-b$};
\node[left] at (-\vaxessep,3) {$b$};
\node[left] at (-\vaxessep,-2) {$a-\sh$};
\node[left] at (-\vaxessep,2) {$\sh-a$};
\node at (3,\axessep) {$\blue b$};
\node at (5,-\axessep) {$\reds -b+\sh$};
\node at (6,\axessep) {$\blue a$};
\node at (8.4,-\axessep) {$\reds \sh\phantom{-b}$};
\node at (9,\axessep) {$\blue a+b$};
\node at (12,-\axessep) {$\reds a-b+\sh$};
\node at (15,-\axessep) {$\reds a+\sh$};
\end{tikzpicture}}

\medskip
\item[(iii)] \centerfigure{\begin{tikzpicture}[x=\ggscale cm,y=\ggscale cm]
\draw[->](0,-3.5)--(0,3.5); \draw[->](-0.5,0)--(15.5,0);
\draw (0,0)--(3,-3)--(5,-3)--(6,-2)--(8,2)--(9,3)--(12,3)--(15,0);
\bigdot{(0,0)}\htick{0}
\bigdot{(3,-3)}\htick{3}\vtick{-3}
\bigdot{(5,-3)}\htick{5}
\bigdot{(6,-2)}\htick{6}\vtick{-2}
\bigdot{(8,2)}\htick{8}\vtick{2}
\bigdot{(9,3)}\htick{9}\vtick{3}
\bigdot{(12,3)}\htick{12}
\bigdot{(15,0)}\htick{15}
\node[left] at (-\vaxessep,-3) {$-b$};
\node[left] at (-\vaxessep,3) {$b$};
\node[left] at (-\vaxessep,-2) {$\scriptstyle -a-2b+\sh$};
\node[left] at (-\vaxessep,2) {$\scriptstyle a+2b-\sh$};
\node at (3,\axessep) {$\blue b$};
\node at (5,\axessep) {$\blue a$};
\node at (5.8,-\axessep) {$\reds -b+\sh$};
\node at (8.1,\axessep) {$\blue a+b$};
\node at (9.4,-\axessep) {$\reds \sh\phantom{-b}$};
\node at (12,-\axessep) {$\reds a-b+\sh$};
\node at (15,-\axessep) {$\reds a+\sh$};
\end{tikzpicture}}

\medskip
\item[(iv)] \centerfigure{\begin{tikzpicture}[x=\ggscale cm,y=\ggscale cm]
\draw[->](0,-2.5)--(0,2.5); \draw[->](-0.5,0)--(14.5,0);
\draw (0,0)--(2,-2)--(4,-2)--(6,0)--(8,0)--(10,2)--(12,2)--(14,0);
\bigdot{(0,0)}\htick{0}
\bigdot{(2,-2)}\htick{2}\vtick{-2}
\bigdot{(4,-2)}\htick{4}
\bigdot{(6,0)}\htick{6}
\bigdot{(8,0)}\htick{8}
\bigdot{(10,2)}\htick{10}\vtick{2}
\bigdot{(12,2)}\htick{12}
\bigdot{(14,0)}\htick{14}
\node[left] at (-\vaxessep,-2) {$-b$};
\node[left] at (-\vaxessep,2) {$b$};
\pnode
\node at (2,\axessep) {$\blue b$};
\node at (4,\axessep) {$\blue a$};
\node at (6,\axessep) {$\blue a+b$};
\node at (8,-\axessep) {$\reds -b+\sh$};
\node at (10.4,-\axessep) {$\reds \sh\phantom{-b}$};
\node at (11.75,-\axessep) {$\reds a-b+\sh$};
\node at (14,-\axessep) {$\reds a+\sh$};
\end{tikzpicture}}

\medskip
\item[(v)] \centerfigure{\begin{tikzpicture}[x=\ggscale cm,y=\ggscale cm]
\draw[->](0,-3.5)--(0,3.5); \draw[->](-0.5,0)--(15.5,0);
\draw (0,0)--(3,-3)--(5,-3)--(6,-2)--(8,2)--(9,3)--(12,3)--(15,0);
\bigdot{(0,0)}\htick{0}
\bigdot{(3,-3)}\htick{3}\vtick{-3}
\bigdot{(5,-3)}\htick{5}
\bigdot{(6,-2)}\htick{6}\vtick{-2}
\bigdot{(8,2)}\htick{8}\vtick{2}
\bigdot{(9,3)}\htick{9}\vtick{3}
\bigdot{(12,3)}\htick{12}
\bigdot{(15,0)}\htick{15}
\node[left] at (-\vaxessep,-3) {$-a$};
\node[left] at (-\vaxessep,3) {$a$};
\node[left] at (-\vaxessep,2) {$\scriptstyle a+2b-\sh$};
\node[left] at (-\vaxessep,-2) {$\scriptstyle -a-2b+\sh$};
\node at (3,\axessep) {$\blue a$};
\node at (5,\axessep) {$\blue b$};
\node at (5.8,-\axessep) {$\reds -b+\sh$};
\node at (8.1,\axessep) {$\blue a+b$};
\node at (9.5,-\axessep) {$\reds -b+a+\sh$};
\node at (12.4,-\axessep) {$\reds \sh\phantom{-b}$};
\node at (15,-\axessep) {$\reds a+\sh$};
\end{tikzpicture}}

\medskip
\item[(vi)] \centerfigure{\begin{tikzpicture}[x=\ggscale cm,y=\ggscale cm]
\draw[->](0,-2.5)--(0,2.5); \draw[->](-0.5,0)--(14.5,0);
\draw (0,0)--(2,-2)--(4,-2)--(6,0)--(8,0)--(10,2)--(12,2)--(14,0);
\bigdot{(0,0)}\htick{0}
\bigdot{(2,-2)}\htick{2}\vtick{-2}
\bigdot{(4,-2)}\htick{4}
\bigdot{(6,0)}\htick{6}
\bigdot{(8,0)}\htick{8}
\bigdot{(10,2)}\htick{10}\vtick{2}
\bigdot{(12,2)}\htick{12}
\bigdot{(14,0)}\htick{14}
\node[left] at (-\vaxessep,-2) {$-a$};
\node[left] at (-\vaxessep,2) {$a$};
\pnode
\node at (2,\axessep) {$\blue a$};
\node at (4,\axessep) {$\blue b$};
\node at (6,\axessep) {$\blue a+b$};
\node at (7.0,-\axessep) {$\reds -b+\sh$};
\node at (10.1,-\axessep) {$\reds -b+a+\sh$};
\node at (12.4,-\axessep) {$\reds \sh\phantom{-b}$};
\node at (14,-\axessep) {$\reds a+\sh$};
\end{tikzpicture}}

\end{itemize}
\caption{Graphs
 of $\ch{(a^b)}{\sh}$ in each of the cases in Lemma~\ref{lemma:orderInterleave}.}
\end{figure}
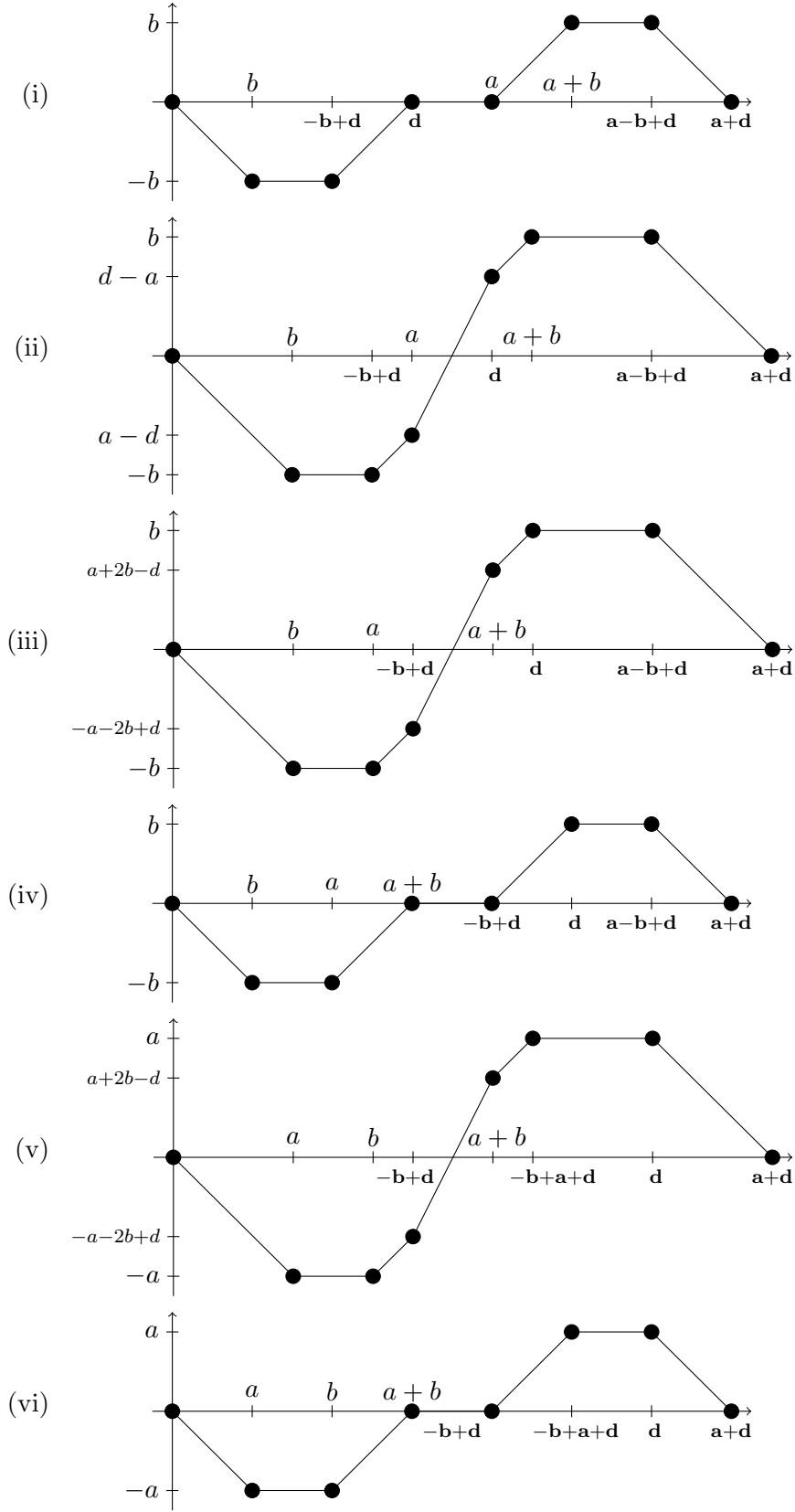

We are now ready to prove the `only if' direction
of Theorem~\ref{thm:rectangles}.


\begin{proof}
By hypothesis $\ell \ge \ell(\lambda)$ and $(a^b) \eqv{b+c-1}{\ell} \lambda$.
Let $d = b+c$. It follows from
Lemma~\ref{lemma:removable} that $\lambda$ is a rectangle.
Let \smash{$\lambda = ({a'}^{b'})$} and let $\ell = b' + c' - 1$ where $c' \in \N$. Set $d' = b'+c'$.
Since the six-fold equivalences given by the `if' direction of Theorem~\ref{thm:rectangles} form
a group, we may use Lemma~\ref{lemma:complementReduction} to assume that
$2b \le d$ and $2b' \le d'$.
Using~\eqref{eq:chiff}, it suffices to show that \smash{$\ch{(a^b)}{d} = \ch{({a'}^{b'})}{d'}$}
only if $(a',b',c')$ is a permutation of $(a,b,c)$. 

Say that a graph  in Lemma~\ref{lemma:chGraphs} is \emph{generic} if it has
a piecewise-linear part of gradient $0$ or $2$ in its middle. For example,
the graph in (i) is generic if and only if ${\red d} < a$.
For each generic graph there are two other generic graphs with which it may agree,
giving six cases we must check.
%
These are surprisingly simple to resolve. 
To give a typical instance,
suppose that the hook-content 
functions in case (i) for $(a^b)$ and shift $\sh$
and case (iv) for $\smash{(\ap^\bp)}$ and shift $\shp$ agree.
Comparing the inequality chains from Lemma~\ref{lemma:orderInterleave}, namely
\begin{align*}
&\casei \\
&\caseivd
\end{align*} 
shows that $\bp = b$, $\ap = -b+\sh = c$ and $\shp = a+b$.
Hence $\cp = \shp - \bp = a$ and
$(\ap,\bp,\cp) = (c,b,a)$. The corresponding equivalence is
$(a^b) \eqv{b+c-1}{a+b-1} (c^b)$.
The remaining generic cases are similar.
The equations satisfied by $\ap$, $\bp$, $\shp$,
the permutation of $(a,b,c)$ and the corresponding equivalence are shown in the table below;
the first line is the case already considered.

\begin{center}
\begin{tabular}{llllll} \toprule 
(i), (iv) & $\scriptstyle \ap = -b + \sh$ & $\scriptstyle \bp = b$ & $\scriptstyle \shp = a+b$ 
& $(c,b,a)$ 
& $(a^b)\eqv{b+c-1}{a+b-1}(c^b)$ \\
(i), (vi)  & $\scriptstyle \ap = b$ & $\scriptstyle \bp = -b+\sh$ & $\scriptstyle \shp = a-b+\sh$ 
& $(b,c,a)$ 
& $(a^b)\eqv{b+c-1}{a+c-1}(b^c)$ \\
(iv), (vi) & $\scriptstyle \ap = b$ & $\scriptstyle \bp = a$ & $\scriptstyle \shp = a-b+ \sh$ 
& $(b,a,c)$
& $(a^b)\eqv{b+c-1}{a+c-1}(b^a)$ \\ \midrule
(ii), (iii) & $\scriptstyle \ap = -b+\sh$ & $\scriptstyle \bp = b$ & $\scriptstyle \shp = a+b$ 
& $(c,b,a)$
& $(a^b)\eqv{b+c-1}{a+b-1}(c^b)$ \\
(ii), (v) & $\scriptstyle \ap = b$ & $\scriptstyle \bp = -b+\sh$ & $\scriptstyle \shp = a-b+\sh$ 
& $(b,c,a)$
& $(a^b)\eqv{b+c-1}{a+c-1}(b^c)$ \\
(iii), (v) & $\scriptstyle \ap = b$ & $\scriptstyle \bp = a$ & $\scriptstyle  \shp = a-b+\sh$ 
& $(b,a,c)$
& $(a^b)\eqv{b+c-1}{a+c-1}(b^a)$  \\
\bottomrule
\end{tabular}
\end{center}

In the non-generic cases
(i) and (ii) agree when $a=\sh$; (iii) and~(iv) agree when
$-b+\sh = a+b$; (v) and (vi) agree when $-b+\sh = a+b$. Therefore we need
only compare cases (i), (iii) and (v) using Lemma~\ref{lemma:chGraphs}. 
If (i) and~(iii) agree
then $\sh=a=-b+\sh = a+b$, hence $b = 0$, a contradiction. If (i) and~(v) agree then
$\sh=a=-b+\sh = a+b$, hence $b=0$, again a contradiction. It is impossible
for (iii) and (iv) to agree because $b \le a$ in (iii) and $a < b$ in~(v).
\end{proof}

\subsection{One-row partitions}

The special case of Theorem~\ref{thm:rectanglesFull}
for plethystic equivalences with a one-row partition 
is a natural generalization of Hermite reciprocity.
It was stated as Corollary~\ref{cor:oneRow} in the introduction.

\begin{proof}[Proof of Corollary~\ref{cor:oneRow}]
By definition, there is an isomorphism $\nabla^\lambda \SSym^\ell \E \cong \Sym^a \SSym^c \E$
of $\SL_2(\C)$-representations if and only if
$\lambda \eqv{\ell}{c} (a)$.
By Theorem~\ref{thm:rectanglesFull}, this
holds if and only if $\lambda$ is obtained by adding columns of length $\ell+1$
to a rectangle $(\ap^\bp)$  and $(\ap, \bp, \cp)$ is a permutation of $(a,b,c)$. After the usual reduction using Lemma~\ref{lemma:columnRemoval}
and Lemma~\ref{lemma:transitive}, we may assume the plethystic equivalence is
$(\ap^\bp) \eqv{\ell}{c} (a)$. Thus by Theorem~\ref{thm:rectanglesFull},
$(\ap,\bp, \ell-\bp+1)$ is a permutation of $(a,1,c)$.
Considering rectangles in conjugate pairs, we see that $(\ap^\bp)$ is one
of $(a)$, $(1^a)$, $(c)$, $(1^c)$, $(a^c)$, $(c^a)$ and the equivalence 
is respectively $(a) \eqv{c}{c} (a)$, $(1^a) \eqv{a+c-1}{c} (a)$,
$(c) \eqv{a}{c} (a)$, $(1^c) \eqv{a+c-1}{c}(a)$, $(a^c) \eqv{c}{c} (a)$,
$(c^a) \eqv{a}{c} (a)$, as required.
\end{proof}

\section{Irreducible skew-Schur functions}\label{sec:irreducible}

In this section we work in the more general setting of skew Schur functions. Recall that 
$\lambda / \lambda^\star$ is a \emph{skew partition} if $\lambda$ and $\lambda^\star$ are partitions
with $[\lambda^\star] \subseteq [\lambda]$.
Let $\SSYT_{\le \ell}(\lambda/\lambda^\star)$ be
the set of semistandard tableaux of shape $\lambda /\lambda^\star$ with
entries in $\{0,1,\ldots, \ell\}$, defined as in \S\ref{subsec:tableaux} but
replacing $[\lambda]$ with $[\lambda] / [\lambda^\star]$. 
The \emph{weight} of a skew tableau $t$, denoted $|t|$ is, as expected, its sum of entries.
Extending Definition~\ref{defn:SchurFunction} in the obvious way, the \emph{skew Schur function}
$s_{\lambda / \lambda^\star}$ is the symmetric function defined by
\begin{equation}
\label{eq:skewSchur}
s_{\lambda / \lambda^\star}(x_0,x_1,\ldots, x_n) 
= \sum_{t \in \SSYT_{\le \ell}(\lambda / \lambda^\star)} x^{|t|}.\end{equation}
Similarly, let
$\Se{e}{\ell}{\lambda / \lambda^\star}$ be the subset of  $\SSYT_{\le \ell}(\lambda/\lambda^\star)$ consisting of tableaux of weight~$e$.
Then
\begin{equation}
\label{eq:SeSkew} 
s_{\lambda / \lambda^\star}(1,q,\ldots, q^\ell) = \sum_{e\in\N_0} |\Se{e}{\ell}{\lambda /\lambda^\star}| 
q^e. \end{equation}

\begin{definition}\label{defn:ellIrred}
Let $\ell \in \N_0$ and let $\lambda /\lambda^\star$ be a skew-partition.
We say that $s_{\lambda / \lambda^\star}$ is \emph{$\ell$-irreducible} if 
there exists $b \in \N_0$ and $m \in \N_0$ such that $s_{\lambda/\lambda^\star}(1,q,\ldots, q^\ell)
= q^b(1 + q + \cdots + q^m)$. 
\end{definition}

In~\eqref{eq:skewIrred} we show that $b$ and $m$ are determined
in a simple way by $\lambda/\lambda^\star$ and~$\ell$. This result and Lemma~\ref{lemma:centralSkew}
can also be proved using 
the following remark; it is not
logically essential, but should help to motivate Definition~\ref{defn:ellIrred}.

\begin{remark}\label{remark:skewNabla}
The $\GL$-polytabloids $F(t)$ defined in~\S\ref{subsec:Schur} for tableaux $t$ of partition shape
with entries from $\{0,1,\ldots, \ell\}$
generalize in the obvious way to skew partitions. Using them we may
define $\nabla^{\lambda / \lambda^\star}V$, where as before $V = \langle v_0, \ldots, v_\ell \rangle$ 
is an $(\ell+1)$-dimensional complex vector space, to be the submodule
of $\bigotimes_{i=1}^{\ell(\lambda)} \Sym^{\lambda_i - \lambda_i^\star} E$
spanned by the $F(t)$ for $t$ a $\lambda/\lambda^\star$-tableau with entries
from $\{0,1, \ldots, \ell\}$. This defines skew Schur functors $\nabla^{\lambda / \lambda^\star}$
in a way that does not depend on Littlewood--Richardson coefficients, or the complete
reducibility of representations of $\SL_2(\C)$.
Generalizing Lemma~\ref{lemma:SchurChar}, we have
\[ \Phi_{\nabla^{\lambda / \lambda^\star} \SSym^\ell \E}(1,q) = 
s_{\lambda /\lambda^\star}(1,q,\ldots, q^\ell).\] 
By a generalization
of the equivalence of (a) and (e) in Theorem~\ref{thm:eqvConds}, 
$s_{\lambda / \lambda^\star}$ is $\ell$-irreducible in the sense of Definition~\ref{defn:ellIrred}
if and only if the polynomial representation 
$\nabla^{\lambda/\lambda^\star} \SSym^\ell \E$ of $\SL_2(\C)$ is irreducible.
\end{remark}

Note that $\ell =0$ is permitted in Definition~\ref{defn:ellIrred}
and in the previous remark. 
Since $s_{\lambda/\lambda^\star}(1,q,\ldots,q^\ell) $ is non-zero
if and only if every column of $[\lambda/\lambda^\star]$ has length at most $\ell+1$,
the $0$-irreducible skew-partitions are precisely those with at most one box
in each column.

\subsection{Irreducible skew Schur functions}

In this section we state a classification of all
skew partitions $\lambda /\lambda^\star$ and $\ell \in \N$
such that $s_{\lambda / \lambda^\star}$ is $\ell$-irreducible.
We then deduce Corollary~\ref{cor:irreducible}.
The following definition leads to a useful reduction.

\begin{definition}\label{defn:proper}
We say that a skew partition $\lambda / \lambda^\star$ is \emph{proper} if $\lambda_1 > \lambda^\star_1$
and $\lambda_1' > \lambdastarp_1$.
\end{definition}

Given a non-empty skew partition $\pi / \pi^\star$ 
one may repeatedly remove the longest rows and columns from each of $\pi$ and $\pi^\star$
to obtain the Young diagram of
a unique proper skew partition $\lambda / \lambda^\star$ such
that $[\lambda /\lambda^\star] = [\pi / \pi^\star]$, as illustrated
in Figure~5.
\begin{figure}[t]
\begin{center}
\scalebox{0.9}{\includegraphics{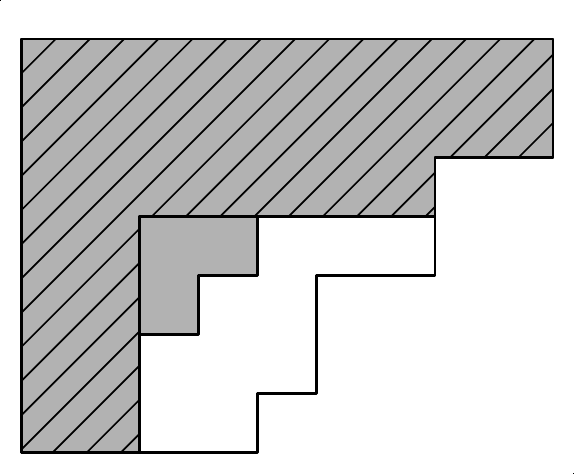}}
\caption{The Young diagram of a skew partition $\pi /\pi^\star$ with $\pi^\star$
shaded in grey is shown. 
Deleting the hatched boxes 
leaves the Young diagram of the proper skew partition $\lambda / \lambda^\star$, where $[\lambda^\star]$ consists
of the shaded unhatched boxes.}
\end{center}
\end{figure}
There is an obvious
bijection between
$\SSYT_{\{0,\ldots, \ell\}}(\pi / \pi^\star)$ and
$\SSYT_{\{0,\ldots, \ell\}}(\lambda / \lambda^\star)$. Therefore, by~\eqref{eq:skewSchur}, we have
$s_{\pi / \pi^\star}(1,q,\ldots, q^\ell) = s_{\lambda / \lambda^\star}(1,q,\ldots, q^\ell)$.
Thus there is no loss of generality
in restricting to proper skew partitions.
Our classification in this case uses the following definition.

\begin{definition}\label{defn:nearRectangle}
Given a proper skew partition $\lambda / \lambda^\star$ with $a(\lambda) = p$,
we define the \emph{column lengths} $c(\lambda / \lambda^\star) \in \N^p$ by $c(\lambda / \lambda^\star)_j = 
\lambda'_j - \lambdastarp_j$ for $1 \le j \le p$. 
We say that $\lambda / \lambda^\star$~is
\begin{itemize}
\item[(a)] a \emph{skew $\ell$-rectangle} where $\ell \in \N_0$ if $\clambda{1} = \ldots = \clambda{p} = \ell+1$;
\item[(b)] a \emph{skew $1$-near rectangle} of \emph{width} $d \in \N_0$ if there exist $y$ 
such that 
\begin{align*} \clambda{y} &= \ldots = \clambda{y+d-1} = 1, \\
\lambda'_y &= \ldots = \lambda'_{y+d-1}, \end{align*}
 and $\clambda{j} = 2$ if $1 \le j < y$
or $y+d \le j \le p$;
\item[(c)] a  \emph{skew $\ell$-near rectangle} where $\ell \ge 2$ if
there exists 
$\z$ such that $\clambda{\z} \in \{1, \ell\}$
and $\clambda{j} = \ell+1$ if $1 \le j \le p$ and $j \not= \z$.
\end{itemize}
\end{definition}

The Young diagrams of a skew $0$-rectangle,
a skew $1$-near rectangle of width $3$ and a skew $2$-near rectangle are
shown below; the final diagram fails
the second displayed condition in (ii), so is not a skew $1$-near rectangle.
\[ \renewcommand{\b}{\ }\young(:::\b\b,:\b\b,\b)\, \quad\quad\! \young(:::::\b,:\b\b\b\b\b,\b\b,\b) \, 
\qquad \young(::::\b,:\b\b\b\b,\b\b\b\b\b,\b\b\b,\b) \qquad
\young(:::::\b,::::\b\b,\b\b\b\b,\b) \]
We can now state the classification.

\begin{theorem}\label{thm:irreducible}
Let $\lambda / \lambda^\star$ be a proper skew partition.
Then $s_{\lambda / \lambda^\star}$ is $\ell$-irreducible
if and only if $\lambda / \lambda^\star$ is a skew $\ell$-rectangle
\emph{or} $\lambda / \lambda^\star$ is a skew $\ell$-near rectangle.
\end{theorem}

We immediately deduce the corollary for Schur functors labelled by partitions stated in the introduction.

\begin{proof}[Proof of Corollary~\ref{cor:irreducible}]
By Lemma~\ref{lemma:SchurChar}, 
$\nabla^\lambda \Sym^\ell \E \cong \Sym^n \E$ for some $n \in \N_0$
if and only if $s_\lambda$ is $\ell$-irreducible.
The only skew $0$-rectangles are one-part partitions.
If, as in (b), $\ell = 1$ and so $\lambda$ is either a skew $1$-rectangle,
in which case $\lambda = (n/2, n/2)$ for some even $n$, or a skew $1$-near rectangle,
in which case $\lambda = (n-m,m)$, for some $1 \le m \le n/2$ and $\ell(\lambda) = 2$.
This gives case (i) of the corollary.
If, as in~(c), $\ell \ge 2$, then $\lambda$ is either
a skew $\ell$-rectangle, of the form $(p^{\ell+1})$, or
a skew $\ell$-near rectangle; then all but the final
column of $\lambda$ has length $\ell+1$ and the final
column (which may be the only column)
has length either $1$ or~$\ell$. This gives case (ii).
\end{proof}

To prove Theorem~\ref{thm:irreducible} we need the preliminary results
in the following subsection.

\subsection{Unimodality of specialized skew Schur functions}
Fix a skew partition $\lambda / \lambda^\star$ of size $n$.
The minimum weight defined in 
Definition~\ref{defn:b} generalizes as follows to skew tableaux.
%
\begin{definition}\label{defn:minimumDegreeSkew}
We define the \emph{minimum weight} of $\lambda / \lambda^\star$ by
\[ \b(\lambda/\lambda^\star) = \sum_{j=1}^{a(\lambda)}  \binom{\lambda_j'-\lambdastarp_j}{2} 
.\]
\end{definition}
Equivalently, $b(\lambda/\lambda^\star) = \sum_{j=1}^{a(\lambda)} \binom{c(\lambda/\lambda^\star)_j}{2}$.
Observe that $\b(\lambda/\lambda^\star)$ is the weight
of the tableau $t(\lambda/\lambda^\star)$ having entries $0,1,\ldots \lambda'_j-1$
in column $j$, for $1 \le j \le a(\lambda)$. It is easily seen
that this tableau is semistandard and has the minimum 
weight of any tableau in $\SSYT_{\le \ell}(\lambda/\lambda^\star)$.

\begin{lemma}\label{lemma:centralSkew}
The specialization
$s_{\lambda /\lambda^\star}(1,q,\ldots, q^\ell)$
is unimodal and centrally symmetric about $\mfrac{\ell n}{2}$.
\end{lemma}

\begin{proof}
Like any symmetric function, $s_{\lambda / \lambda^\star}$ can be expressed as a
linear combination of Schur functions labelled by partitions. 
The lemma therefore follows from 
the `moreover' part of Theorem~\ref{thm:eqvConds}.
\end{proof}

By Lemma~\ref{lemma:centralSkew} and~\eqref{eq:SeSkew}, 
$s_{\lambda /\lambda^\star}$ is $\ell$-irreducible if and only if
\begin{equation}
\label{eq:skewIrred} 
|\Se{\b(\lambda)}{\ell}{\lambda/\lambda^\star}| 
= |\Se{\b(\lambda)+1}{\ell}{\lambda/\lambda^\star}| = 
\ldots = |\Se{\lfloor \ell n /2\rfloor}{\ell}{\lambda/\lambda^\star}| = 1 
\end{equation}
Moreover, if~\eqref{eq:skewIrred} holds
then $s_{\lambda / \lambda^\star}(1,q,\ldots, q^\ell) = 
q^{\b(\lambda/\lambda^\star)}(1+q + \cdots + q^m)$ where
$m = \ell n - \b(\lambda/\lambda^\star)$. 
We also obtain the following  lemma.

\begin{lemma}\label{lemma:safeToIncrease}
If $|\Se{e}{\ell}{\lambda/\lambda^\star} |< |\Se{e+1}{\ell}{\lambda/\lambda^\star}|$
then $e < \mfrac{\ell n}{2}$.
\end{lemma}

\begin{proof}
This is immediate from the unimodality property in 
Lemma~\ref{lemma:centralSkew} and~\eqref{eq:SeSkew}.
\end{proof}


\subsection{Bumping and the proof of Theorem~\ref{thm:irreducible}}

\begin{definition}\label{defn:bump}
Given $t \in \SSYT_{\le \ell}(\lambda/\lambda^\star)$
and a box $(i,j) \in [\lambda/\lambda^\star]$, we define the \emph{bump} of $t$ in box $(i,j)$
to be the $\lambda/\lambda^\star$-tableau $t^+$ that agrees with~$t$ except in this box,
where $t^+_{(i,j)} = t_{(i,j)} + 1$. We say that $t$ is \emph{bumpable} in box $(i,j)$
if $t^+ \in 
\SSYT_{\le \ell}(\lambda/\lambda^\star)$.
\end{definition}

Equivalently, $t$ is bumpable in box $(i,j)$ if and only if $t_{(i,j)} < \ell$,
and increasing the entry of $t$ in position $(i,j)$ by $1$ does not violate the semistandard condition. 
The following example shows the use of Definition~\ref{defn:bump} in the harder `only if' part of the proof
of Theorem~\ref{thm:irreducible}.


\begin{example}\label{ex:bumping}
Let $\ell = 3$. Suppose that $\lambda/\lambda^\star = (3^4,1) / (2)$, so $\lambda / \lambda^\star$
is a skew $3$-near rectangle. Then
$\b(\lambda / \lambda^\star) = 15$ and 
\[ 
\young(::0,001,112,223,3)\, , \quad
\young(::0,001,112,233,3)\, , \quad
\young(::0,001,122,233,3)\, , \quad
\young(::0,011,122,233,3)\, \phantom{.} \]
are the unique tableaux in $\Seb{e}{3}{(3^4,1) / (2)}$ for
$e \in \{15,16,17,18\}$.
The first tableau  is $t(\lambda / \lambda^\star)$, and the rest
are obtained by successive bumps in positions $(4,2)$, $(3,2)$ and $(2,2)$.
By~\eqref{eq:skewIrred}, $s_{(3^4,1)/(2)}$ is $3$-irreducible.
Suppose instead that $\lambda / \lambda^\star = (3^4,1^2) / (2^2)$. Then
$\b(\lambda / \lambda^\star) = 13$ and
\[ 
\young(::0,::1,002,113,2,3)\, , \quad
\young(::0,::1,002,123,2,3)\, , \quad
\young(::0,::1,012,123,2,3)\, , \quad
\young(::0,::1,002,133,2,3)\,  \]
are the unique tableaux in $\Seb{e}{3}{(3^4,1^2)/(2^2)}$ for $e \in \{13,14\}$,
and the two tableaux in $\Seb{15}{3}{(3^4,1^2)/(2^2)}$.
Again the first tableau is $t(\lambda / \lambda^\star)$. The second is its bump
in position $(4,2)$, and the third and fourth both of weight
$|t(\lambda / \lambda^\star)| + 2$ are the bumps of the second
in positions $(3,2)$ and $(4,2)$, respectively.
Since the condition in~\eqref{eq:skewIrred} fails, $s_{(3^4,1^2)/(2^2)}$ is not $3$-irreducible.
Note that, as implied
by Lemma~\ref{lemma:safeToIncrease}, $b(\lambda/\lambda^\star) + 1 < \mfrac{\ell n}{2}$.
\end{example}

\subsubsection*{Sufficiency}
To illuminate the condition in Theorem~\ref{thm:irreducible},
we prove a slightly stronger result.

\begin{lemma}\label{lemma:sufficient}
 If $\lambda / \lambda^\star$ is a skew $\ell$-rectangle,
a skew $1$-near rectangle, or a skew $\ell$-near rectangle where $\ell \ge 2$ then
$s_{\lambda / \lambda^\star}$ is $\ell$-irreducible. Moreover,
$s_{\lambda / \lambda^\star}(1,q,\ldots,q^\ell)$ is respectively
$q^{\ell n /2}$, $q^{\b(\lambda)} + q^{\b(\lambda)+1} + \cdots + q^{\b(\lambda) + d}$
and $q^{\b(\lambda)} + q^{\b(\lambda)+1} + \cdots + q^{\b(\lambda) + \ell}$.
\end{lemma}

\begin{proof}
Write $c$ for $c(\lambda /\lambda^\star) \in \N^p$.
If $\lambda / \lambda^\star$ is a skew $\ell$-rectangle
then $c = (\ell+1, \ldots, \ell+1)$ and $\b(\lambda) = p\binom{\ell+1}{2} =
p \ell (\ell+1)/2 = \mfrac{\ell n}{2}$, 
so~\eqref{eq:skewIrred} obviously holds. By~\eqref{eq:SeSkew},
$s_{\lambda / \lambda^\star}(1,q,\ldots,q^\ell) = q^{\ell n /2}$.

Suppose that $\ell=1$ and $\lambda / \lambda^\star$ is a skew $1$-near rectangle of width $d$.
Let $y$ be as in  Definition~\ref{defn:nearRectangle},
so $c_y = \ldots = c_{y + d-1} = 1$ and $c_j = 2$ if $0 \le j < y$ or $y+d \le j \le p$. 
The minimum weight tableau $t(\lambda/\lambda^\star)$ has 
entries of~$0$ in the boxes $(1,y), \ldots, (1,y+d-1)$;
all other boxes are in a column $j$ with $c_j = 2$, having entries $0$ and $1$.
More generally,
for each $k$ such that $0\le k \le d$, the unique
tableau in $\SSYT_{\le 1}(\lambda / \lambda^\star)$ of weight $\b(\lambda / \lambda^\star)  + k$
has entries of $0$ in the boxes $(1,y), \ldots, (1,y+d-k-1)$ and entries
of $1$ in the boxes $(1,y+d-k), \ldots, (1,y+d-1)$. 
Hence~\eqref{eq:skewIrred} holds. 
By~\eqref{eq:SeSkew},
$s_{\lambda / \lambda^\star}(1,q) = q^{\b(\lambda / \lambda^\star)} +
q^{\b(\lambda / \lambda^\star)+1} + \cdots + q^{\b(\lambda / \lambda^\star)+d}$.

Now suppose that $\ell \ge 2$ and that $\lambda /\lambda^\star$ is a skew $\ell$-near rectangle.
Let $z$ be unique such that $c_z \in \{1,\ell\}$.
Let $t\in \SSYT_{\le \ell }(\lambda / \lambda^\star)$.
If $j\not= z$ then since $c_j = \ell+1$, the entries in column $j$ of $t$
are $0,\ldots, \ell$, and $t$ agrees with $t(\lambda / \lambda^\star)$ in these columns.
We now consider the two cases for $c_z$.
\begin{itemize}
\item[(i)] When $c_z = 1$ the unique entry in column $k$ of $t$
is determined by~$|t|$.
Moreover $|t|$ takes all values in $\{ \b(\lambda), \ldots, \b(\lambda) + \ell \}$
so by~\eqref{eq:SeSkew},
$s_{\lambda / \lambda^\star}(1,q,\ldots, q^\ell) = q^{\b(\lambda / \lambda^\star)} +
q^{\b(\lambda / \lambda^\star)+1} + \cdots + q^{\b(\lambda / \lambda^\star)+\ell}$.
In particular~\eqref{eq:skewIrred} holds. 

\smallskip
\item[(ii)] When $c_z = \ell$, we have $\b(\lambda) = (p-1) \binom{\ell+1}{2} + \binom{\ell}{2}
= p \ell(\ell+1)/2 - \ell$ and $\mfrac{\ell n}{2} = p \ell (\ell+1)/2 - \ell/2$.
Let $b \in \{0, 1, \ldots, \ell\}$. The unique tableau in
$\SSYT_{\le \ell }(\lambda/\lambda^\star)$ of weight $\b(\lambda) + b$
 has entries $\{0,\ldots, \ell\} \backslash \{ \ell - b\}$ in column $z$.
Thus again~\eqref{eq:skewIrred} holds. A similar argument to (i)
shows that $s_\lambda(1,q,\ldots,q^\ell)$ has the
required $\ell+1$ summands.\hfill$\qedhere$
\end{itemize}
\end{proof}

\subsubsection*{Necessity}
The following lemma implies that if $s_{\lambda/\lambda^\star}$ is $\ell$-irreducible
then $t(\lambda/\lambda^\star)$  has at most one bumpable box.

\begin{lemma}\label{lemma:oneBump}
Let $B$ be the number of boxes $(i,j) \in [\lambda/\lambda^\star]$ such that $t(\lambda/\lambda^\star)$ 
is bumpable in box $(i,j)$. If $B \ge 2$ then
$|\Se{\b(\lambda/\lambda^\star)+1}{\ell}{\lambda/\lambda^\star}|
= B$ and
$\b(\lambda/\lambda^\star) < \mfrac{\ell n}{2}$.
\end{lemma}

\begin{proof}
The first equality is immediate from the definition of bumpable in Definition~\ref{defn:bump}.
The inequality $\b(\lambda/\lambda^\star) < \mfrac{\ell n}{2}$ now follows by 
taking $e = \b(\lambda) $
 in Lemma~\ref{lemma:safeToIncrease}.
\end{proof}

\begin{proposition}\label{prop:oneBump}
Let $\lambda / \lambda^\star$ 
be a proper skew partition.
If $s_{\lambda / \lambda^\star}$ is $\ell$-irreducible then
\emph{either}
\begin{thmlist}
\item $c(\lambda / \lambda^\star) = (\ell+1,\ldots, \ell+1)$ \emph{or} 
\item there exists
$k < \ell$ such that
$c(\lambda / \lambda^\star) = (\ell+1,\ldots, \ell+1, k+1, \ldots, k+1, \ell+1, \ldots, \ell+1)$
and $\lambda'$ is constant in the positions in which $c(\lambda /\lambda^\star)$ is $k+1$.
\end{thmlist}
\end{proposition}

\begin{proof}
Suppose that $s_{\lambda / \lambda^\star}$ is $\ell$-irreducible. 
Since $s_{\lambda/\lambda^\star}(1,q,\ldots,q^\ell) \not=0$,
the minimum weight tableau 
$t(\lambda/\lambda^\star)$ has entries in $\{0,\ldots, \ell\}$, and so $\clambda{j} \le \ell + 1$
for each $j$.
By Lemma~\ref{lemma:oneBump},
$t(\lambda / \lambda^\star)$ is bumpable in at most one box. If~(i) does not hold then 
there exists a column $j$ such that~\hbox{$\clambda{j} \le \ell$}. 
Take~$y$ minimal with this
property and let $z$ be greatest such that
$\clambda{y} = \ldots = \clambda{z}$. Thus
\[ t(\lambda/\lambda^\star)_{(\lambda'_y,y)}, \ldots, t(\lambda/\lambda^\star)_{(\lambda'_{z},z)} < \ell. \]
Either $(\lambda'_z, z+1) \not\in [\lambda/\lambda^\star]$
or $\clambda{z+1}>\clambda{z}$.
 In either case $t(\lambda/\lambda^\star)$ is bumpable in the box
$(\lambda'_z, z)$.
Since $\lambda'$ is a partition, $\lambda'_y \ge \ldots \ge \lambda'_z$.
Suppose that $\lambda'_j > \lambda'_{j+1}$ where $y \le j < z$. Then $(\lambda'_j, j+1) \not\in [\lambda/\lambda^\star]$
so $t(\lambda /\lambda^\star)$ is bumpable in the box $(\lambda'_j, j)$, as well as in
the box $(\lambda'_z, z)$, a contradiction. Therefore $\lambda'_y = \ldots = \lambda'_z$. 
If there exists a column $j$ such that $\clambda{j} \le \ell$
and $j \not\in \{y,\ldots, z\}$ then repeating this argument gives
another box in which $t(\lambda/\lambda^\star)$ is bumpable, again a contradiction. Therefore
$c(\lambda/\lambda^\star)$ is as claimed in (ii).
\end{proof}


\begin{proof}[Proof of Theorem~\ref{thm:irreducible}]
We have already shown the condition is sufficient.
Suppose that  $s_{\lambda / \lambda^\star}$ is $\ell$-irreducible but $\lambda / \lambda^\star$ is not a skew $\ell$-rectangle and that
$\lambda / \lambda^\star$ is not a skew $\ell$-near rectangle. 
By Proposition~\ref{prop:oneBump}, there exists $y$, $z \in \{1,\ldots, p\}$
and $k < \ell$ such that
$c_1 = \ldots = c_{y-1} = \ell+1$, $c_y = \ldots = c_z = k+1$,
$c_{z+1} = \ldots = c_p = \ell + 1$ and
$\lambda_{y}' = \ldots = \lambda_z'$.
If $\ell = 1$ then $\lambda / \lambda^\star$
is a $1$-near rectangle, as required.
Suppose that $\ell \ge 2$.  
Note that $t(\lambda/\lambda^\star)_{(\lambda_z', z)} = k$
and that $(\lambda_z', z)$ is the unique box in which $t(\lambda/\lambda^\star)$ is bumpable.
 Let $u$ be the bump of $t(\lambda/\lambda^\star)$ in this box; thus
 $\Se{\b(\lambda / \lambda^\star)+1}{\ell}{\lambda / \lambda^\star} = \{ u \}$.
We consider three cases.

\begin{itemize}
\item[(a)] Suppose that $1\le k < \ell-1$. (This is the case in the second
example in Example~\ref{ex:bumping}.) Since $u_{(\lambda_z', z)} = k+1 < \ell$
and either
 $u_{(\lambda_{z}', z+1)} = \ell$ or $(\lambda_{z}', z+1) \not\in [\lambda/\lambda^\star]$,
$u$ is bumpable in position
$(\lambda_z', z)$. Similarly, either 
$u_{(\lambda_{z}'-1, z+1)}  \ge \ell-1$ or $(\lambda_{z}'-1, z+1) \not\in [\lambda/\lambda^\star]$. 
Therefore
$u$ is bumpable in box  $(\lambda_z'-1, z)$.
Thus $|\Se{\b(\lambda/\lambda^\star)+2}{\ell}{\lambda/\lambda^\star}| \ge 2$ and 
since $|\Se{\b(\lambda/\lambda^\star)+1}{\ell}{\lambda/\lambda^\star}| = 1$, it follows
from Lemma~\ref{lemma:safeToIncrease} that $\b(\lambda/\lambda^\star) +1 <\mfrac{\ell n}{2}$.
Therefore~\eqref{eq:skewIrred} does not hold.

\smallskip
\item[(b)] Suppose that $k = \ell-1$. Since $\lambda / \lambda^\star$ is not a 
skew $\ell$-near rectangle, we have $y < z$. 
 As in (a), $u$ is bumpable 
in position $(\lambda_z'-1, z)$.
Moreover, since $\lambda_{z-1}' = \lambda_z'$ and $c_{z-1} = c_z$,
we have $u_{(\lambda_{z-1}', z-1)} = \ell-1$ and so~$u$ is bumpable in position $(\lambda_{z-1}', z-1)$.
Thus $|\Se{\b(\lambda/\lambda^\star)+2}{\ell}{\lambda/\lambda^\star}| \ge 2$ and as in (a)
we conclude that~\eqref{eq:skewIrred} does not hold.

\smallskip
\item[(c)] Suppose $k=0$. Then  $u_{(\lambda_{z}', z)} = 1$ and box $(\lambda_{z}', z)$ is bumpable as $\ell \ge 2$. But if
$y<z$ then box $(\lambda_{z}', z-1)$ is also bumpable in $u$, giving that $|\Se{\b(\lambda/\lambda^\star)+2}{\ell}{\lambda/\lambda^\star}| \ge 2$ and~\eqref{eq:skewIrred} again  fails to hold.
\end{itemize}

This completes the proof.
\end{proof}

\section{Two row, two column and hook equivalences}
\label{sec:twoRowAndHook}

We say that a partition of hook shape $(a+1,1^b)$ is \emph{proper} if $a$, $b \in \N$.

%

\begin{theorem}\label{thm:hook2row}
Let $\lambda$ and $\mu$ be partitions that each, separately, 
have either precisely two rows, precisely two columns or are of proper hook shape. 
Let $\ell$, $m \in \N$ be such that $\ell \ge \ell(\lambda)$ and $m \ge \ell(\mu) $. 
 Then all plethystic equivalences $\lambda \eqv{\ell}{m} \mu$ are
 listed in one of the cases in
 Table~1.
\end{theorem}

\begin{table}
\begin{center}
\begin{tabular}{lll}\toprule 
\multicolumn{3}{l}{\emph{Equivalences between hook partitions}}  \\
(a) & $(a+1, 1^b)  \eqv{\ell}{\ell} (a+1, 1^b)$ \\
(b) & $(a+1, 1^b)  \eqv{\ell}{\ell+a-b} (b+1, 1^a)$ & (conjugate, Theorem~\ref{thm:conjugates}) \\[3pt]
\midrule
\multicolumn{3}{l}{\emph{Equivalences between two-row non-hook partitions}} \\
(c) & $(a,b)  \eqv{\ell}{\ell} (a,b)$ \\
(d) & $(a,a)  \eqv{c+1}{a+1} (c,c)$ & (rectangular, Theorem~\ref{thm:rectangles}) \\
(e) &  $(a,b)  \eqv{2}{2} (a,a-b)$  & (complement, Theorem~\ref{thm:complements}) \\
(f) & $(2\ell,\ell+2) \eqv{\ell} {\ell+2}(2\ell-2, \ell-2) $ \\ \midrule
\multicolumn{3}{l}{\emph{Equivalences between two-column non-hook partitions}} \\
(g) & $(2^a,1^b)   \eqv{\ell}{\ell} (2^a,1^b) $ \\
(h) & $(2^a)  \eqv{a-1}{c-1} (2^c)$ & (rectangular, Theorem~\ref{thm:rectangles}) \\
(i) & $(2^a,1^b)  \eqv{a+b+c-1}{a+b+c-1} (2^c,1^b)$ & (complement, Theorem~\ref{thm:complements}) \\[3pt]
\midrule
\multicolumn{3}{l}{\emph{Equivalences between a two-row non-hook and a hook  partition}} \\
(j) & $(a,b) \raisebox{3pt}{$\begin{matrix*}[l]\eqv{a-b+1}{a} (a-b+1,1^b) \\ 
                                               \eqv{a-b+1}{2(a-b)} (b+1,1^{a-b})\end{matrix*}$}$
                                               \\
(k) & $(3\ell-3, 2\ell-1) 
\raisebox{3pt}{$\begin{matrix*}[l]\eqv{\ell}{3\ell-4} (\ell+1,1^{\ell-2}) \\
                                  \eqv{\ell}{3\ell-2} (\ell-1,1^{\ell})\end{matrix*}$}$ \\[9pt]
                                  \midrule
\multicolumn{3}{l}{\emph{Equivalences between a two-column non-hook and a  hook  partition}} \\
(l) & $(2^a,1^b)  \eqv{a+b}{a+b} (2,1^b)$  &  (complement, Theorem~\ref{thm:complements}) \\
(m) & $(2^a,1^c)  \eqv{a+c}{a+1} (c+1,1)$ \\[3pt] \midrule
\multicolumn{3}{p{4.4in}}{\emph{Equivalences between a  two-row and a  two-column partition
both non-hooks}}\\
(n) & $(a,a)  \eqv{\ell}{\ell+a-2} (2^a)$ &  (conjugate, Theorem~\ref{thm:conjugates}) \\
(o) & $(a,a)  \eqv{b+1}{a+b-1} (2^b)$  & (rectangular, Theorem~\ref{thm:rectangles}) \\
(p) & $(6,5)  \eqv{3}{7} (2^4,1^3)$.
\\[3pt] \bottomrule
\end{tabular}
\end{center}

\medskip
\caption{All plethystic equivalences between partitions that, separately,
have either precisely two rows, precisely two columns, or are of proper hook shape.
In cases (j) and (k) the two  hook partitions on the right are conjugates.}
\end{table}

\enlargethispage{15pt}

\begin{proof}
The proofs for each family in Table~1 are similar.
We illustrate the method by finding all 
plethystic equivalences between a two-row non-hook partition 
$\lambda=(a,b)$ and a proper hook 
$\mu=(c+1, 1^d)$. 
By our assumption,  
$\ell \ge 2$ and $m \ge d+1$.
By Lemma~\ref{lemma:transitive} and Theorem~\ref{thm:conjugates}
we may assume that $c \ge d$.
By Theorem~\ref{thm:eqvConds}(h), $\lambda \eqv{\ell}{m} \mu$ if and only if 
there is an equality of multisets 
$(C(\lambda) + \ell + 1) \cup H(\mu) = (C(\mu) + m + 1) \cup H(\lambda)$. Equivalently,
$$\left\{
\begin{array}{l}
 \ell+1, \ldots, \ell+a, \\
\ell,\ldots, \ell+b-1,\\
1, \ldots, d, 1 \ldots, c,\\
c+d+1 \end{array}
\right\}
= \left\{\begin{array}{l} 
  m-d+1, \ldots, m+c+1,\\
  1,\ldots, b,\\
  1\ldots, a-b, a-b+2, \ldots a+1
  \end{array}
\right\}.
$$
Comparing the greatest element of each side as in Proposition~\ref{prop:aRelation}
 shows that if equality holds then
 $ \ell+a= m+c+1$, that is $ m=\ell+a-c-1$. Substituting for $m$ using this relation, and
 inserting $a-b+1$ into each multisubset,
 we find that $\lambda \eqv{\ell}{\ell+a-c-1} \mu$ if and only if
 \enlargethispage{24pt}
\begin{equation} \label{eqn:hook2rowmultisets}
\left\{
\begin{array}{l}
\ell+1, \ldots, \ell+a, \\
\ell,\ldots, \ell+b-1,\\
1, \ldots, d, 1 \ldots, c,\\
c+d+1 , a-b+1 \end{array}
\right\}
= \left\{\begin{array}{l} 
  \ell+a-c-d, \ldots, \ell+a,\\
  1,\ldots, b,\\
  1\ldots, a+1
  \end{array}
\right\}.
\end{equation}

Firstly consider the case when $a-c-d \ge 1$. 
We may cancel the 
elements $\ell+a-c-d, \ldots, \ell+a$ from each side to get that
$\lambda \eqv{\ell}{\ell+a-c-1} \mu$ if and only~if
$$\left\{
\begin{array}{l}
\ell+1, \ldots, \ell+a-c-d-1, \\
\ell,\ldots, \ell+b-1,\\
1, \ldots, d, 1 \ldots, c,\\
c+d+1 , a-b+1 \end{array}
\right\}
= \left\{\begin{array}{l} 
  1,\ldots, b,\\
  1,\ldots, a+1
  \end{array}
\right\}.
$$
We claim that this multiset equality implies $a-c-d \le b$. Indeed, if $a-c-d>b$ then, 
on the left hand side, $\ell+1 \le \ell+b-1 < \ell+a-c-d-1$. 
Therefore the multiplicity of $\ell+b-1$ is two, and comparing with the multiset on the right shows
that  $\ell=1$, contrary to our initial assumption.

As $a-c-d \le b$,  we may 
compare greatest elements of the above multisets to show that $a+1=\ell+b-1$, 
that is $\ell=a-b+2$. We substitute for $\ell$ using this relation 
and cancel the elements  
\[ \{ a-b+1, \ell, \ldots, \ell+b-1\} = \{a-b+1, a-b+2, \ldots, a+1 \} \] 
from each side to reduce to
$$\left\{
\begin{array}{l}
 a-b+3, \ldots , 2a-b-c-d+1, \\
1, \ldots, d, 1 \ldots, c,\\
c+d+1  \end{array}
\right\}
= \left\{\begin{array}{l} 
  1,\ldots, b,\\
  1\ldots, a-b
  \end{array}
\right\}.
$$
Since $c+d+1 \ge c+2 \ge d+2$, for the multiset on the left to equal a union of two intervals 
each containing $1$, we must have $c = a-b+2$ and $d = a-b$.
Then we have an equality of multisets if and only if $c+d+1=b$.
We obtain the case with $c \ge d$ in (k), namely
$$(3\ell-3, 2\ell-1) \eqv{\ell}{3\ell-4}(\ell+1, 1^{\ell-2}).$$

In the remaining case $a-c-d \le 0$, and so $c+d \ge a \ge b$. 
We cancel the elements  $\ell+1, \ldots, \ell+a$ from each side in~\eqref{eqn:hook2rowmultisets} 
to see that $\lambda \eqv{\ell}{\ell+a-c-1} \mu$ if and only if
$$
\left\{
\begin{array}{l}
\ell,\ldots, \ell+b-1,\\
1, \ldots, d, 1 \ldots, c,\\
c+d+1 , a-b+1 \end{array}
\right\}
= \left\{\begin{array}{l} 
  \ell+a-c-d, \ldots, \ell,\\
  1,\ldots, b,\\
  1\ldots, a+1
  \end{array}
\right\}.
$$
Since $b \ge 2$, the element $\ell+1$ lies in the left hand side.
Hence the greatest element of the right hand side is $a+1$ rather than $\ell$. 
But $a+1 \le c+d+1$ which appears on the left; 
hence $a=c+d$, and on the right $\{\ell+a-c-d, \ldots, \ell\} = \{\ell\}$. 
After cancelling $\{\ell, c+d+1\} = \{\ell, a+1\}$
from each side, the multiset equation becomes
$$
\left\{
\begin{array}{l}
\ell+1,\ldots, \ell+b-1,\\
1, \ldots, d, 1 \ldots, c ,\\ c+d-b+1\end{array}
\right\}
= \left\{\begin{array}{l} 
  1,\ldots, b,\\
  1\ldots, c+d
  \end{array}
\right\}.
$$
Since $b \ge 2$, and $c+d$ is in the right-hand side, if equality holds then
the greatest element on the left-hand side
is not $c+d-b+1$. Hence it is
$\ell+b-1=c+d$ and
$$
\left\{
\begin{array}{l}
c+d-b+2,\ldots, c+d,\\
1, \ldots, d, 1 \ldots, c, \\ c+d-b+1 \end{array}
\right\}
= \left\{\begin{array}{l} 
  1,\ldots, b,\\
  1\ldots, c+d
  \end{array}
\right\}.
$$
Either $c+d-b+1=c+1$ and $d=b$, or $c+d-b+1=d+1$ and $c=b$.
Using that $a= c+d$ we have $c = a-b$, $d = b$ or $c = b$, $d=a-b$, respectively.
The corresponding plethystic equivalences are
$(a,b) \eqv{a-b+1}{a}(a-b+1, 1^b)$
and
$(a,b) \eqv{a-b+1}{2(a-b)}(b+1, 1^{a-b})$,
respectively, as in (j).
\end{proof}

We remark that the
Haskell \cite{Haskell98} software \texttt{HookContent.hs} available 
from the second author's website\footnote{See \url{www.ma.rhul.ac.uk/~uvah099/}}
was used to discover many of the  equivalences appearing in Table 1.
It has also been used to verify the more fiddly part of the authors' proof, 
by showing that every plethystic equivalence between two partitions of the types
above, each of size at most $30$, appears in our classification. 
Finally we observe that 
it follows from Proposition~\ref{prop:eqvCondsMoreover} and elementary number-theoretic arguments
that the only plethystic equivalences in Table~1 involving distinct partitions
that lift to isomorphisms of $\GL_2(\C)$-representations are the infinite families
\[ \bigl(\mfrac{d(d+1)}{2}-1,\mfrac{d(d-1)}{2}\bigr)\eqv{d}{2(d-1)} \bigl( \mfrac{d(d-1)}{2}+1, 1^{d-1} \bigr); \]
for $d > 2$ from the second case in (j),
and 
$\bigl( b(b-1),b(b-1) \bigr) \eqv{b+1}{b^2-1} (2^b)$ for $b>2$
from~(o). 

\section{Equal degree equivalences}
\label{sec:equalDegree}

Let $\lambda$ be a partition. 
By Theorem~\ref{thm:complements} we have $\lambda \eqv{\ell}{\ell} \lambda^{\circ (\ell+1)}$
for any $\ell \ge \ell(\lambda)$, where $\lambda^{\circ (\ell+1)}$ denotes the complement
of $\lambda$ in the $(\ell +1) \times a(\lambda)$ box. 
We say that a plethystic equivalence $\lambda \eqv{\ell}{\ell} \mu$, where
$\ell \ge \ell(\mu)$, is \emph{exceptional} 
if $\lambda \not=\mu$ and $\lambda^{\circ (\ell+1)} \not= \mu$.
Thus Theorem~\ref{thm:equalDegree} asserts that there are exceptional equivalences if and only
if $\ell \ge 5$.

\subsection{Ruling out exceptional equivalences}
To prove Theorem~\ref{thm:equalDegree}(a) we use Theorem~\ref{thm:eqvConds}(i), that
$\lambda\eqv{\ell}{\ell}\mu$ if and only if $\Delta_\ell(\lambda) = \Delta_\ell(\mu)$.
Recall that the differences $\delta(\lambda)_j = \lambda_j - \lambda_{j+1} + 1$
and the multiset $\Delta_\ell(\lambda) = \{ \delta(\lambda)_j + \cdots + \delta(\lambda)_{k-1} : 1 \le j < k \le \ell+1\}$
were defined in Definition~\ref{defn:differences}. 
From the definition of $\lambda^{\circ r}$ before Theorem~\ref{thm:complements},
we have 
\[ \delta(\lambda^{\circ r})_j = \bigl( a(\lambda) - \lambda_{r+1-j} \bigr) -
\bigl( a(\lambda) - \lambda_{r-j} \bigr) +1 = \lambda_{r-j} - \lambda_{r+1-j}+1
= \delta(\lambda)_{r-j} \]
and so the difference sequence $\bigl( \delta(\lambda^{\circ (\ell+1)})_1, 
\ldots, \delta(\lambda^{\circ(\ell+1)})_\ell \bigr)$ for $\lambda^{\circ (\ell+1)}$ is 
the reverse of the difference sequence $\bigl( \delta(\lambda)_1, \ldots, \delta(\lambda)_{\ell} 
\bigr)$ for $\lambda$. Thus, as expected, the multisets $\Delta_\ell(\lambda)$
and $\Delta_{\ell}(\lambda^{\circ (\ell+1)})$ agree.
For the small $\ell$ cases, it is surprisingly useful that
 this multiset determines the
minimum weight~$\b(\lambda)$ defined in Definition~\ref{defn:b}.
To prove this we use the following statistic: let
\[ d(\lambda) = \sum_{j=1}^\ell \frac{j(\ell+1-j)}{2} \delta(\lambda)_j - \frac{1}{2}\binom{\ell+2}{3} .\]

%
%
%
%

\begin{lemma}\label{lemma:dStat}
Let $\lambda$ be a partition of $n$ such that $\ell \ge \ell(\lambda)$. 
Then $-\frac{\ell n}{2} + \b(\lambda) = -d(\lambda)$.
\end{lemma}

\begin{proof}
We have $\lambda_i = \delta_i(\lambda) + \cdots + \delta_\ell(\lambda) - (\ell - i + 1)$
for $1 \le \lambda \le \ell$. 
Hence the coefficient of $\delta(\lambda)_j$ in $\sum_{i=1}^\ell \lambda_i$
is $j$ and we have
\[ n = \sum_{j=1}^\ell j \delta(\lambda)_j - \sum_{j=1}^\ell (\ell - j+1) = 
\sum_{j=1}^\ell j \delta(\lambda)_j - \binom{\ell+1}{2}. \]
Similarly, since $\b(\lambda) = \sum_{i=1}^{\ell} (i-1)\lambda_i$, the coefficient
of $\delta(\lambda)_j$ in $\b(\lambda)$ is $\sum_{i=1}^j (i-1) = \binom{j}{2}$ and, using
$\sum_{i=1}^\ell (i-1)(\ell - i +1) = \sum_{k=1}^{\ell-1} k(\ell-k) = 
\mfrac{1}{2} \ell^2 (\ell-1) - \mfrac{1}{6} \ell(\ell-1)(2\ell-1) =
\binom{\ell+1}{3}$
we have
\[
\b(\lambda)  = \sum_{j=1}^\ell \binom{j}{2} \delta(\lambda)_j - \binom{\ell+1}{3} .
\]
The result now follows from the two displayed equations.
\end{proof}



\begin{proof}[Proof of Theorem~\ref{thm:equalDegree}(a)]
By Theorem~\ref{thm:eqvConds}(i) 
if $\lambda \eqv{\ell}{\ell} \mu$ then 
$\Delta_\ell(\lambda) = \Delta_{\ell}(\mu)$.
By the final part of this theorem, 
$-\mfrac{\ell|\lambda|}{2} + \b(\lambda) = -\mfrac{\ell |\mu|}{2} + \b(\mu)$.
Hence by Lemma~\ref{lemma:dStat},
we also have $d(\lambda) = d(\mu)$.

For $1 \le j \le \ell$, let
 $\delta_j = \delta(\lambda)_j$ and let $\epsilon_j = \delta(\mu)_j$.
 Observe that the greatest elements of $\Delta_\ell(\lambda)$
and $\Delta_{\ell}(\mu)$ are $\delta_1 + \cdots + \delta_\ell = \lambda_1 + \ell$ and
$\epsilon_1 + \cdots + \epsilon_\ell = \mu_1 + \ell$, respectively.
Hence, as also follows from Proposition~\ref{prop:aRelation}, we have
\begin{equation}\label{eq:aRelationD}
\delta_1 + \cdots + \delta_\ell = \epsilon_1 + \cdots + \epsilon_\ell. \end{equation}
If $\ell = 1$ then $\lambda_1 = \delta_1 = \epsilon_1 = \mu_1$ and $\lambda = \mu$ as required.
If $\ell \ge 2$ then
the least two elements of $\Delta_\ell(\lambda)$ are $\delta_c$ and $\delta_{c'}$
for some distinct $c$ and $c'$, and similarly for $\Delta_\ell(\mu)$.
Hence the multisubsets of the least two elements in $\Delta_\ell(\lambda)$
and $\Delta_\ell(\mu)$ agree.

Suppose that $\ell = 2$. We have just seen that 
$\{\delta_1, \delta_2 \} = \{\epsilon_1, \epsilon_2\}$. This is the case
if and only if  $\lambda = \mu$
or $\lambda = \mu^{\circ 3}$.

Suppose that $\ell = 3$. The multisets $\{\delta_1,\delta_2,\delta_3\}$
and $\{\epsilon_1,\epsilon_2,\epsilon_3\}$ have the same least two elements,
and, by~\eqref{eq:aRelationD}, the same sum. Hence they are equal.
By Lemma~\ref{lemma:dStat}, 
$3\delta_1 + 4\delta_2 + 3\delta_3 = 3\epsilon_1 + 4\epsilon_2
+ 3\epsilon_3$. Hence, again using~\eqref{eq:aRelationD}, we
have $\delta_2 = \epsilon_2$. Now either $\delta_1 = \epsilon_1$,
and so $(\delta_1,\delta_2,\delta_3) = (\epsilon_1,\epsilon_2,\epsilon_3)$
and $\lambda = \mu$, or $\delta_1 = \epsilon_3$ and so
$(\delta_1,\delta_2,\delta_3) = (\epsilon_3,\epsilon_2,\epsilon_1)$ and
$\lambda = \mu^{\circ 4}$.

Suppose that $\ell = 4$. By replacing $\lambda$ and $\mu$ with their complements
if necessary, we may assume that $\delta_1 \le \delta_4$ and $\epsilon_1 \le \epsilon_4$.
By Lemma~\ref{lemma:dStat},
$2\delta_1 + 3\delta_2 + 3\delta_3 + 2\delta_4 = 
  2\epsilon_1 + 3\epsilon_2 + 3\epsilon_3 + 2\epsilon_4$. 
Hence by~\eqref{eq:aRelationD} we have
\begin{equation}\label{eq:l4}
\delta_1 + \delta_4 = \epsilon_1 + \epsilon_4\text{ and } 
 \delta_2 + \delta_3 = \epsilon_2 + \epsilon_3 . \end{equation}
Since $\delta_1 \le \delta_4$,
after $\delta_1 + \delta_2 + \delta_3 + \delta_4$, the second greatest
element of $\Delta_4(\lambda)$ is $\delta_2 + \delta_3 + \delta_4$.
Similarly the second greatest element of $\Delta_4(\mu)$ is $\epsilon_2 + \epsilon_3 + \epsilon_4$. Therefore $\delta_1 = \epsilon_1$ and so by~\eqref{eq:l4}, $\delta_4 = \epsilon_4$.
Cancelling the equal elements $\delta_2 + \delta_3 = \epsilon_2 + \epsilon_3$,
$\delta_1 + \delta_2 +\delta_3 = \epsilon_1 + \epsilon_2 + \epsilon_3$,
$\delta_2 + \delta_3 + \delta_4 = \epsilon_2 + \epsilon_3 + \epsilon_4$ and
$\delta_1 + \delta_2 + \delta_3 + \delta_4 = \epsilon_1 + \epsilon_2 + \epsilon_3 + \epsilon_4$
 from the multisets $\Delta_4(\lambda)$ and $\Delta_4(\mu)$ we obtain
\begin{equation}\label{eq:l4b} \{ \delta_2, \delta_3, \delta_1 + \delta_2, \delta_3 + \delta_4
 \}
= \{ \epsilon_2, \epsilon_3, \delta_1 + \epsilon_2, \epsilon_3 + \delta_4
\}. \end{equation}
The least element on the left is either $\delta_2$ or $\delta_3$, and the least
element on the right is either $\epsilon_2$ or $\epsilon_3$.
If $\delta_2 = \epsilon_2$ then from~\eqref{eq:l4} we get $\delta_3 = \epsilon_3$.
Hence $(\delta_1,\delta_2,\delta_3,\delta_4) = (\epsilon_1,\epsilon_2,\epsilon_3,\epsilon_4)$
and $\lambda = \mu$. A symmetric argument, swapping $2$ and $3$, applies if $\delta_3 = \epsilon_3$.
In the remaining case we may suppose, by swapping $\lambda$ and $\mu$ if necessary,
that $\delta_2 = \epsilon_3$. From~\eqref{eq:l4} we get $\delta_3 = \epsilon_2$. 
Hence $(\delta_1,\delta_2,\delta_3,\delta_4) = (\epsilon_1, \epsilon_3,\epsilon_2, \epsilon_4)$.
From~\eqref{eq:l4b} we now get $\{\delta_1+\delta_2,\delta_3+\delta_4\} = \{\delta_1 + \delta_3,
\delta_2 + \delta_4\}$. Hence either $\delta_2 = \delta_3$ and $\lambda = \mu$ or
$\delta_1 = \delta_4$ and $\lambda = \mu^{\circ 5}$.

\end{proof}

\subsection{Existence of exceptional equivalences}
To prove Theorem~\ref{thm:equalDegree}(ii)
we use the pyramid notation  seen in Example~\ref{ex:pyramid}, applied
to the following partitions.


\begin{definition}\label{defn:exceptionalPartitions}
For $\k$, $\ell \in \N_0$ with $\ell \ge 5$ 
we define 
$\lambdae{\k}{\ell}$ and $\mue{\k}{\ell}$ by
\begin{align*}
\lambdae{\k}{\ell} &= \begin{cases} (8+5\k, 7+4\k, 2+2\k, 2+\k)\hspace*{0.385in} & \text{if $\ell = 5$} \\
(6+\k, 4+\k, 3+\k, 3+\k, 3) & \text{if $\ell = 6$} \\
(6+\k, 4+\k, 4+\k, 3+\k, 3+\k, 3) & \text{if $\ell = 7$} \\
(5+\k, 4+\k, (3+\k)^4, 3, 1^{\ell-8}) & \text{if $\ell \ge 8$} \end{cases}
\intertext{and}
\mue{\k}{\ell} &= \begin{cases} (8+5\k, 6+4\k, 3+2\k, \k) & \text{if $\ell = 5$} \\
(6+\k, 3+\k, 3+\k, 3, 1) & \text{if $\ell = 6$} \\
(6+\k, 3+\k, 3+\k, 3, 1, 1) & \text{if $\ell = 7$} \\
(5+\k, (3+\k)^{\ell -8}, 2+\k, 2+\k, 2, 1) & \text{if $\ell \ge 8$.} \end{cases}
\end{align*}
%
\end{definition}

The partition $\mue{0}{6}$ is the lexicographically least partition in 
an exceptional equivalence when $\ell =  6$. The other partitions were discovered 
by a computer search using the software already mentioned.
The special case $\ell = 5$ and $\k=0$ 
of the following proposition was seen
in Example~\ref{ex:pyramid}.

\begin{proposition}\label{prop:exceptionalEqv}
For all $\k \in \N_0$ and $\ell \ge 5$ there is an exceptional  equivalence $\lambdae{\k}{\ell} 
\eqv{\ell}{\ell} \mue{\k}{\ell}$. 
\end{proposition}

\enlargethispage{12pt}

\begin{proof}
It is clear from Definition~\ref{defn:exceptionalPartitions}
that $\lambdae{\k}{\ell} \not= \mue{\k}{\ell}$ for
any $\ell$ and $\k$. Moreover,
since the second part in each $\mue{\k}{\ell}$ 
is strictly smaller than $a(\mue{\k}{\ell})$, 
each partition $\mue{\k}{\ell}^{\circ (\ell+1)}$ has precisely $\ell$ parts.
Since each partition $\lambdae{\k}{\ell}$ has precisely $\ell-1$ parts,
it follows that $\lambdae{\k}{\ell} \not= \mue{\k}{\ell}^{\circ (\ell+1)}$
for any $k$ and $\ell$.
To proceed further, it is most convenient to work with 
the complementary partitions $\etae{\k}{\ell} = \mue{\k}{\ell}^{\circ (\ell+1)}$.
By Theorem~\ref{thm:complements} and Theorem~\ref{thm:eqvConds}(i), it suffices
to prove that $\Delta_\ell(\lambdae{\k}{\ell}) = \Delta_\ell(\etae{\k}{\ell})$ for
all $\ell$ and $\k$.

For small $\ell$ this is a routine verification using the pyramid notation seen
in Example~\ref{ex:pyramid}.
To illustrate
the method we take  $\ell = 8$. It will be useful to say that a pyramid entry
\emph{involves $\k$} if it of the form $c + \k$ for some $c \in \N$.
The difference sequences for $\lambdae{\k}{\ell}$ and $\etae{\k}{\ell}$ are
\begin{align*}
&(2,2,1,1,1,1+\k,3,1^{\ell-9},2,1) \\
&(1,1,1,2,2,1+\k,1,2,1^{\ell-9},3),
\end{align*}
respectively.
When $\ell = 8$,
the corresponding pyramids are
\[ \hskip-0.375in \scalebox{0.775}{$\begin{matrix} 2 & 2 & 1 &  1 & 1 & 1+\k & 4 & 1 \\
4 & 3 & 2 & 2 & 2+\k & 5 + \k & 5 \\
5 & 4 & 3 & 3+\k & 6+\k & 6+\k \\
6 & 5 & 4+\k & 7+\k & 7+\k \\
7 & 6+\k & 8+\k & 8+\k \\
8+\k & 10 + \k & 9 + \k \\
12 + \k & 11 + \k \\
13 + \k 
\end{matrix}$,} \quad\   \scalebox{0.775}{$\begin{matrix} 
1 & 1 & 1 & 2 & 2 & 1 + \k & 1 & 4 \\
2 & 2 & 3 & 4 & 3 +\k & 2 +\k & 5 \\
3 & 4 & 5 & 5+\k & 4+\k & 6 +\k \\
5 & 6 & 6+\k & 6+\k & 8+\k \\
7 & 7+\k & 7+\k & 10+\k \\
8+\k & 8+\k & 11+\k \\
9+\k & 12 +\k \\
13 +\k
\end{matrix}$}. \]
Note that the entries involving $m$ lie in the same positions. This helps one to see that
in either case the multiset of pyramid entries is
\[ \{1^4,2^4,3^2,4^3,5^3,6,7\} \cup  \bigl( \{1,2,3,4,5,6^3,7^2,8^3,9,10,11,12,13\} + \k\bigr). \]
Similarly one can check that $\Delta_\ell(\lambdae{k}{\ell}) = \Delta_\ell(\etae{k}{\ell})$
for all $k \in \N_0$ and all~$\ell$ such that $5 \le \ell \le 18$. This can be done programmatically 
using the {\sc Mathematica} \cite{Mathematica}
notebook \texttt{ExceptionalEquivalences.nb} available from the second author's website.

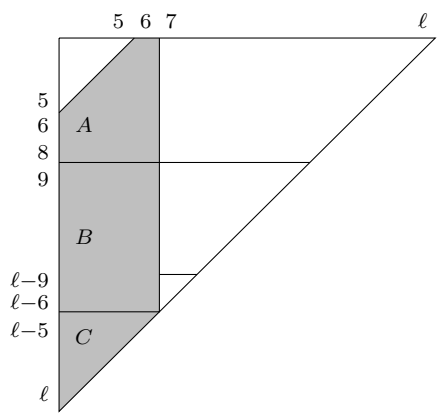
\begin{figure}[t]
\begin{center}
\begin{tikzpicture}[x=0.33cm,y=0.33cm,yscale=-1]

\fill[color = lightgray] (3,0)--(4,0)--(4,11)--(0,15)--(0,3)--(3,0);

\node[above left] at (3,0) {$\scriptstyle 5$};
\node[above left] at (0,3.2) {$\scriptstyle 5$};
\node[below left] at (0,2.8) {$\scriptstyle 6$};

\node[below left] at (0,3.9) {$\scriptstyle 8$};
\node[below left] at (0,5) {$\scriptstyle 9$};
\node[above left] at (0,11.4) {$\scriptstyle \ell-6$};
\draw(4,9.5)--(5.5,9.5);
\node[above left] at (0,10.5) {$\scriptstyle \ell-9$};

\node[below left] at (0,11) {$\scriptstyle \ell-5$};
\node[above left] at (0,15) {$\scriptstyle \ell$};

\node[above left] at (4.1,0) {$\scriptstyle 6$};
\node[above left] at (5.1,0) {$\scriptstyle 7$};

\node[above left] at (15.1,0) {$\scriptstyle \ell$};

\node at (1,3.5) {$\scriptstyle A$}; 
\node at (1,8) {$\scriptstyle B$}; 
\node at (1,12) {$\scriptstyle C$};
\draw(0,0)--(15,0)--(0,15)--(0,0);
\draw(3,0)--(0,3);
\draw(4,0)--(4,11)--(0,11);
\draw(0,5)--(10,5);

\end{tikzpicture}
\end{center}
\caption{Partition of the pyramids
for $\lambdae{\k}{\ell}$ and $\etae{\k}{\ell}$. The important
 row and column numbers are indicated; an entry involves $\k$
if and only if it is in the shaded region.}
\end{figure}

For the generic case when $\ell \ge 18$ 
we partition the pyramids $P$ and $Q$ for $\lambdae{\k}{\ell }$
and $\etae{\k}{\ell}$ as shown in Figure~6.  
Using the calculation rule from \S\ref{subsec:pyramids} one finds that
the first $8$
rows of the pyramid $P$ for $\lambdae{\k}{\ell}$ are
\newcommand{\mdots}[1]{\hskip-5pt\stackrel{#1}{\ldots}\hskip-12pt}
\[ \setcounter{MaxMatrixCols}{20}\scalebox{0.85}{$\begin{matrix} 
2    & 2    & 1   & 1   & 1   & 1+\k & 3   & 1 & \mdots{\ell-9} & \hskip-5pt 1 & 2 & 1 \\
4    & 3    & 2   & 2   & 2+\k & 4+\k & 4   & 2 & \mdots{\ell-10} &\hskip-5pt 2 & 3 & 3 \\
5    & 4    & 3   & 3+\k & 5+\k & 5+\k & 5   & 3 & \mdots{\ell-11}& \hskip-5pt3 & 4 & 4 \\
6    & 5    & 4+\k & 6+\k & 6+\k & 6+\k & 6   & 4 & \mdots{\ell-12}&\hskip-5pt 4 & 5 & 5 \\
7    & 6+\k  & 7+\k & 7+\k & 7+\k & 7+\k & 7   & 5 & \mdots{\ell-13} &\hskip-5pt 5 & 6 & 6 \\
8+\k  & 9+\k  & 8+\k & 8+\k & 8+\k & 8+\k & 8   & 6 & \mdots{\ell-14} &\hskip-5pt 6 & 7 & 7 \\
11+\k & 10+\k & 9+\k & 9+\k & 9+\k & 9+\k & 9   & 7 & \mdots{\ell-15} &\hskip-5pt 7 & 8 & 8 \\
12+\k & 11+\k & 10+\k& 10+\k& 10+\k& 10+\k& 10  & 8 & \mdots{\ell-16} &\hskip-5pt 8 & 9 & 9 \\
\end{matrix}$}
\]
where in each case the notation $c \stackrel{m}{\ldots} c$ indicates $m$ consecutive entries of $c$.
Similarly, the first $8$ rows of the pyramid $Q$ for $\eta{\k}{\ell}$ are
\[ \scalebox{0.85}{$\begin{matrix} 
1    & 1    & 1   & 2   & 2   & 1+\k & 1   & 2 & 1 & \mdots{\ell-9} &\hskip-5pt 1 & 3 \\
2    & 2    & 3   & 4   & 3+\k & 2+\k & 3   & 3 & 2 & \mdots{\ell-10} &\hskip-5pt 2 & 4 \\
3    & 4    & 5   & 5+\k & 4+\k & 4+\k & 4   & 4 & 3 & \mdots{\ell-11} & \hskip-5pt3 & 5 \\
5    & 6    & 6+\k & 6+\k & 6+\k & 5+\k & 5   & 5 & 4 & \mdots{\ell-12} &\hskip-5pt 4 & 6 \\
7    & 7+\k  & 7+\k & 8+\k & 7+\k & 6+\k & 6   & 6 & 5 & \mdots{\ell-13} &\hskip-5pt 5 & 7 \\
8+\k  & 8+\k  & 9+\k & 9+\k & 8+\k & 7+\k & 7   & 7 & 6 & \mdots{\ell-14} &\hskip-5pt 6 & 8 \\
9+\k  & 10+\k & 10+\k &10+\k& 9+\k & 8+\k & 8   & 8 & 7 & \mdots{\ell-15} &\hskip-5pt 7 & 9 \\
11+\k & 11+\k & 11+\k& 11+\k& 10+\k& 9+\k & 9   & 9 & 8 & \mdots{\ell-16} &\hskip-5pt 8 & \,10. \\
\end{matrix}$}
\]
Observe that if $r \le 5$ then the multisets of entries of $P$ and $Q$ in row $r$ not involving $\k$
are the same. Moreover, it is easily proved by induction
on~$r$ that if~$6 \le r \le \ell-9$ then the entries of $P$ and $Q$ in row $r$ not involving $\k$
are $r+2, r \stackrel{\ell-8-r}{\ldots} r, r+1, r+1$
and $r+1, r+1, r\stackrel{\ell-8-r}{\ldots} r, r+2$, respectively.
As can be seen from Figure~6, the remaining entries in $P$ and $Q$
not involving $\k$ lie in rows $\ell-9$, $\ell-8$, $\ell -7$ and $\ell-6$ 
and columns $7$, $8$, $9$, $10$. They are
\[ \hskip-0.25in \scalebox{0.85}{$\begin{matrix}
\ell-9 & \ell-11 & \ell-11 & \ell-11 & \ell-10 & \ell-10 \\ 
\ell-8 & \ell-10 & \ell-10 & \ell-9 & \ell-9 \\ \hline
\ell-7 & \ell-9 & \ell-8 & \ell-8 \\ \ell-6 & \ell-7 & \ell -7 \\ \boldsymbol\ell\boldsymbol-\boldsymbol4 & \ell-6 \\ 
\boldsymbol\ell\boldsymbol-\boldsymbol3
\end{matrix}$},\qquad \scalebox{0.85}{$\begin{matrix}
\ell-10 & \ell-10 & \ell-11 & \ell-11 & \ell-11 & \ell-9 \\ 
\ell-9 & \ell-9 & \ell-10 & \ell-10 & \ell-8 \\ \hline
\ell-8 & \ell-8 & \ell-9 & \ell-7 \\
\ell-7 & \ell-7 & \ell -6 \\ \ell-6 & \ell - 4 \\ \boldsymbol\ell\boldsymbol-\boldsymbol3 
\end{matrix}$}
\]
where the first two rows shows the known entries from rows $\ell-11$, $\ell-10$
needed to compute the following rows. Three exceptional entries are highlighted.
Again the multisets of entries 
agree row by row. Hence the multisets of entries in $P$ and $Q$
agree on entries not involving $\k$.
We note for later use that, from the pyramids immediately above,
\begin{equation}\label{eq:col7} 
\P{r}{7} = r+2 \text{ and } \QQ{r}{7} = r+1 \text{ if $8 \le r \le \ell - 8$} \end{equation}
and $\P{\ell-7}{7}  = \ell-4$, $\P{\ell-6}{7} = \ell - 3$,
$\QQ{\ell-7}{7} = \ell -6$ and $\QQ{\ell-6}{7} = \ell - 3$.

We now consider entries involving $\k$.
For $1 \le r \le \ell$, let $\mathcal{P}^{(r)}+\k$ and $\mathcal{Q}^{(r)}+\k$ 
be the multisets of entries in row $r$ of $P$ and $Q$ 
involving $\k$.
Comparing the $33$ entries involving~$\k$ in rows $r$ for $1\le r \le 8$ (region $A$ in Figure~6)
we find that 
\begin{equation}
\label{eq:regionA}
\bigcup_{r=1}^8 \mathcal{P}^{(r)} \bigl/ \, \bigcup_{r=1}^8 \mathcal{Q}^{(r)} = \{4,2\} / \{3,3\}  + 8.
\end{equation}
If $8 \le r \le \ell-6$ then the entries involving $k$
in row $r$ are precisely those in region $B$, lying in the first six columns of the pyramids.
Let $\bP^{(r)}$ and $\bQ^{(r)}$ be the $6$-tuples
defined by \smash{$\bP^{(r)}_j = \P{r}{j} - \k$} and 
\smash{$\bQ^{(r)}_j = \QQ{r}{j} -\k$}, respectively.
An induction on $r$,
using~\eqref{eq:col7} to find the entries in column $6$,
shows that if $8 \le r \le \ell-7$ then
\begin{align*}
\bP^{(r)} &= (4+r,3+r,2+r,2+r,2+r,2+r) \\
\bQ^{(r)} &= (3+r,3+r,3+r,3+r,2+r,1+r). 
\end{align*}
(We stop at $\ell-7$ because of the exceptional entry $\P{\ell-7}{7}$.)
The corresponding difference multiset is $\{4,2,2,2\}/ \{3,3,3,1\} + r$. 
We now claim that
\begin{equation}\label{eq:regionB}
\bigcup_{r=1}^{s} \mathcal{P}^{(r)} \bigl/\, \bigcup_{r=1}^{s} \mathcal{Q}^{(r)} 
=  \{4,2\} / \{3,3\} + s \end{equation} 
for $8 \le s \le \ell-7$. Indeed, when $s =8$ this follows from~\eqref{eq:regionA}, and
the inductive step is immediate from
$\bigl( \{4,2\} / \{3,3\} + s \bigr) \cup \bigl( \{4,2,2,2\} / \{3,3,3,1\} + s+1\bigr) =
\{ 4, 2 \} / \{3,3\} + (s+1)$. 
Using the exceptional
entries \smash{$\P{\ell-7}{7} = \ell-4$} and \smash{$\QQ{\ell-7}{7} = \ell-3$}  seen after~\eqref{eq:col7},
we have $\bP^{(r)} = (4+r,3+r,2+r,2+r,2+r,3+r)$
and $\bQ^{(r)} = (3+r,3+r,3+r,3+r,2+r,1+r)$ when $r = \ell -6$. 
The corresponding difference multiset is $\{4,2,2\} / \{3,3,1\} + \ell-6$.
Therefore, by~\eqref{eq:regionB},
\begin{align*} 
\bigcup_{r=1}^{\ell-6} \mathcal{P}^{(r)} \bigl/\, \bigcup_{r=1}^{\ell-6} \mathcal{Q}^{(r)} 
&= \bigl( \{4,2\} / \{3,3\} + \ell-7 \bigr) \cup 
\bigl( \{4,2,2\} / \{3,3,1\} + \ell - 6 \bigr)\\
&= \{5,4,3,3,2\} / \{4,4,3,3,2\} + \ell - 7 \\
&= \{5 \} / \{4\} + \ell-7 \\
&= \{\ell-2\} / \{\ell-3\}.
\end{align*}
The final six rows of the pyramids (region~$C$) are, with the constant factor $+\k$ removed,
\[ \scalebox{0.85}{$\begin{matrix} \ell-1 & \ell -2 & \ell - 3 & \ell -3 & \ell -2 & \ell -2 \\
\ell & \ell-1 & \ell -2 & \ell -1 & \ell -1 \\
\ell + 1& \ell & \ell & \ell \\
\ell + 2 & \ell +2 & \ell+1 \\
\ell + 4& \ell + 3 \\
\ell + 5 \end{matrix}$},
\qquad
\scalebox{0.85}{$\begin{matrix} \ell -2 & \ell -2 & \ell -2 & \ell -2 & \ell - 3 & \ell -2 \\
\ell - 1 & \ell -1 & \ell -1 & \ell -1 & \ell \\
\ell & \ell & \ell & \ell+2 \\
\ell+1 & \ell+1 & \ell+3 \\
\ell+2 & \ell+4 \\
\ell + 5 \end{matrix}$}\]
respectively. The corresponding difference multiset is
$\{\ell-3\} / \{\ell - 2\}$.
Hence the multisets of entries involving $\k$ in $P$ and $Q$ are the same.
\end{proof}

We end by remarking that, by Proposition~\ref{prop:eqvCondsMoreover},
an exceptional equivalence $\lambda \eqv{\ell}{\ell} \mu$ 
lifts to an isomorphism $\nabla^\lambda \SSym^\ell \E \cong \nabla^\mu \SSym^\mu \E$
of representations of $\GL(E)$ if and only if $|\lambda| = |\mu|$.
A computer search shows that, by
size of partitions, the smallest such example is
\[ (5,4,3^5, 1^6) \eqv{14}{14} (5,3^6,2^3,1) \]
between two partitions of $30$.
As a curiosity, we note that $146$ of the $493$ exceptional equivalences $\lambda\eqv{\ell}{\ell}{\mu}$
between partitions $\lambda$ and $\mu$
of size at most~$35$ have $\ell=8$. Next most frequent are $\ell = 11$ with $99$ equivalences and $\ell = 14$ with $56$ equivalences.




\section{Solitary partitions}
\label{sec:solitary}

\begin{definition}\label{defn:solitary}
A partition $\lambda$ is \emph{solitary} if 
whenever $\lambda\eqv{\ell}{m} \mu$
with $\ell \ge \ell(\lambda)$ and $m \ge \ell(\mu)$, we have
$\ell = m$ and either $\mu = \lambda$ or $\mu = \lambda^{\circ (\ell+1)}$.
\end{definition}

By Theorem~\ref{thm:complements}, the equivalences in Definition~\ref{defn:solitary}
exist for any partition. The solitary partitions
are therefore those with the fewest possible prime plethystic equivalences.
Using Theorem~\ref{thm:complements}, Theorem~\ref{thm:solitary} reduces to the following proposition.
Recall that $\delta(k) = (k,k-1,\ldots, 1)$.

\begin{proposition}
For each $k \in \N$, the partition $\delta(k)$ is solitary.
\end{proposition}

\begin{proof}
Suppose that $\delta(k) \eqv{\ell}{m} \mu$ where $\ell \ge k$ and $m \ge \ell(\mu)$
and that $\mu \not= \delta(k)$. 
Using
the difference multiset notation from~\S\ref{subsec:differenceMultiset}, Theorem~\ref{thm:eqvConds}(h)
implies that
\begin{equation}
\label{eq:solitary} \frac{C\bigl( \mu \bigr) + m + 1} {C\bigl( \delta(k) \bigr) + \ell + 1}
= \frac{H(\mu)}{H\bigl( \delta(k) \bigr)}. \end{equation}
Let $\Hc$ be the number of boxes $(i,j) \in [\mu]$ such that $h_{(i,j)}(\mu) = 2$;
we say that such boxes are \emph{$2$-hooks}.
Since all the hook lengths in $\delta(k)$ are odd, the multiplicity of $2$
in the right-hand side is $\Hc$. 
By Lemma~\ref{lemma:removable}, $\mu$ has precisely~$k$ removable boxes. 
Since $\mu \not= \delta(k)$, $\mu$ has at least one $2$-hook, and so $\Hc \ge 1$.
For any partition $\nu$ and $n$ such that $n \ge \ell(\nu)$, we have
$2 \in C(\nu) + n+1$ if and only if $n = \ell(\nu)$; in this
case the multiplicity is $1$. Therefore $\Hc \le 1$ and we conclude that $\mu$
has a unique $2$-hook. Moreover, $m = \ell(\mu)$ and $\ell > k$.

To identify $\ell$ we use Proposition~\ref{prop:aRelation} to get
$a\bigl( \delta(k) \bigr) + l = a(\mu) + m$. (Or one may follow the proof of this proposition
and instead compare greatest elements in~\eqref{eq:solitary}.) Hence $l = a(\mu) + m - k$.
Since $\mu$ has a unique $2$-hook and precisely $k$ removable boxes, 
it is obtained from $\delta(k)$ by inserting
either (i) $d$ new columns or (ii) $d$ new rows of a fixed
length $c \le k$. We consider these cases separately below. Observe that in either case the greatest hook length in either
$\mu$ or $\delta(k)$ is $(k+d-1) + (k-1) + 1 = 2k+d-1$, coming uniquely from the box $(1,1)$ of $\mu$.
Hence $2k+d-1$ has multiplicity $1$ in the right-hand side of~\eqref{eq:solitary}.

\begin{itemize}
\item[(i)] In this case $a(\mu) = k + d$ and $\ell(\mu) = k$. Hence $m = k$ and $\ell = k + d$.
In $C(\mu) + m+1$, the greatest element is $(k+d-1) + (m+1) = 2k+d$ and, since $\mu_1 > \mu_2$,
the next greatest element is $(k+d-1-1) + (m+1) = 2k+d-1$, also with multiplicity 
$1$. In $C\bigl( \delta(k) \bigr) + \ell+1$, the second greatest element is 
$(k-1-1) + (\ell+1) = 2k+d-1$, again with multiplicity $1$. Therefore
$2k+d-1$ has multiplicity $0$ in the left-hand side of~\eqref{eq:solitary}, a contradiction.

\item[(ii)] In this case $a(\mu) = k$ and $\ell(\mu) = k + d$. Hence $m = k+d$ and $\ell = m$.
A similar argument considering the multiplicity of $2k+d-1$ in
the left-hand side of~\eqref{eq:solitary} shows that $2k+d-1$ must appear with multiplicity~$2$
in $C(\mu) + m+1$. Hence $\mu_1 = \mu_2$ and $c = k$. This shows that $\mu$ is the complement
of $\delta(k)$ in the box with $k+d+1$ rows; that is $\mu = \delta(k)^{\circ (\ell + 1)}$.
\end{itemize}

This completes the proof.
\end{proof}

\def\cprime{$'$} \def\Dbar{\leavevmode\lower.6ex\hbox to 0pt{\hskip-.23ex
  \accent"16\hss}D} \def\cprime{$'$}
\providecommand{\bysame}{\leavevmode\hbox to3em{\hrulefill}\thinspace}
\providecommand{\MR}{\relax\ifhmode\unskip\space\fi MR }
\providecommand{\MRhref}[2]{%
  \href{http://www.ams.org/mathscinet-getitem?mr=#1}{#2}
}
\renewcommand{\MR}[1]{\relax}
\providecommand{\href}[2]{#2}


\end{document}